\documentclass[11pt]{amsart}
\usepackage{amsmath}
\usepackage{amssymb, comment}
\usepackage{color}
\usepackage{soul}

\title[]{Sasaki-Einstein metrics and K-stability}
\author[T.C. Collins]{Tristan C. Collins}

\thanks{T.C.C is supported in part by NSF grant DMS-1506652.
  G.Sz. is supported in part by NSF grants DMS-1306298 and DMS-1350696.}
\author[G. Sz\'ekelyhidi]{G\'abor Sz\'ekelyhidi}

\address{Department of Mathematics, Harvard University, 1 Oxford Street, Cambridge, MA 02138}
\email{tcollins@math.harvard.edu}

\address{Department of Mathematics, University of Notre Dame, Notre
  Dame, IN 46556}  
\email{gszekely@nd.edu}

\theoremstyle{plain}
\newtheorem{thm}{Theorem}[section]
\newtheorem{prop}[thm]{Proposition}
\newtheorem{defn}[thm]{Definition}
\newtheorem{lem}[thm]{Lemma}
\newtheorem{cor}[thm]{Corollary}
\newtheorem{conj}[thm]{Conjecture}
\newtheorem{que}[thm]{Question}

\theoremstyle{definition}

\newtheorem{rk}[thm]{Remark}

\numberwithin{equation}{section}

\newcommand{\del}{\partial}

\newcommand{\dbar}{\overline{\del}}
\newcommand{\ddb}{\sqrt{-1}\del\dbar}

\newcommand{\Spec}{\text{Spec }}

\newcommand{\Ric}{\mathrm{Ric}}
\renewcommand{\d}{\partial}
\newcommand{\wt}[1]{\widetilde{#1}}

\renewcommand{\leq}{\leqslant}
\renewcommand{\geq}{\geqslant}

\renewcommand{\epsilon}{\varepsilon}
\renewcommand{\phi}{\varphi}

\begin{document}
\begin{abstract}
We show that a polarized affine variety admits a Ricci flat K\"ahler
cone metric if and only if it is K-stable. 
This generalizes Chen-Donaldson-Sun's solution of the Yau-Tian-Donaldson
conjecture to K\"ahler cones, or equivalently, Sasakian
manifolds. As an application we show that the five-sphere admits
infinitely many families of Sasaki-Einstein metrics. 
\end{abstract}
\maketitle

\section{Introduction}
The existence of K\"ahler-Einstein metrics is a fundamental problem in
K\"ahler geometry. If $M$ is a compact complex manifold with
$c_1(M)=0$ or $c_1(M) < 0$, then the work of Yau~\cite{Y}
shows that $M$ admits K\"ahler-Einstein metrics with zero or negative
Ricci curvature. The case when $c_1(M) > 0$ is more subtle, and the
Yau-Tian-Donaldson conjecture~\cite{Yauprob,TianKEpos,DonTor}, proved by
Chen-Donaldson-Sun~\cite{CDS1,CDS2, CDS3}, relates the existence of
a K\"ahler-Einstein metric on $M$ to the K-stability of $M$, which is
a certain algebro-geometric condition. Our goal in the present paper
is to generalize this result to the setting of K\"ahler cones, 
giving a criterion for the existence of a Ricci flat K\"ahler cone metric,
or equivalently, a Sasaki-Einstein metric on the link. The question of
existence of such metrics has received increasing attention in the
physics community through their connection to the AdS/CFT
correspondence (see \cite{Mal,KlWi}), and we anticipate further
developments along these lines (see e.g. \cite{CXY}).  

Quite generally, if $X$ is an affine variety with an isolated singular
point, one can ask whether $X$ admits a Ricci flat K\"ahler cone
metric. We will address this question under the extra assumption that
we fix the vector field on $X$ that gives the homothetic scaling on
the cone. More precisely, suppose that $X \subset \mathbf{C}^N$ is an affine
variety, with an isolated singular point at the origin, invariant
under the action of a torus $\mathbf{T}\subset U(N)$, which for
simplicity we assume to be diagonal. We call $\xi\in
\mathfrak{t}$ a \emph{polarization} of $X$ if it acts with positive weights
on the coordinate functions, i.e. the corresponding holomorphic vector
field $\xi$ satisfies $L_\xi(z_i) = ia_i z_i$ with $a_i > 0$.
We then seek a K\"ahler Ricci flat metric
$\omega$ on $X$ such that $L_{-J\xi} \omega = 2\omega$. We say that such a
metric $\omega$ is a Ricci flat K\"ahler cone metric on the pair
$(X,\xi)$. Such a Ricci flat K\"ahler cone metric can only exist on
$X$ if the pair $(X,\xi)$ is a normalized Fano cone singularity, in
the terminology of Definition~\ref{defn:Fanocone} below.

Our main result is the following.
\begin{thm}\label{thm:mainthm} Let $(X, \xi)$ be a normalized Fano
  cone singularity. Then $(X, \xi)$ admits a Ricci flat K\"ahler cone
  metric if and only if it is K-stable. 
\end{thm}

We will give a precise definition of K-stability in this setting
below. For now let us say that if $(X,\xi)$ does not admit a Ricci
flat K\"ahler cone metric, then there exists an embedding $X \hookrightarrow
\mathbf{C}^{N'}$, a corresponding embedding of the torus $\mathbf{T}\subset
U(N')$, and a one-parameter subgroup $\lambda: \mathbf{C}^* \to
GL(N')^\mathbf{T}$ 
generated by a vector field $w$ with the following properties:
\begin{enumerate}
  \item The limit $Y = \lim\limits_{t\to 0} \lambda(t)\cdot X$ is
    normal.
  \item The Futaki invariant $\mathrm{Fut}(Y,\xi, w) \leq 0$, and
    $Y\not\cong X$ if equality holds.
  \end{enumerate}

  While in principle $\mathbf{T}$ just needs to be a torus of automorphisms of
$X$ for which $\xi\in \mathfrak{t}$, in practice it is useful to
choose a maximal such torus. In fact as in 
\cite{Datar-Sz} we can also obtain an equivariant version of the theorem for
the action of any compact group on $X$, but to simplify the exposition
we will mostly focus on the case of a torus. 

Usually there are infinitely many such degenerations that one needs to
check in order to determine whether a pair $(X,\xi)$ is K-stable, and
so there does not seem to be an effective way to test
K-stability. This become possible, however, in
certain situations with large symmetry group, where 
there are only a finite number of possible normal limits
$Y$ under equivariant degenerations of $X$. 
Just as in \cite{Datar-Sz}, a simple example is when $X$
is toric, in which case if we work equivariantly with respect to the
maximal torus $\mathbf{T}$, then we necessarily have $Y\cong X$, and we only need
to test the Futaki invariants $\mathrm{Fut}(X,\xi,\eta)$ for $\eta\in
\mathfrak{t}$. If this vanishes for all $\eta\in\mathfrak{t}$, then
$(X,\xi)$ admits a Ricci flat Kahler cone metric, recovering the
result of Futaki-Ono-Wang~\cite{FOW}.

A more general situation is when $X$ has a complexity-one action of a
torus $\mathbf{T}$, i.e. $\dim \mathbf{T} =
\dim X -1$. In this case, 
using the methods in Ilten-S\"u{\ss}~\cite{IS} we can still effectively
test K-stability, by checking a finite number of degenerations.
In Section~\ref{sec:examples} we will apply these techniques
to several explicit families of hypersurface singularities. For example,  we study

\[ Z_{BP}(p,q) = \{x^2 + y^2 + z^p + w^q = 0\} \subset \mathbf{C}^4,\]
where $p,q > 1$. We show the following. 

\begin{thm}\label{thm:mainexample}
  For a suitable choice of $\xi$ the pair $(Z_{BP}(p,q), \xi)$ admits
  a Ricci flat K\"ahler cone metric if and only if $2p > q$ and $2q >
  p$. As a consequence $S^5$ admits infinitely many families of
  Sasaki-Einstein metrics. 
\end{thm}

The necessary conditions $2p > q$ and $2q > p$ follow from the Lichnerowicz
obstruction of Gauntlett-Martelli-Sparks-Yau~\cite{GMSY}, while the
existence result was only known previously for $(p,q) = (2,2)$ and
$(2,3)$, where the latter was shown by Li-Sun~\cite{LiSun}.  

To date, many Sasaki-Einstein manifolds
have been found by employing estimates for the $\alpha$-invariant ~\cite{TianSurf, DemKo}. For example, the affine varieties $Z_{BP}(p,q)$ are a special
case of the Brieskorn-Pham singularities, which have been thoroughly
studied in the literature.  Boyer-Galick-Koll\'ar \cite{BGK} used
estimates for the $\alpha$-invariant of Brieskorn-Pham singularities
to produce $68$ distinct Sasaki-Einstein metrics on $S^5$, as well as
SE metrics on all 28 oriented diffeomorphism types of $S^7$, and the
the standard and Kervaire spheres $S^{4m+1}$.  Note that previously
infinitely many Einstein (not Sasakian) metrics on spheres in
dimensions 5 to 9 were constructed by B\"ohm~\cite{Bohm}. 

Estimates for the $\alpha$-invariant were also used by Boyer-Galicki
\cite{BG1, BG2}, Boyer-Nakamaye \cite{BN}, Koll\'ar-Johnson
\cite{KoJo}, Ghigi-Koll\'ar \cite{GK}, Koll\'ar \cite{Ko1, Ko2} and
others to produce many infinite families of Sasaki-Einstein metrics in
dimensions 5 and 7, and higher. For example, $\#k(S^2\times S^3)$ is
known to admit infinite families of Sasaki-Einstein metrics for any
$k\geq 1$.  We refer the reader to \cite{BGbook} for a thorough
discussion of these results.  We note that  Koll\'ar has classified
the possible topologies of Sasaki-Einstein manifolds \cite{Ko2, Ko3, Ko4}.
For example it is known that for affine varieties of
complex dimension $3$ with a $2$-torus action, 
the only possible topologies of the links are $S^5$ and $k\#(S^2\times
S^3)$ for any $k \geq 1$ (see \cite[Proposition 10.2.27]{BGbook}).
Our techniques also produce new infinite families of distinct
Sasaki-Einstein metrics on $k\#(S^2\times S^3)$ for all $k \geq 1$,
and hence cover all possible topologies that can occur with a
$2$-torus action.

We expect that many more examples can be found along the same lines.  
A particularly interesting
problem is to find Sasaki-Einstein metrics with irregular Reeb vector
fields.  Remarkably, the first examples 
of irregular Sasaki-Einstein metrics were
discovered by Gauntlett-Martelli-Sparks-Waldram \cite{GMSW} by explicitly writing down the metric in
coordinates.  We expect K-stability to be particularly useful for
finding irregular 
Sasaki-Einstein manifolds in real dimension 5, since if the cone $X$ has 
$\dim_{\bf C} X=3$, and $\xi$ is an
irregular Reeb field, then $X$ admits a complexity-one action of a
2-torus.  In particular, using the methods of Ilten-S\"u{\ss} \cite{IS}  we can effectively test whether $(X,\xi)$ admits a Ricci
flat K\"ahler cone metric.   

The overall strategy of our proof is the same as that of
Chen-Donaldson-Sun~\cite{CDS1,CDS2,CDS3}, as adapted in
\cite{Sz,Datar-Sz} to the smooth continuity method. We will set up
this continuity method in Section~\ref{sec:background}, where we also
give the precise definition of K-stability based on our previous work
\cite{CoSz}, extending the definition of Ross-Thomas~\cite{RT} from
the quasi-regular case.
 The main technical results
are contained in Sections~\ref{sec:Ding} and \ref{sec:partialC0}. In
Section~\ref{sec:Ding} we discuss weak solutions of the 
equations along the continuity method, which is analogous to the theory
of singular K\"ahler-Einstein metrics, as was studied by
Eyssidieux-Guedj-Zeriahi~\cite{EGZ}. Much of this discussion, such
as the convexity of the Ding functional due to Berndtsson~\cite{Ber},
extends to the case of cones without substantial difficulties. In
Section~\ref{sec:partialC0} we generalize the partial $C^0$-estimate
along the smooth continuity method from \cite{Sz} to the setting of
cones. The main new technical difficulty is that in the method of \cite{Sz} the
strict positivity of the Ricci curvature was a crucial ingredient,
while in our setting the Ricci curvature on a cone is never strictly
positive. Instead we need to exploit the transverse K\"ahler
structure, which does have strictly positive Ricci curvature.  The
proof of Theorem~\ref{thm:mainthm} is given in
Section~\ref{sec:mainproof}, primarily along the lines of the argument
in \cite{Datar-Sz}. In Section~\ref{sec:algfutaki} we collect some
more algebraic results, with the goal of establishing the equality
between the differential geometric and the algebraic definitions of
the Futaki invariant. In Section~\ref{sec:Kstab} we prove the other
implication in Theorem~\ref{thm:mainthm} along the lines of the work
of Berman~\cite{RBer}. In Section~\ref{sec:examples} 
we give some example calculations of K-stability, including the proof
of Theorem~\ref{thm:mainexample} and we finish with some further
discussion and questions in Section~\ref{sec:discuss}.

\section{Basic definitions}\label{sec:background}
In this section we fix some basic definitions, and set up the
continuity method that we would like to use to find Ricci flat
K\"ahler cone metrics. The continuity method is equivalent to
the usual continuity path for finding K\"ahler-Einstein metrics, but
it involves a scaling to ensure that we have metrics of non-negative
Ricci curvature on our cones.

\begin{defn}\label{defn:Fanocone}
 A \emph{polarized affine variety} of dimension $n$ is a triple $(X,\mathbf{T},\xi)$, where
  $X$ is a normal affine variety, $\dim_{\mathbf{C}}X=n$, $\mathbf{T}$ is a torus of automorphisms of $X$,
  and $\xi\in \mathfrak{t}$ acts on the ring of functions of $X$ with
  postive weights in the following sense. We have a decomposition
  \[ R(X) = \bigoplus_{\chi\in\mathfrak{t}^*} R_\chi(X) \]
  under the torus action into weight spaces, and we require that
  $\chi(\xi) > 0$ for all non-zero $\chi$, for which $R_\chi$ is
  non-trivial. Often we simply speak of a pair $(X,\xi)$, where
  $\xi$ is a vector field on $X$ generating a compact torus of
  automorphisms, and then $\mathbf{T}$ is understood to be this torus. We call
  $\xi$ a Reeb field or polarization on $X$.  We denote by
  $\mathcal{C}_{R}
 \subset \mathfrak{t}$
  the cone of Reeb fields.

  We say that the pair $(X,\xi)$ is a \emph{Fano cone singularity}, if $X$ is
  $\mathbf{Q}$-Gorenstein, together with a
  trivializing section $\Omega$ of $mK_X$ for some $m > 0$,
  such that $L_\xi \Omega = i\lambda \Omega$ for some $\lambda > 0$.
  The last condition is equivalent to $X$ having log-terminal
  singularities (see Section \ref{sec:algfutaki}).
  The Fano cone singularity $(X,\xi)$ is \emph{normalized} if $\lambda = nm$. 
\end{defn}

The basic example is obtained by taking a  Fano manifold $M$,
and letting $X$ be the total space of $mK_M$, with the zero
section blown down, for some $m$ such that $-mK_M$ is very ample.
In other words $X$ is the cone over $M$ under a projective
embedding by $-mK_M$. 

In \cite{CoSz} we defined a notion of K-semistability for a pair
$(X,\xi)$, in terms of test-configurations for $X$ that
commute with a torus $\mathbf{T}$ whose Lie algebra contains $\xi$. Here we
give a very similar definition, which is adapted to our work here, but
is closer in spirit to the definition of K-stability by
Tian~\cite{TianKEpos}, which only allows test-configurations with normal
central fibers. In addition, in view of possible future applications we work
equivariantly for a compact group acting on $X$ in analogy with
\cite{Datar-Sz}. 

Suppose that $(X,\xi)$ is a normalized Fano cone singularity of dimension $n$,
with $X$ only having
an isolated singularity, and $G$ is a compact group of automorphisms
of $X$, such that $\xi$ is in the center of its Lie algebra. In
applications we will take $G$ to be a maximal torus of automorphisms,
containing the torus generated by $\xi$.

A $G$-equivariant
\emph{special degeneration} (or test-configuration) 
of $X$ consists of an embedding
$X\to\mathbf{C}^N$ such that $G$ acts linearly through an embedding
$G \subset U(N)$, together with a one-parameter subgroup
$\lambda : \mathbf{C}^* \to GL(N)^G$ commuting with $G$, such that $\lambda(S^1)\subset
U(N)$ and $Y = \lim_{t\to 0} \lambda(t)\cdot X$ is normal. In this
case $(Y,\xi)$ is also a normalized Fano cone singularity, together with a
$\mathbf{C}^*$-action given by $\lambda$ commuting with $\xi$. Let us
write $T$ for the torus generated by $\lambda$ and $\xi$. 
By a slight abuse of notation we will denote by $\lambda\in\mathfrak{t}$ the
generator of the corresponding $S^1$-action, and note that for small
$s\in\mathbf{R}$ the pairs $(Y, \xi + s\lambda)$ are also Fano cone singularities
(which may not be normalized). We showed in \cite{CoSz} that the
index character
\[ F(\xi, t) = \sum_{\chi\in \mathfrak{t}^*}
e^{-t\chi(\xi)}\mathrm{dim}\,R_\chi(Y) \]
can be extended meromorphically to a neighborhood of the origin, and
we can define functions $a_i(\xi)$ by
\[ F(\xi,t) = \frac{a_0(\xi) (n-1)!}{t^n} + \frac{a_1 (\xi)
  (n-2)!}{t^{n-1}} + O(t^{2-n}). \]
As a matter of notation we will write
\begin{equation}\label{eq:Dwnotation}
 D_\lambda a_i(\xi) = \left. \frac{d}{ds}\right|_{s=0} a_i(\xi +
s\lambda). 
\end{equation}

\begin{defn}\label{defn:Futaki} The Futaki invariant of a special degeneration as above
  is defined to be
  \[ \mathrm{Fut}(X,\xi, \lambda) = \frac{a_0(\xi)}{n-1}
  D_\lambda\left(\frac{a_1}{a_0}\right)(\xi) + \frac{a_1(\xi)
    D_\lambda a_0(\xi)}{n(n-1) a_0(\xi)}, \]
  where each $a_i$ is computed on the variety $Y$. 

A normalized Fano cone singularity $(X,\xi)$ is called $G$-equivariantly K-stable, if for all
special degenerations as above, we have $\mathrm{Fut}(X,\xi,\lambda)
\geq 0$, and equality holds only if $(Y, \xi)$ is isomorphic to
$(X,\xi)$.
\end{defn}

Our main result, Theorem~\ref{thm:mainthm} then says that $X$ admits a
Ricci flat K\"ahler cone 
metric with homotheties given by $-J\xi$, if and only if 
$(X,\xi)$ is equivariantly K-stable.
  
Let us digress briefly on the Futaki invariant in
Definition~\ref{defn:Futaki}.  As remarked above, given a special test
configuration generated by a $\mathbf{C}^{*}$ action $\lambda$, the
central fiber $Y$ is again a $\mathbf{Q}$-Gorenstein variety with log
terminal singularities, and hence $(Y, \xi+s\lambda)$ is a Fano cone singularity,
which is not necessarily normalized.  As we will see in
Proposition~\ref{prop:a1a0formula}, the normalized Reeb vector
fields form a linear subspace $\Sigma_{Y}$ of the Reeb cone of $Y$
defined by the {\em linear} equation
\[
\frac{a_{1}(w)}{a_{0}(w)} = \frac{n(n-1)}{2}.
\]
Let $N := \nabla(a_1/a_0)$ denote the normal vector to the normalized
hyperplane $\Sigma_{Y}$.  

First consider the case when $Y \cong X$,
and that $\lambda$ is generated by $w \in {\rm Lie}(T)$.  Assume that
$w$ is tangent to $\Sigma_{X}$, then the Futaki invariant is just
$(1/2)D_{w}a_0(\xi)$.  Since we can also consider $-w$, this implies
that if $(X,\xi)$ is $K$-stable, then $\xi$ must be an extremal value
of $a_0$, on $\Sigma_X$.  By \cite{MSY}, $a_0$ is a convex function on
$\Sigma_X$ which has a unique minimum.  Since $a_0$ can be interpreted
as the volume of the link, this is called ``volume minimization", and
was discovered in fundamental work of Martelli-Sparks-Yau \cite{MSY}.

Now suppose we have a nontrivial test configuration, so that $Y
\not\cong X$, and $\lambda$ is generated by $w$ and suppose that $\xi$
is the Reeb vector field minimizing the volume.  
We can compute the Futaki invariant by the formula
\[
\frac{1}{2}D_{w'} a_{0}(\xi) = {\rm Fut}(X,\xi, \lambda), \quad \text{ if } \quad w' =w- 2\frac{N\cdot w}{n(n-1)} \xi,
\]
where now $w'$ is normalized, but it may not generate a test
configuration if $\xi$ is irrational. This observation extends the
interpretation of stability as volume minimization
\cite{MSY,MSY2,CoSz} from trivial test configurations to all test
configurations, and will be useful in Section~\ref{sec:algfutaki}.
Note that when $Y \not\cong X$ we cannot replace $w$ with $-w$, since
this will change the central fiber of the test configuration.  These
observations have applications in conformal field theory, where the
AdS/CFT correspondence provides an interpretation of K-stability as a
maximization problem for the central charge of the dual conformal
field theory \cite{CXY}.

We next set up the relevant continuity method for polarized affine manifolds. 
Suppose that $(X,\xi)$ is normalized Fano. Fixing any metric $\alpha$
on $(X,\xi)$, our continuity method is to find metrics $\omega_t$
on $(X,\xi)$ satisfying
\begin{equation} \label{eq:contmethod}
  \mathrm{Ric}(\omega_t) = 2n[t\omega_t^\tau + (1-t) \alpha^\tau]- 2n\omega_t^{\tau},
\end{equation}
where $\omega_t^\tau, \alpha^\tau$ denote the transverse metrics
induced by $\omega_t, \alpha$.  In terms of the transverse metrics induced on the Reeb
foliation of the link $L = \{r_t=1\}$ the method of continuity is
\[
\Ric^{\tau}(\omega_{t}^{\tau}) = 2n[t\omega_t^\tau + (1-t) \alpha^\tau].
\]
In particular, ~\eqref{eq:contmethod} is the natural lift to the cone of the continuity
method for K\"ahler-Einstein metrics (see, for example, \cite{Aub}). We will
call Equation~\ref{eq:contmethod} the twisted equation, with twisting
form $\alpha^\tau$. 

\begin{prop} \label{prop:contmethod} 
Let $I = \{t \in [0,1] : \eqref{eq:contmethod} \text{ has a solution } \}$.  Then $I$ is non-empty, and open. 
\end{prop}

The non-emptiness follows from the transverse version of Yau's theorem \cite{Y}, due 
to El-Kacimi Alaoui \cite{ElKac}, while the openness is also analogous
to the K\"ahler case as in Aubin~\cite{Aub}. 
 
As in the compact K\"ahler case, we  must study the Gromov-Hausdorff limit of a
sequence $(X, \omega_{t_i})$ as $t_i \to T$. For this it is convenient
to do a scaling of the Reeb fields to ensure that we have metrics with
Ricci curvature bounded below. Let us denote the radial function of
$\omega_t$ by $r_t$, and define $\wt{r}_t = r_t^t$; in the Sasakian literature, this 
is often referred to as a $D$-homothetic transformation. It is straightforward to
verify that
$\wt{\omega}_t = \ddb \wt{r}_t$ satisfies
\[ \mathrm{Ric}(\wt{\omega}_t) = 2n\frac{1-t}{t} \alpha^\tau, \]
i.e. the Ricci curvature is non-negative. 

\section{Weak solutions, twisted Futaki invariants and
  the Ding functional}\label{sec:Ding}
The key result that we will ultimately need is that in the context of
the continuity method defined in the previous section, as $t\to\sup
I$, we can extract a limit that is a normalized Fano cone singularity $(Y, \xi)$ together with a
transverse positive current $\beta^\tau$, and a weak
solution $\omega_T$ on $(Y, \xi)$ of the equation
\[ \mathrm{Ric}(\omega_T) = 2n(1-T)[\beta^\tau-\omega_{T}^{\tau}]. \]
In this section we will give a precise definition of such weak
solutions, and describe how in analogy with the compact K\"ahler case,
the existence of such a metric implies the reductivity of a certain
automorphism group, and the vanishing of a twisted Futaki invariant.
 
We first define the weak solutions of the twisted equation. We assume
that $(Y,\xi)$ is a
normalized Fano cone singularity, so we have a non-vanishing global
holomorphic section $\Omega$ of $mK_{Y}$ for some $m \geq 1$,
with $L_\xi(\Omega) = imn\Omega$.  This gives rise to  the volume
form 
\[
dV = i^{n^2}\left(\Omega\wedge \overline{\Omega}\right)^{1/m},
\]
which satisfies $L_\xi dV = 2n dV$. This volume form is uniquely
defined up to a constant multiple, and we will call it the canonical
volume form on a normalized Fano cone singularity. 

Suppose that we have an embedding $Y\to \mathbf{C}^N$, such that the
Reeb field (or rather the torus it generates) acts diagonally, 
and $Y$ is not contained in
a linear subspace.  Then $\xi$ defines a Reeb field on $\mathbf{C}^N$ and so we can
then fix a smooth reference radial function $\hat{r}$ on $\mathbf{C}^N$ which
is compatible with this Reeb field \cite{CoSz}. 
In the presence of the action of a
torus $\mathbf{T}$, we can take our embedding to be
$\mathbf{T}$-equivariant as well, where $\mathbf{T}$ acts diagonally.  

The space of transverse
psh potentials is the space of basic functions $\phi$ (i.e. $L_\xi
\phi = L_{J\xi}\phi = 0$) such
that $r_\phi = \hat{r}e^\phi$ is psh.  Recall that a psh function on a normal variety can always be viewed as
the restriction of a psh function from an ambient space, after embedding \cite[Theorem 1.10]{Dem85}.
 In particular, $r_{\phi}$ is always the restriction
of a psh function defined in a neighborhood of the origin. For smooth such $\phi$ we write
\[ \omega_\phi = \frac{1}{2}\ddb r_\phi^2, \]
and we suppose that we have a twisting form $\beta^\tau$  given as
\[ \beta^\tau = \ddb \log r_\psi, \]
where $\psi$ is also a transverse psh potential. 
If $Y, \beta^\tau$ were smooth, then the twisted equation
\begin{equation}\label{eq:t2}
  \mathrm{Ric}(\omega_\phi) = 2n(1-t)[\beta^\tau - \omega_t^\tau]
\end{equation}
could be written on the level of volume forms as
\begin{equation}\label{eq:t3}
  \omega_\phi^n = C e^{2n(1-t)[\phi -\psi]} dV,
\end{equation}
and we can use this latter formulation to define solutions of the
twisted equation in a weak sense (see \cite{EGZ} for the
analogous definition of weak K\"ahler-Einstein metrics).
More precisely, given $\psi, t$, a weak solution of \eqref{eq:t2} is a
continuous transverse psh potential $\phi$, such that \eqref{eq:t3}
holds as an equality of measures. In particular this implies that
$e^{-2n(1-t)\psi}dV$ must be integrable in a neighborhood of the cone
vertex.

\begin{defn}\label{defn:weaksoln} Let  $(Y,\xi)$ be a
  normalized Fano cone singularity together with a reference radial function $\hat
  r$ as above. Suppose that $\psi$ is a transverse psh potential such
  that $e^{-2n(1-t)\psi}dV$ is integrable in a neighborhood of the
  vertex for some $t\in [0,1]$. We say that $(Y, \xi, (1-t)\psi)$
  admits a (weak) solution of the twisted equation if $(Y,\xi)$ has a continuous
  transverse psh potential $\phi$ satisfying
  \[\omega_\phi^n = C e^{2n(1-t)[\phi -\psi]} dV\]
  in the sense of measures. Often we will write $(Y, \xi,
  (1-t)\beta^\tau)$ when the twisting form $\beta^\tau$ is more
  natural than its potential $\psi$. 
\end{defn}

  As in \cite{Datar-Sz}, Remark 4 (see also \cite[Proposition 3.8]{BEGZ}), it is
  enough to check that outside of a closed set  $\Sigma$ with vanishing $(2n-2)$-dimensional
  Hausdorff measure, $\omega_\phi^n$ defines a singular metric
  $e^{-f}$ on $K_Y$ with $f \in L^1_{loc}$, and in addition,  on $Y\backslash \Sigma$
  \[ \ddb f = 2n(1-t)[ \beta^\tau - \omega_\phi^\tau]. \]

The main properties of weak solutions of the twisted equation are the
reductivity of the automorphism group, and the vanishing of the
twisted Futaki invariant, analogous to Propositions 7, 8 in
\cite{Datar-Sz}. Let us first define the relevant automorphism
group, or rather its Lie algebra.

\begin{defn}\label{defn:stab}
  Suppose we have a triple $(Y, \xi, \beta^\tau)$ as
  above. We define $\mathfrak{g}_{Y,\xi}$ to be the space of
  holomorphic vector fields on $Y$ (more precisely on the regular
  part), commuting with $\xi$. We then define
  \[ \mathfrak{g}_{Y,\xi,\beta^\tau} = \{ w \in \mathfrak{g}_{Y,\xi}
  \,:\, \iota_w\beta^\tau = 0\}. \]
\end{defn}

As in \cite{Datar-Sz}, note that $\mathfrak{g}_{Y,\xi, \beta^\tau}$
is in general smaller than the space of holomorphic vector fields $w$ on
$Y$ commuting with $\xi$, and preserving $\beta^\tau$ in the sense
that $L_w \beta^\tau = 0$. For instance if $\beta^\tau$ is a
transverse K\"ahler form, then $\mathfrak{g}_{Y,\xi, \beta^\tau}$ is
spanned by $\Xi$, where $\mathrm{Im}\,\Xi = \xi$. 

The following is analogous to \cite[Proposition 7]{Datar-Sz}, the
special case of which, without the twisting form, has been shown in
Donaldson-Sun~\cite{DS2}. The proof is based on a uniqueness theorem
due to Berndtsson~\cite{Ber2}, generalizing the classical
Bando-Mabuchi theorem~\cite{BandoMabuchi},
further extended by Boucksom-Eyssidieux-Guedj-Zeriahi~\cite{BEGZ}, Berman-Witt-Nystr\"om~\cite{BWN}, Chen-Donaldson-Sun~\cite{CDS3}, Yi~\cite{Yi}. 
\begin{prop}\label{prop:reductive}
 Suppose that $(Y, \xi, (1-t)\beta^\tau)$ admits a solution $\omega_t$
 of the twisted equation. Then $\mathfrak{g}_{Y,\xi, \beta^\tau}$ is
  reductive. Moreover if $G$ is a compact group of biholomorphisms of
  $Y$, commuting with $\xi$ and fixing $\omega_t$, then the
  centralizer $(\mathfrak{g}_{Y,\xi, \beta^\tau})^G$ is also reductive. 
\end{prop}

Next we define the twisted Futaki invariant. Suppose that $w$ is
a vector field commuting with $\xi$, that
preserves the radial function $\hat{r}$
above. In addition suppose that $w$ is the real part of a holomorphic
vector field in $\mathfrak{g}_{Y,\xi, \beta^\tau}$. The
transverse Hamiltonian $\theta_w$ is defined by letting $\theta_w
\hat{r}^2$ be a Hamiltonian for $w$, i.e. satisfying the equation
\begin{equation}\label{eq:tham}
 \theta_w \hat{r}^2 = \iota_{Jw}
d\left(\frac{1}{2}\hat{r}^2\right). 
\end{equation}
Note that with this convention the transverse Hamiltonian of the Reeb field $\xi$ is
$\theta_\xi=-1$. 
We then define the twisted Futaki invariant to be
\[ \mathrm{Fut}_{Y,\xi, (1-t)\beta^\tau}(w) =
\frac{t}{V} \int_Y \theta_w e^{-\frac{1}{2} \hat{r}^2} \omega^n -
t \frac{
  \int_Y \theta_w e^{-2n(1-t)\psi}
  e^{-\frac{1}{2}\hat{r}^2} dV }{\int_Y e^{-2n(1-t)\psi}
  e^{-\frac{1}{2}\hat{r}^2} dV}. \]
Proposition~\ref{prop:equalFut} will show that for $t=1$, this
definition agrees with the algebraic definition given in
Definition~\ref{defn:Futaki}. Below we will also give a different
formula for the twisted Futaki invariant when $t\ne 1$. 

We have the following, analogous to \cite[Proposition 8]{Datar-Sz}.

\begin{prop}\label{prop:Fut0} If $(Y, \xi)$ admits a weak solution of
  the twisted equation above, then $\mathrm{Fut}_{Y,\xi,
    (1-t)\beta^\tau}(w) = 0$ for all vector fields $w$ as above.
\end{prop}

The proofs of Propositions~\ref{prop:reductive}, \ref{prop:Fut0} both
follow from convexity properties of the twisted Ding functional along weak
geodesic segments, based essentially on work of
Berndtsson~\cite{Ber2}. The arguments follow those in \cite{Ber2} (see
also \cite[Section 6]{Datar-Sz})
closely together with the discussion in Donaldson-Sun~\cite{DS2} on
extending these to the setting of cones. See also
Guan-Zhang~\cite{GZ11} for geodesics of Sasakian metrics. 

\begin{defn}\label{defn:twistedDing}
  Suppose that $(Y,\xi)$ is a normalized Fano cone singularity, and $\psi$ is a
  transverse psh potential as above (relative to a reference radial
  function), such that $e^{-2n(1-t)\psi} dV$ is locally
  integrable. The twisted Ding functional $\mathcal{D}_{(1-t)\psi}$ is
  defined for continuous transverse psh potentials $\phi$ by
\[ \mathcal{D}_{(1-t)\psi}(\phi) = -tE(\phi) - \frac{1}{2n} \log \int_Y
e^{2n(1-t)[\phi - \psi]} e^{-\frac{1}{2}r_\phi^2} dV. \]
Here $E$ is defined by its variation:
\[ \delta E(\phi) = \frac{1}{V(\xi)} \int_Y \dot{\phi} e^{-\frac{1}{2}r_\phi^2}
\omega_\phi^n. \]
\end{defn}

The properties of the twisted Ding functional in this setting follow
calculations analogous to those in the compact K\"ahler case that was
studied in \cite{Datar-Sz} (see Ding-Tian~\cite{DingTian},
Berndtsson~\cite{Ber2}, Chen-Donaldson-Sun~\cite{CDS3} for earlier
work), using some additional identities that are valid on cones. An
alternative approach is to work with pluripotential techniques on the
link, as developed recently by van Coevering~\cite{vCoev}, but we
prefer to work directly on the cone. As a sample of the calculations
involved we have the following simple result. 

\begin{lem} $E(\phi)$ is well defined.
\end{lem}
\begin{proof}
  Let us consider first the restriction of $E$ to smooth transverse
  psh potentials $\phi$. 
  We show that the 1-form defined by $\delta E$ is
  closed. Note that the variation of $\frac{1}{2} r_\phi^2$ is $\delta(\frac{1}{2}r_\phi^2) =
  (\delta \phi) r_\phi^2$. The differential of $\delta E$ is the map
  \[ (\psi_1, \psi_2) \mapsto -\int_Y \psi_1\psi_2 r_\phi^2  e^{-\frac{1}{2}r_\phi^2} \omega_\phi^n +n\int_{Y}\psi_1e^{-\frac{1}{2}r_\phi^2}
  \ddb(\psi_2 r_\phi^2)\wedge \omega_\phi^{n-1}, \]
    and we need to show that this is symmetric in $\psi_1, \psi_2$. We
  have
  \[ \begin{aligned} \int_Y \psi_1 \Delta(\psi_2 r_\phi^2)
    e^{-\frac{1}{2}r_\phi^2} \omega_\phi^n &=
    \int_Y \Big[ \psi_1\Delta\psi_2\, r_\phi^2 + \psi_1
    \nabla\psi_2\cdot \nabla r_\phi^2 + \psi_1\psi_2 \Delta
    r_\phi^2\Big] e^{-\frac{1}{2}r_\phi^2}\,\omega_\phi^n \\
    &=\int_Y \Big[ -\frac{1}{2} \nabla\psi_1\cdot\nabla\psi_2 r_\phi^2
    + 2n \psi_1\psi_2\Big] e^{-\frac{1}{2}r_\phi^2} \omega_\phi^n, 
  \end{aligned}\]
  where the integration by parts is justified since $r_\phi^2 = O(\hat{r}^2)$ and $\nabla r_\phi^2 =
  O(\hat{r})$, and we used
  \[ \nabla\psi_2\cdot \nabla r_\phi = 0, \quad \Delta
  \frac{1}{2}r_\phi^2 = n. \]
  We obtain that the differential of $\delta E$ is
  \[ (\psi_1, \psi_2) \mapsto n \int_Y \nabla\psi_1\cdot
  \nabla\psi_2\, e^{-\frac{1}{2}r_\phi^2} \omega_\phi^n, \]
  where we used the formula
  \begin{equation}\label{eq:intr^2}
    \int_Y f r_\phi^2 e^{-\frac{1}{2}r_\phi^2}\,dV = 2n \int_Y f
    e^{-\frac{1}{2} r_\phi^2}\,dV,
    \end{equation}
for any basic function $f$.

  This shows that $E$ is well defined on the space of smooth
  transverse psh potentials. We can then extend $E$ to the space of
  continuous transverse psh potentials by continuity, since the
  formula for the variation of $E$ implies that $E$ is
  uniformly continuous for the $L^\infty$ norm on potentials. 
\end{proof}

The critical points of the twisted Ding functional are given by
solutions of the twisted equation. To see this, note that the
variation of $\mathcal{D}_{(1-t)\psi}(\phi)$ is given by
\[ \begin{aligned}
  \delta\mathcal{D}_{(1-t)\psi}(\phi) &= \frac{-t}{V} \int_Y \dot\phi
  e^{-\frac{1}{2}r_\phi^2} \omega_\phi^n \\
  &\quad\quad - \frac{1}{2n}\frac{\int_Y \Big[
    2n(1-t)\dot\phi - \dot\phi r_\phi^2\Big] e^{2n(1-t)[\phi-\psi]}
    e^{-\frac{1}{2}r_\phi^2} \,dV}{\int_Y
    e^{2n(1-t)[\phi - \psi]} e^{-\frac{1}{2}r_\phi^2} dV} \\
  &= \frac{-t}{V} \int_Y \dot\phi
  e^{-\frac{1}{2}r_\phi^2} \omega_\phi^n + t \frac{\int_Y \dot\phi
    e^{2n(1-t)[\phi-\psi]}
    e^{-\frac{1}{2}r_\phi^2} \,dV}{\int_Y
    e^{2n(1-t)[\phi - \psi]} e^{-\frac{1}{2}r_\phi^2} dV},
\end{aligned} \]
where we used \eqref{eq:intr^2} again. It follows that critical points
satisfy
\[ \omega_\phi^n = Ce^{2n(1-t)[\phi-\psi]} dV, \]
which is what we wanted. 

In addition from this calculation of the variation we see that the
variation of $\mathcal{D}_{(1-t)\psi}$ along a suitable 1-parameter
family of biholomorphisms recovers the twisted Futaki invariant. 
Suppose that $w$ is a vector field as
above, and let $f_s : Y\to Y$ denote the 1-parameter group of
biholomorphisms generated by $-Jw$.  We claim that the
 twisted Futaki invariant is given by the
variation of the twisted Ding functional along $f_s$. Writing $\phi_s$
for the induced family of potentials, we have
\[ \frac{1}{2} \hat{r}^2e^{2\phi_s}= f_s^*(\frac{1}{2}\hat{r}^2e^{2\phi}), \]
and so
\[ \dot\phi_s r_\phi^2 = -\iota_{Jw} d(\frac{1}{2} r_\phi^2). \]
We obtain that $\dot\phi = -\theta_w$ in terms of the transverse
Hamiltonian of $w$ as in \eqref{eq:tham}. It follows that
\[ \left.\frac{d}{ds}\right|_{s=0} \mathcal{D}_{(1-t)\psi}(\phi_s) =
\frac{t}{V} \int_Y \theta_w e^{-\frac{1}{2} r_\phi^2} \omega_\phi^n -
t \frac{
  \int_Y \theta_w e^{2n(1-t)[\phi - \psi]}
  e^{-\frac{1}{2}r_\phi^2} dV }{\int_Y e^{2n(1-t)[\phi - \psi]}
  e^{-\frac{1}{2}r_\phi^2} dV}, \]
which when $\phi=0$ is just
the twisted Futaki invariant as we have defined it above. 
We can rewrite this in a different form, as in the proof of
Proposition 8 in \cite{Datar-Sz}. 
\begin{prop}\label{prop:FuteqProp}
 The twisted Futaki invariant is given by
\[ \mathrm{Fut}_{Y,\xi, (1-t)\psi}(w) = \mathrm{Fut}_{Y,\xi}(w) -
n(n-1)\frac{1-t}{V} \int_Y \theta_w e^{-\frac{1}{2}\hat{r}^2} \ddb \psi \wedge
\omega^{n-1}, \]
where
\begin{equation}\label{eq:Futdef2}
 \mathrm{Fut}_{Y,\xi}(w) = \frac{1}{V} \int_Y \theta_w
e^{-\frac{1}{2}\hat{r}^2} \omega^n - \frac{\int_Y \theta_w e^{-\frac{1}{2}\hat{r}^2}dV}{\int_Y
  e^{-\frac{1}{2}\hat{r}^2} dV} 
\end{equation}
is the ``untwisted'' Futaki invariant.
Note that here, as before, we are assuming that $\iota_{Jw} \ddb\psi =
0$, since $w$ is the real part of a holomorphic vector field in
$\mathfrak{g}_{Y, \xi, \beta^\tau}$. 
\end{prop}
Note that in addition when $w$ is normalized, i.e. $L_w\Omega = 0$, then we have
\begin{equation}\label{eq:normalizedFut}
 \mathrm{Fut}_{Y,\xi}(w) = \frac{1}{V} \int_Y \theta_w
e^{-\frac{1}{2}\hat{r}^2} \omega^n, 
\end{equation}
since in this case we have
\[ \int_Y \theta_w e^{-\frac{1}{2}\hat{r}^2} dV = 0, \]
as can be seen by considering the variation of the integral $\int_Y
e^{-\frac{1}{2}\hat{r}^2} dV$ along the flow generated by $Jw$. 

\begin{proof}[Proof of Proposition~\ref{prop:FuteqProp}]
  We define
  \[ \begin{aligned} I(\phi) &= \frac{1}{V} \int_Y \log \frac{ \left(\int_Y e^{-\frac{1}{2}r_\phi^2}
      dV\right)^{-1} e^{-\frac{1}{2}r_\phi^2}dV}{ \left(\int_Y
      e^{2n(1-t)[\phi - \psi]} e^{-\frac{1}{2}r_\phi^2} dV\right)^{-1}
    e^{2n(1-t)[\phi-\psi]}
e^{-\frac{1}{2}r_\phi^2} dV} e^{-\frac{1}{2}r_\phi^2}
  \omega_\phi^n \\
  &= \log \frac{\int_Y e^{2n(1-t)[\phi-\psi]}e^{-\frac{1}{2}r_\phi^2}
    dV}{\int_Y e^{-\frac{1}{2}r_\phi^2}dV} -
  2n\frac{1-t}{V} \int_Y[\phi-\psi] e^{-\frac{1}{2}r_\phi^2}\omega_\phi^n.
\end{aligned}\]
  Differentiating along the one-parameter group generated by $Jw$ we
  must get zero. To see this, we use that all the terms in the
  integral are invariant under biholomorphisms up to constant factors,
  and these constants cancel. For instance that fact that
  $\iota_{Jw}\beta^\tau=0$ implies that $L_{Jw}\psi$ is a constant.  The
  result of the differentiation is the following. 
  \[ \begin{aligned}
    2n\frac{\int_Y \dot\phi e^{-\frac{1}{2}r_\phi^2}dV}{\int_Y e^{-\frac{1}{2}r_\phi^2}dV} &-
    2nt \frac{\int_Y \dot\phi e^{2n(1-t)[\phi -
        \psi]}e^{-\frac{1}{2}r_\phi^2} dV}{\int_Y e^{2n(1-t)[\phi-\psi]}
      e^{-\frac{1}{2}r_\phi^2} dV} - 2n \frac{1-t}{V}\int_Y \dot\phi
    e^{-\frac{1}{2}r_\phi^2}\omega_\phi^n \\
    &- 2n \frac{1-t}{V}\int_Y (\phi-\psi) \big[-\dot\phi r_\phi^2 +
    \Delta(\dot\phi r_\phi^2)\big] e^{-\frac{1}{2}r_\phi^2}\omega_\phi^n = 0.
  \end{aligned}\]
  Similar calculations to before, using also that $r_\phi^2
  \Delta(\phi-\psi)$ is a basic function,  show that
  \[ \int_Y (\phi-\psi) \big[-\dot\phi r_\phi^2 + \Delta(\dot\phi
  r_\phi^2)\big] e^{-\frac{1}{2}r_\phi^2} \omega_\phi^n = 2n(n-1) \int_Y
  e^{-\frac{1}{2}r_\phi^2} \dot\phi \ddb(\phi-\psi) \wedge
  \omega_\phi^{n-1}. \]
 We obtain
  \[ \begin{aligned}
    \frac{\int_Y \dot\phi e^{-\frac{1}{2}r_\phi^2}dV}{\int_Y e^{-\frac{1}{2}r_\phi^2}dV} &-
    t \frac{\int_Y \dot\phi
      e^{2n(1-t)[\phi-\psi]}e^{-\frac{1}{2}r_\phi^2}dV}{\int_Y
      e^{2n(1-t)[\phi-\psi]}
      e^{-\frac{1}{2}r_\phi^2}dV} = \frac{1-t}{V}\int_Y \dot\phi
    e^{-\frac{1}{2}r_\phi^2}\omega_\phi^n \\
    & +n(n-1)\frac{1-t}{V}\int_Y \dot\phi
    e^{-\frac{1}{2}r_\phi^2}\ddb (\phi-\psi)\wedge \omega_\phi^{n-1}. 
  \end{aligned}\]
  From this, using that at the reference metric $\phi=0$ we have $\dot\phi = \theta_w$, 
  we obtain the required formula.
\end{proof}

We can write the twisting term above in a more intrinsic way as
follows:
\[ -\int_Y \theta_w e^{-\frac{1}{2}\hat{r}^2} \ddb \psi \wedge
  \omega^{n-1} = -\int_Y \theta_w e^{-\frac{1}{2}\hat{r}^2}
  (\beta^\tau - \omega^\tau)\wedge \omega^{n-1}, \]
recalling that $\beta^\tau = \ddb \log r_\psi$ and $\omega^\tau =
\ddb \log \hat{r}$. 

As in \cite{Datar-Sz} we use the twisted Futaki invariant to define a notion
of (equivariant) twisted stability following also Dervan~\cite{Dervan}. We
will also choose a twisting form of a special form for which we will
calculate an alternative formula for the twisted Futaki invariant. 

Suppose that $(X,\xi)$ is a normalized Fano cone singularity, smooth
away from the vertex, and write $R(X)$ for the coordinate ring as
before. We also denote by $R_{<D}(X)$ the direct sum of weight spaces
for the action of $\mathbf{T}$ with weights $\chi\ne 0$ such that
$\chi(\xi) < D$. Suppose that $D$ is large enough so that the
functions in $R_{<D}(X)$ give an
 embedding $X\hookrightarrow \mathbf{C}^N$. It will be convenient to
 separate the spaces corresponding to different characters as
\[ \mathbf{C}^N = \mathbf{C}^{N_1} \times\ldots\times \mathbf{C}^{N_m}, \]
and so $\xi$ acts diagonally on the coordinate functions
 with weights $a_1,\ldots, a_m > 0$. 

A test-configuration for $X$ commuting with $\mathbf{T}$ is 
given by a one-parameter subgroup
$\lambda : \mathbf{C}^*\to GL(N)^{\mathbf{T}}$, and
we assume that $\lambda(S^1)\subset U(N)^{\mathbf{T}}$,
generated by a vector field $w$. Note that $GL(N)^{\mathbf{T}}$ is simply the
product of $GL(N_i)$ for $i=1,\ldots,m$. By a
further unitary change of basis we will assume that $\lambda$ is diagonal. We 
define
\[ Y = \lim_{t\to 0} \lambda(t)\cdot X, \]
and we suppose that $Y$ is normal, and $\mathbf{Q}$-Gorenstein. We define the
limiting current
\[ \beta^\tau = \lim_{t\to 0} \lambda(t)\cdot \alpha^\tau, \]
whose existence follows from the proof of
Proposition~\ref{prop:twistedformula} below, see
\eqref{eq:betataueq}. 
The twisted Futaki invariant of the corresponding
test-configuration is then defined by
\[ \mathrm{Fut}_{X, \xi, (1-t)\alpha^\tau}(w) = \mathrm{Fut}_{Y, \xi,
(1-t)\beta^\tau}(w).\] 
Note here that $w$ is not tangent to $X$, but it is tangent to $Y$,
and it is the real part of a holomorphic vector field in
$\mathfrak{g}_{Y, \xi, \beta^\tau}$.

As in \cite{Datar-Sz}, a crucial role is played by an alternative formula for
this twisted Futaki invariant. We assume that $\xi$ is quasi-regular.
There is then a constant $M$ such
that $\frac{M}{a_i}\in \mathbf{Z}$. We define a reference radial
function analogous to the Fubini-Study metric for projective
varieties, with radial function $\hat{r}$ given by
\begin{equation}\label{eq:hatr2}
 \hat{r}^2 = \left[\sum_{i=1}^m \left(\sum_{j=1}^{N_i}
    |z^{(i)}_j|^2\right)^{M/a_i}\right]^{1/M},
\end{equation}
where the $\{z^{(i)}_j\}$ form an orthonormal basis for
$\mathbf{C}^{N_i}$. 
We take the background metric to be $\omega = \frac{1}{2}\ddb
\hat{r}^2$, and the transverse form 
$\alpha^\tau = \ddb \log \hat{r}$, i.e. $\psi=0$ relative to
this radial function. To see that $\omega$ is indeed a metric, note
first that for any cone metric $\ddb r^2$ and $\gamma > 0$, the forms
$\ddb r^{2\gamma}$ also define cone metrics, with Reeb fields obtained
by scaling. In this way, 
\[ \sum_{i=1}^m \left( \sum_{j=1}^{N_i}
  |z^{(i)}_j|^2\right)^{M/a_i} \]
defines a product metric on $\mathbf{C}^N$, and it follows that
$\omega$ is also a cone metric.  In addition $\hat{r}$ is preserved by
the action of $U(N)^{\mathbf{T}}$. With this setup we have the following.

\begin{prop}\label{prop:twistedformula}
  \[\mathrm{Fut}_{X,\xi,(1-t)\alpha^\tau}(w) = \mathrm{Fut}_{X, \xi}(w)
  + c(n) \frac{1-t}{V} \int_Y (\max_Y\theta_w - \theta_w)
  e^{-\frac{1}{2}\hat{r}^2}\,\omega^n, \]
  where $\omega = \ddb \frac{1}{2}\hat{r}^2$, and $c(n)$ is a
  dimensional constant. 
\end{prop}
\begin{proof}
  The proof follows a similar argument to that in \cite[Proposition
  11]{Datar-Sz}, expressing
  $\alpha^\tau$ as an average of currents of integration along
  hypersurfaces in $X$. One new difficulty is that the limit $Y$ may
  be contained in a coordinate hyperplane. 

  We will use hypersurfaces defined by functions of the form
  \[ f_\mu = \sum_{j=1}^B \mu_j u_j^{M/b_j}, \]
  where the $u_j$ are monomials in the $z_i$ (including each $z_i$ as
  well),  the $b_j$ are  corresponding weights,
  and we think of $\mu\in \mathbf{P}= \mathbf{P}^{B-1}$. It may happen that $X$ is
  contained in some of these hypersurfaces, but there is a linear
  subspace $E\subset \mathbf{P}$ such that for $\mu\in
  \mathbf{P}\setminus E$, the function $f_\mu$ does not vanish on
  $X$. Let us write $V_\mu = X\cap f_\mu^{-1}(0)$, which may have
  multiplicity. By Shiffman-Zelditch~\cite[Lemma 3.1]{ShiffZel}, whose
  proof is entirely local, we have that on
  $\mathbf{C}^N$, 
  \[ 2\pi\int_{\mathbf{P}} [f_\mu^{-1}(0)]\,d\mu = \ddb \log
  \left( \sum_{j=1}^B |u_j|^{2M/b_j}\right), \]
  in the sense of distributions,
  for the standard probability measure $d\mu$ on the projective space
  $\mathbf{P}$. Multiplying out the $M/a_i$ power in \eqref{eq:hatr2}
  we see that for suitable choices of the $u_j$ restricted to $X$
  we will have
  \begin{equation}\label{eq:alphaintegral}
  \alpha^\tau = \frac{\pi}{M}\int_{\mathbf{P}\setminus E}
  [V_\mu]\,d\mu = \frac{1}{2M} \ddb \log \left( \sum_{j=1}^B
    |u_j|^{2M/b_j}\right). 
\end{equation}

  We can then compute $\beta^\tau$ on $Y$ by taking the limits of the
  currents $[V_\mu]$ under the $\mathbf{C}^*$-action $\lambda$. We
  will do this by computing the limit of the underlying schemes. 
  Suppose that $\lambda(t)$ acts on the $u_j$ 
  diagonally with entries $t^{w_j}$. Then we have
  \[ \lambda(t)\cdot f_\mu  =  \sum_{j=1}^B \mu_j
    t^{\frac{w_jM}{b_j}} u_j^{M/b_j}. \]
  If $I$ denotes the homogeneous ideal defining $X$, then $V_\mu$ is
  defined by the ideal $I + (f_\mu)$. The weights of the $\mathbf{C}^*$-action
  $\lambda$ define a partial order on the monomials, and the limit
  $\lambda(t)\cdot V_\mu$ has ideal defined by the lowest weight
  parts of elements of $I+(f_\mu)$, i.e. the initial ideal
  $\mathrm{in}_\lambda(I + (f_\mu))$. In the same way the limiting variety $Y$ is
  defined by the ideal $\mathrm{in}_\lambda(I)$.

  Suppose that we find
  a function $g_\mu\in \mathrm{in}_\lambda(I+(f_\mu))$ such that $g_\mu \not\in
  \mathrm{in}_\lambda(I)$, and in terms of the grading of the
  polynomial ring given by $\xi$ the function $g_\mu$ has the same degree
  as the $f_\mu$ (i.e. degree $M$). In this case, since the coordinate
  ring of $Y$ is an integral domain (we have assumed that $Y$ is
  reduced and irreducible), we will see that the Hilbert functions of
  the ideals $I + (f_\mu)$ and $\mathrm{in}_\lambda(I) + (g_\mu)$
  coincide, and so $\lim_{t\to 0} V_\mu$ is the hypersurface in $Y$
  cut out by $g_\mu$. The key information that we need is the weight $\Lambda$ of
  the $\lambda$-action on $g_\mu$. We will determine this for generic
  $\mu$ and in fact we
  claim that generically $\Lambda = -M\max_Y\theta_w$.

  For simplicity let us assume that the (relative) weights
  $w_i/b_i$ are ordered so that $w_1/b_1 \leq w_2/b_2 \leq\ldots \leq
  w_B/b_B$. Let $c$ be the smallest index so that $u_c$ does not
  vanish on $Y$. Suppose
  that $\mu$ is chosen to be in general position in the sense that we
  cannot write $f_\mu = h + f'_\mu$, where $h\in I$, and $f'_\mu$ has strictly
  larger weights than $u_c$. If we write $u_{c'}, u_{c'+1}, \ldots, u_B$
  for the monomials with strictly larger weight (we have
  $c' \geq c+1$, but this inequality may be strict), then the
  condition is that $f_\mu \not\in I + (u_{c'}, \ldots, u_B)$, or in other
  words that $f_\mu$ does not vanish on the intersection of $X$ with the
  subspace $H = \{u_{c'}, \ldots, u_B = 0\}$. If this intersection were just
  the origin, then we would have $u_c \in I +(u_{c'}, \ldots, u_B)$,
  so that $u_c \in \mathrm{in}_\lambda(I)$. But then $u_c$ vanishes on
  $Y$ contrary to our assumption. This means that the intersection
  $X\cap H$ is non-trivial, and so the space of $\mu$ for which
  $f_\mu$ vanishes on $X\cap H$ is contained in a hyperplane in
  $\mathbf{P}$. We can then enlarge the subspace $E$ above by this
  hyperplane, and focus on $\mu\in \mathbf{P}\setminus E$. 

  By assumption $u_1, \ldots, u_{c-1}\in \mathrm{in}_\lambda(I)$, and so
  $I$ must contain $u_i^{M/b_i}$ modulo higher weight terms, for
  $i=1,\ldots,c-1$. It follows that $I$ contains
  \[ \sum_{i=1}^{c-1} \mu_i u_i^{M/b_i} \]
  modulo higher weight terms, but it cannot contain
  \[ \sum_{i=1}^{c'-1} \mu_i u_i^{M/b_i}, \]
  modulo higher weight terms, by our genericity assumption. It follows
  that $\mathrm{in}_\lambda(I +
  (f_\mu)) = \mathrm{in}_\lambda(I + (f_\mu'))$, where
  \[ f_\mu' = \sum_{i=c}^B \mu_i' u_i^{M/b_i}, \]
  for some new coefficients $\mu_i'$ with at least one of $\mu_c',
  \mu_{c+1}',\ldots, \mu_{c'-1}' \ne 0$. We can then
  let $g_\mu$ be the part of $f_\mu'$ which has the same weight as
  $u_c^{M/b_c}$ under $\lambda$, i.e. $\Lambda = Mw_c/b_c$. 

  Since the transverse Hamiltonian is 
  \[ \theta_w = \frac{\sum_{i=1}^B -\frac{w_i}{b_i} |u_i|^{2M/b_i}}{
    \sum_{i=1}^B |u_i|^{2M/b_i}}, \]
  and $u_1,\ldots, u_{c-1}$ vanish on $Y$, but $u_c$ does not, we have
  $\frac{Mw_c}{b_c}= -M \max_Y\theta_w$, and so the weight of $g_\mu$
  is $\Lambda = -M\max_Y\theta_w$ as we claimed. 

  Let us write $Y_\mu = Y \cap g_\mu^{-1}(0)$. By the discussion above
  we then have
  \begin{equation}\label{eq:betataueq}
 \beta^\tau = \frac{\pi}{M}\int_{\mathbf{P}\setminus E} [Y_\mu]\,
  d\mu, 
  \end{equation}
  where recall that now $E$ is a suitable union of two hyperplanes. 
  It follows that
  \begin{equation}\label{eq:intYtheta}
    \int_Y \theta_w e^{-\frac{1}{2}\hat{r}^2} \beta^\tau\wedge
  \omega^{n-1} = \frac{\pi}{M}\int_{\mathbf{P}\setminus E}
  \int_{Y_\mu} 
  \theta_w e^{-\frac{1}{2}\hat{r}^2}
  \omega^{n-1}.
  \end{equation}
  From Propositions~\ref{prop:volumeformula} and \ref{prop:volumevariation} we have
  \begin{equation}\label{eq:intYtheta2}
\int_{Y_\mu} \theta_w e^{-\frac{1}{2}\hat{r}^2}\,\omega^{n-1} = 
   (2\pi)^{n-1} \Big[(n-2)!\Big]^2 D_w a_0(Y_\mu, \xi), 
\end{equation}
  where as in \eqref{eq:Dwnotation}, the notation $D_w$ refers to
  varying the Reeb field $\xi$ in the direction of $w$. We have that
  $\xi$ acts on
  $g_\mu$ with weight $M$, while $w$ acts with weight
  $-M\max_Y\theta_w$. The index characters of $Y$ and $Y_\mu$ are
  related by
  \[ F_{Y_\mu}(t,\xi+sw) = F_{Y}(t, \xi+sw) \Big[ 1 -
  e^{-tM(1-s\max_Y\theta_w)}\Big], \]
  and so, differentiating with respect to $s$, at $s=0$, we obtain
  \[ D_w F_{Y_\mu}(t,\xi) = D_w F_Y(t, \xi) \Big[1 - e^{-tM}\Big] -
  F_Y(t, \xi) t M\max_Y\theta_w. \]
  Comparing the leading terms in the Laurent expansion, we have
  \[ (n-2)! D_w a_0(Y_\mu, \xi) = M (n-1)! D_w a_0(Y, \xi) - M
  \max_Y \theta_w\, (n-1)! a_0(Y, \xi). \]
  It then follows from \eqref{eq:intYtheta}, \eqref{eq:intYtheta2},
  and Proposition~\ref{prop:volumevariation} that
  \[ \int_Y \theta_w e^{-\frac{1}{2}\hat r^2} \beta^\tau \wedge
  \frac{\omega^{n-1}}{(n-1)!} = \frac{1}{2(n-1)} \int_Y \Big[n \theta_w - \max_Y
  \theta_w \Big]\, \frac{\omega^n}{n!}. \]
  At the same time we have
  \[ \omega^\tau \wedge \omega^{n-1} = \frac{1}{\hat r^2}
  \frac{n-1}{n} \omega^n, \]
  and so
  \[ \int_Y \theta_w e^{-\frac{1}{2}\hat r^2} \omega^\tau \wedge
  \frac{\omega^{n-1}}{(n-1)!} = \frac{1}{2} \int_Y \theta_w e^{-\frac{1}{2} \hat
    r^2} \frac{\omega^n}{n!}. \]
  Combining these formulas we obtain the required
  result.
\end{proof}

\begin{rk} \label{rem:maxthetaweight}
  We remark that $\max_Y\theta_w$ and the integral of $\theta_w$ on
  $Y$ depend only on the induced action on $Y$ and can be computed
  from the weights of the action. Indeed, in the argument above we
  have seen that $-\max_Y\theta_w$ is the minimal relative weight
  $w_j/b_j$  of a monomial
  $u_j$ that does not vanish on $Y$, under the action $\lambda$
  relative to the action of the Reeb field.  This is the same as the minimal
  relative weight $w_i/a_i$ of a coordinate function $z_i$ which does
  not vanish on $Y$, and since the $z_i$ generate the ring of algebraic functions
  of $Y$ which vanish at the vertex,
  this is simply the minimum relative weight of any such function on
  $Y$. More precisely, if $f$ is any algebraic function on $Y$
  vanishing at the vertex which is in a weight space of the torus
  spanned by $\lambda$ and $\xi$, then the relative weight is the
  quotient of the weights of these two $\mathbf{C}^*$-actions, and
  $-\max_Y\theta_w$ is the minimum of this quotient over all such
  $f$.  At the same time the integral of $\theta_w$ on $Y$ can be
  interpreted as the variation of the $a_0$ coefficient in the Hilbert
  series of $Y$, by Propositions~\ref{prop:volumeformula} and
  \ref{prop:volumevariation}. 
\end{rk}

From the above proof we also see the following.
\begin{prop}\label{prop:generic_hyper}
  In the above setup, given $X$ and $\lambda$, there is a union of
  $2$ hyperplanes  $E\subset \mathbf{P}$, such that if $\mu\in
  \mathbf{P}\setminus E$ and $V_\mu = X \cap f_\mu^{-1}(0)$, then 
  \[ \mathrm{Fut}_{X, \xi, (1-t)\alpha^\tau}(w) = \mathrm{Fut}_{X, \xi,
   \frac{\pi}{M}[V_\mu]}(w). \]
\end{prop}

In other words, when we want to compute the twisted Futaki invariant,
we can replace $\alpha^\tau$ by a current of integration along a
suitable hypersurface on $X$, as long as this hypersurface, as a point
in a projective space $\mathbf{P}$, is not contained in a certain
union of $2$ hyperplanes. It is important to emphasize that these
$2$ hyperplanes can depend on the choice of $X$ and $\lambda$.

\section{The partial $C^0$ estimate}\label{sec:partialC0}
Our goal in this section is to prove the partial $C^0$-estimate, Theorem~\ref{thm:partialc0}
below, for cone
metrics satisfying a Ricci curvature equation.  A special case of this will 
be the partial $C^0$ estimate along the continuity method. The partial
$C^0$-estimate was introduced by Tian~\cite{TianSurf} in his study
of K\"ahler-Einstein metrics on complex surfaces, and he conjectured a
general version for compact K\"ahler manifolds with a positive lower
bound on the Ricci curvature. For K\"ahler-Einstein metrics in
arbitrary dimension this estimate was obtained by
Donaldson-Sun~\cite{DS}, using the Cheeger-Colding convergence
theory~\cite{CheeCo} under Ricci curvature bounds, together with the
H\"ormander technique~\cite{Hormander} for constructing holomorphic
functions. Many more general results followed this development
(see \cite{CDS2,CDS3,PSS,Sz,ChenWang,Jiang}). 

Our method
will be fairly close to that in \cite{Sz} for the smooth continuity
method. The main difficulty is that along our continuity method the
Ricci curvature of the cone metrics is only non-negative, while in the
approach of \cite{Sz} it is important to treat the Ricci form as a
metric. On the other hand it is not clear how to extend the
Cheeger-Colding theory to the transverse K\"ahler structure, which
does have strictly positive Ricci curvature. We therefore use the
convergence theory on the level of the cones, but at certain crucial
steps we invoke the positivity of the Ricci curvature of the
transverse metric. 

We suppose that
$(X,\xi)$ is a normalized Fano cone singularity, 
and $\alpha$ is a smooth K\"ahler cone metric on
$(X,\xi)$. From the discussion in Section~\ref{sec:background} we know
that we have a family of metrics $\omega_t$ on $(X, t^{-1}\xi)$,
satisfying
\[ \mathrm{Ric}(\omega_t) = 2n\frac{1-t}{t} \alpha^\tau, \]
for $t\in [t_0, T)$, with $T \leq 1$ and $t_0 > 0$. 

Since we will also have to obtain uniform estimates while varying the
Reeb field, we suppose more generally that we have a family of Reeb fields $\xi_t$, and
metrics $\alpha_t, \omega_t$ on $(X,\xi_t)$ satisfying 
\begin{equation}\label{eq:genMOC}
{\rm Ric}(\omega_t) = 2nc_t \alpha_{t}^{\tau}
\end{equation}
where $0\leq c_t \leq c_0 < \infty$, and the pairs $(\xi_t, \alpha_t)$
move in a bounded family in the following sense. 
\begin{defn}
We say that data $(\xi_t, \alpha_t)$ consisting of Reeb vector fields
and a compatible cone metrics on $X$ are in
a $C^2$ bounded family if the $\xi_t$ are in a compact subset of
$\mathcal{C}_{R}$, and the metrics $\alpha_t$ are locally uniformly equivalent
to a fixed reference cone metric, and locally bounded in $C^2$ when measured
with respect to this reference metric.  
\end{defn}

Along our continuity method we will have uniform constants $\kappa, C_L$ so that
\begin{itemize}
\item $(X,\omega_t)$ are uniformly non-collapsed.  That is, ${\rm Vol}(B_1(0, \omega_t)) >\kappa >0$, where $0 \in X$ denotes the cone point.
\item $\Ric(\omega_t)\geq 0$ on $X$, and the corresponding Sasakian
  metric $g_t$ on the link $L$ satisfies
  $\Ric(g_t) = (2n-2)g_t + (2n-2)c_t\alpha_t^\tau$, 
\item ${\rm diam}(L, g_t) <C_L$ for some controlled constant $C_L$.  
\end{itemize}
The first point follows from the lower bound for $t_0$. 
The second point is the formula relating the Ricci curvature on the
cone to the Ricci curvature on $L$, while the last point follows from Myers' Theorem.
Let $r_t$ denote the radial function of $\omega_t$, then  the Bishop-Gromov comparison theorem
implies the metrics $\omega_t$ are uniformly non-collapsed on the annuli $\{1/2 <r_t<2\}$ and so
results of Croke \cite{Croke} and Yau (see, e.g. \cite[page 9]{Yau3}) imply

\begin{lem}
For metrics $\omega_t$ satisfying~\eqref{eq:genMOC} there is a uniform Sobolev inequality on the set $Ann:= \{1/2 <r_t <2\}$.  That is, there exists a constant $C(\kappa, C_L)$ independent of $t$ so that for any $W^{1,2}$ function $f$ on $Ann$ we have
\[
C^{-1}\left(\int_{Ann}|f|^{2\frac{n}{n-1}}\omega_{t}^{n}\right)^{\frac{n-1}{2n}} \leq  \left(\int_{Ann} |\nabla f|_{\omega_t}^{2}\omega_{t}^{n}\right)^{1/2} + \left(\int_{Ann} |f|^{2} \omega_t^n\right)^{1/2}
\]
\end{lem}

Our eventual
estimates will depend only on the dimension, the non-collapsing
condition, and a bound on the geometry of $\alpha$, which roughly
speaking says that we have good control of the transverse metric
$\alpha^\tau$ on sufficiently small balls (see
Definition~\ref{defn:geombound} for a precise statement). 
As motivation consider the following analogous property of a compact
K\"ahler manifold, which is easily proven by covering the manifold
with sufficiently small coordinate balls.

\begin{lem}\label{lem: Kahler ball lem}
Let $(M,\omega)$ be a compact K\"ahler manifold, and let $g$ be the associated K\"ahler metric.  Then for any constant $K>0$ sufficiently large the following holds: if $B \subset M$ is any $g$-ball of radius smaller than $K^{-1}$, then there exist holomorphic coordinates $\{z_{1}, \ldots, z_n\}$ defined on $B$ such that
\[
\frac{1}{2} \delta_{i\bar{j}} < g_{i\bar{j}} < 2 \delta_{i\bar{j}},
\,\, \text{ and } \quad \|g_{i\bar{j}}\|_{C^{2}(g_{Euc})} < K.
\]
Furthermore, $K$ can be chosen to be uniform over $C^{2}$ bounded families of metrics.
\end{lem}

We need a generalization of this to the transverse K\"ahler structure
defined by $\alpha^\tau$. Let us denote by $Q$ the quotient bundle $Q=
TX / \mathbf{C}\xi$, where by $\mathbf{C}\xi \subset TX$ we denote the
complex subbundle spanned by $\xi$.
Note that $Q$ has a natural integrable complex structure,
and $\alpha^\tau$ defines a K\"ahler form on it (see Boyer-Galicki~\cite{BGbook}). 

\begin{defn}  Let $B \subset \mathbf{C}^{n-1}$ be a ball.  We say that
  an immersion $F:B \rightarrow X$ is a $\xi$-transverse immersion
  if for any point $p \in B$, the image $dF_p(T_pB)$ is transverse to
  $\mathbf{C}\xi$.  In particular, we get an induced vector space isomorphism
\[
dF_{p} :T_pB \rightarrow Q_{F(p)}. 
\]
 We say that $F$ is a transverse holomorphic immersion if $dF_p$ is
 complex linear with respect to the standard complex structure on
 $B$, and the transverse complex structure on $Q$. 
\end{defn}

For any transverse holomorphic immersion $F$, we obtain a K\"ahler metric
$F^*\alpha^\tau$ on $B$, compatible with the standard structure on B.
Note that we have a $\mathbf{C}$-action on $X$ induced by $\xi$,
and for any smooth $h : B\to \mathbf{C}$, the
maps $F$ and  $h\cdot F$ induce the same K\"ahler structure on $B$. In
particular we can assume that $F: B\to L$ maps into the unit link
$L\subset X$.

\begin{defn}\label{defn:geombound}
  We say that $\alpha$ (or equivalently $\alpha^\tau$) has geometry
  bounded by $K$ if the following holds. 
Let $F: B\to L$ be a transversal holomorphic immersion as above, where
  $B\subset\mathbf{C}^{n-1}$ denotes a ball. 
  We write $g = F^*\alpha^{\tau}$ for the induced K\"ahler metric on $B$. If
  $(B,g)$ has diameter at most $K^{-1}$, then
  there are holomorphic functions $z_1, \ldots, z_{n-1}$ on $B$ (not
  necessarily giving an embedding of $B$ into $\mathbf{C}^{n-1}$, but an
  immersion), such that
  \[ \frac{1}{2}\delta_{ij} < g_{i\bar j} < 2\delta_{ij}, \qquad \Vert g_{i\bar j}\Vert_{C^2(B,g_{Euc})} < K. \]
Here $g_{i\bar j} = g(\d/\d
  z^i, \d/\d\bar{z}^j)$, and $g_{Euc} =\delta_{ij}$ is the pull-back of the Euclidean metric
  by the map $(z_1,\ldots,z_{n-1}): B \to \mathbf{C}^{n-1}$ .
\end{defn}
We now prove the analog of Lemma~\ref{lem: Kahler ball lem}. 
\begin{prop}\label{prop:aTcoords}
  For $K > 0$ sufficiently large, depending on $\alpha$,
  the geometry of $\alpha$ is bounded by $K$.  Moreover, $K$
  can be chosen uniformly over bounded families. 
\end{prop}
\begin{proof}
We consider the case of a fixed metric $\alpha$. We first
 cover $L$ by a finite number of adapted charts
  $V_i$. This means that on such a $V$ we have coordinates
  \[ (x, z_1, \ldots, z_{n-1}) : V \to \mathbf{R} \times \mathbf{C}^{n-1}, \]
  in which the Reeb field $\xi$ is given by $\d / \d x$ and the
  Sasakian metric $\alpha$ agrees with the Euclidean metric at the
  origin. If we denote by $\alpha_0$ the metric on the slice
  $\{x=0\}$, then $\alpha_0 = \alpha^{\tau}$ is a K\"ahler metric which in the
  coordinates $z_i$ agrees with the Euclidean metric at the origin. 
  By using a cover by smaller charts if
  necessary, we can assume that in these coordinates $\alpha_0$
  satisfies $\frac{1}{2} \delta_{ij} < \alpha_0 < 2\delta_{ij}$ and
  $\Vert \alpha_0\Vert_{C^2} < K$, for some $K$ (independent of the
  chart). Increasing $K$ if necessary, we can ensure that
  every $\alpha$-ball of radius $K^{-1}$ in $L$ is contained in one of our
  adapted charts $V_i$. 

  Fixing again one of our charts $V$, suppose that we have a transverse
  holomorphic  immersion $f : B\to V$, transverse to $\d /
  \d x$. Then the induced K\"ahler structure on $B$ is simply the
  pullback $(\pi\circ f)^*\alpha_0$, where $\pi$ is the projection in
  $V$ onto the $\{x=0\}$ slice. The holomorphic functions $z_i\circ
  \pi\circ f$ then satisfy our requirements. 

  To prove the proposition it would suffice to show that if $F : B\to
  L$ is any transversal holomorphic immersion such that $(B, F^{*}\alpha^{\tau})$
  has diameter smaller than $K^{-1}$, then it is contained in
  one of the $V_i$.  This is clearly impossible, since given such
  an immersion one could easily stretch the immersion by the Reeb flow
  to obtain a new immersion which is not contained in a ball of radius $K^{-1}$.
  Instead, given such an immersion $F$, we will construct a new, 
  ``equivalent'' immersion $f : B\to L$ whose image lies in one of
  our adapted charts $V$.  Since the Reeb vector field is real
  holomorphic, we can flow our transverse holomorphic immersion $F$ to a new
  transverse holomorphic immersion along
  the Reeb field.  Writing $\Phi : L \times \mathbf{R} \to L$ for the Reeb
  action, we are looking for a smooth function $a: B\to\mathbf{R}$
  such that the image of
  \[
  \begin{aligned}
   f : B &\to L \\
      p &\mapsto \Phi(F(p), a(p))
      \end{aligned}
       \]
  lies in one of our adapted charts. We can choose the function $a$,
  so that $f$ maps radial rays $\gamma$ from the origin in $B$ to
  curves $f(\gamma)$ in $L$
  that are orthogonal to $\xi$. The length of $f(\gamma)$ with respect
  to $\alpha$ is then equal to its transversal length--that is, its length
  in $(B,g)$. By assumption the diameter of $(B,g)$ is at most
  $K^{-1}$, and so the image $f(B)$ must be contained in an
  $\alpha$-ball of radius $K^{-1}$, and so it is contained in one of
  our adapted charts. This completes the proof of bounded geometry of
  a fixed metric. Moreover it is clear from the above argument that we
  can choose a uniform $K$ for metrics in a bounded family.
\end{proof}

We now state
the main result that we will prove in this section. Recall that we write $R_\chi(X)$ for the part of the ring of functions
of $X$ on which the torus $\mathbf{T}$ acts by the character $\chi$. For
any $D > 0$ let us write
\[ R_{<D}(X) = \bigoplus_{0 < \chi(\xi) < D} R_\chi(X). \]
Suppose that we have a sequence of solutions $\omega_k$ on $(X,\xi_k)$
of
\[ \mathrm{Ric}(\omega_k) = 2nc_k \alpha_k^\tau \]
for $c_k\in [0,c_0]$ for some $c_0 > 0$.
Choosing $L^2$ orthonormal bases of $R_{<D}$ with respect to $\omega_k$ we
obtain a sequence of maps $F_k : X \to \mathbf{C}^N$. 

\begin{thm}\label{thm:partialc0}
  There exists a constant $D$, depending on the dimension, the
  non-collapsing constant and the bound $K$ on the geometry of
  $\alpha$, such that each $F_k$ is an embedding, and up to choosing a
  subsequence we have $F_k(X) \to Y$ in the sense of currents, 
  where $Y$ is a normal,
  $\mathbf{Q}$-Gorenstein variety with a Reeb field $\xi = \lim
  \xi_k$. In addition $(F_k)_*(\alpha_k^\tau) \to
  \beta^\tau$ for a positive transverse current on $Y$, and $(Y,
  T^{-1}\xi, (1-T)\beta^\tau)$ admits a weak solution of the twisted
  equation, where $t_k\to T$.
\end{thm}

We will spend the rest of this section proving this result, based on
work of Donaldson-Sun~\cite{DS}, Chen-Donaldson-Sun~\cite{CDS2, CDS3}
as well as the second author~\cite{Sz}.  A key ingredient in the work
of Chen-Donaldson-Sun is to make use of the H\"ormander technique for
producing holomorphic sections of positive line bundles.  In our
setting we will use the H\"ormander technique to produce holomorphic
functions on our affine varieties.  The following estimate holds on
non-compact manifolds (see Demailly~\cite[Theorem 4.1]{Dem82} or \cite[Theorem 6.2]{Ber}).

\begin{thm}\label{thm: Hormander}
Let $L$ be a holomorphic line bundle endowed with a metric $e^{-\phi}$ over a complex manifold $X$ which has some complete K\"ahler metric.  Assume the metric $e^{-\phi}$ has strictly positive curvature, and that
\[
\ddb \phi \geq c\omega
\]
where $\omega$ is some K\"ahler form on $X$ (not necessarily complete)
and $c > 0$.

Let $f$ be a $\dbar$-closed $(n,q)$ form $(q>0)$ with values in $L$.  Then there is an $(n,q-1)$ form $u$ with values in $L$ such that $\dbar u =f$, and
\[
\| u\|_{L^{2}(X , e^{-\phi}, \omega)}^{2} \leq \frac{1}{cq} \| f\|_{L^{2}(X , e^{-\phi}, \omega)}^{2},
\]
provided the right hand side is finite.
\end{thm}

Using a resolution of singularities (see Saper~\cite[Example
9.4]{Sap}) we know that $X\backslash \{0\}$ admits a complete K\"ahler metric since $0 \in X$ is an isolated singular point. Hence Theorem~\ref{thm: Hormander} applies in our setting.

We are going to apply the H\"ormander theorem to $L=\mathcal{O}_{X}
\otimes K_{X}^{-1} \simeq K_{X}^{-1}$, where we recall that
$\mathcal{O}_{X}$ is endowed with the metric $e^{-\frac{1}{2}r^2}$ and
we use the corresponding volume form $\omega^{n}$ as a metric on
$K_{X}^{-1}$.  The reason we make this choice for $L$ is that we have isomorphisms
\[
\begin{aligned}
\Lambda^{n,1}(K_X^{-1}) &\cong \Lambda^{0,1} \otimes K_{X} \otimes K_{X}^{-1} \cong \Lambda^{0,1}\\
K_{X}^{-1}\otimes K_{X} &\cong \mathcal{O}_{X}.
\end{aligned}
\]
In particular, Theorem~\ref{thm: Hormander} implies that we can solve
the $\dbar$ equation on $\mathcal{O}_{X}$.  Following the ideas of
Donaldson-Sun~\cite{DS}, we can then use the H\"ormander technique to
transplant holomorphic functions from tangent cones of
Gromov-Hausdorff limits to our non-compact cone manifold $(X,\omega)$.
In order to do this, we need to ensure that there exist iterated
tangent cones which are ``good".

\begin{defn} 
Suppose that $(Z,d_Z)$ is a Gromov-Hausdorff limit of $(X, \omega_t)$
as $t \rightarrow T \leq 1$, and suppose that $C(Y)$ is an (iterated)
tangent cone at $p \in Z$.  We say that the tangent cone is  {\em good} if:
\begin{enumerate}
\item The regular set $Y_{reg} \subset Y$ is open in $Y$ and smooth.
\item The distance function on $C(Y_{reg})$ is induced by a Ricci flat
  cone metric, and on $C(Y_{reg})$ the scaled up metrics along our
  sequence  converge in $L^p_{loc}$ for all $p$ to this Ricci flat metric. 
\item For all $\delta>0$ there is a Lipschitz function $g$ defined on $Y$ which is identically $1$ on a neighborhood of $Y_{sing} = Y \backslash Y_{reg}$, with support contained in the $\delta$ neighborhood of $Y_{sing}$ and with $\|\nabla g\|_{L^{2}} \leq \delta$, where the $L^2$ norm is with respect to the Sasaki-Einstein metric on $Y_{reg}$.
\end{enumerate}
\end{defn}

Suppose that $\omega_t$ are K\"ahler cone metrics on $X$, solving~\eqref{eq:genMOC} where $(\xi_t, \alpha_t)$ 
move in a bounded family.  Suppose that, along a subsequence
$(X,\omega_{t})$ converge in the Gromov-Hausdorff sense to $(Z,d_Z)$.
If we can show that each tangent cone of
$(Z,d_Z)$ is good, then the techniques of \cite{DS}, together with the
above remarks, will imply that there is a number $\epsilon_{0}$ depending only on the dimension,
the non-collapsing constant and a bound for the geometry of $\alpha_t$ with
the following effect:  Let $r_t$ be the radial function for $\omega_t$.
For any point $p \in L = \{r_{t}=1\}$ there is a holomorphic
function $f  \in \mathcal{O}_{X}$ with
$\|f\|_{L^{2}(e^{-\frac{1}{2}r_{t}^2})}=1$, and $|f(p)|>\epsilon_0$.

At this point we will need to pass from arbitrary holomorphic functions to
those with polynomial growth.  This is done in
section~\ref{sec:polygrowth}, essentially by truncating the Taylor
series of $f$ at a sufficiently high (but controlled) order.  Putting
all of these results together with techniques from
Donaldson-Sun~\cite{DS2}
will imply Theorem~\ref{thm:partialc0}. With this discussion, we state
our first goal:

\begin{prop}\label{prop:goodTang}
Suppose $\omega_{t}$ are solutions of~\eqref{eq:genMOC} with data $(\alpha_t, \xi_t)$ which moves
in a bounded family.  If $(Z,d_Z)$ is any Gromov-Hausdorff limit of a sequence $(X,\omega_{t_i})$, then
$Z$ has good tangent cones.
\end{prop}

\subsection{Gromov-Hausdorff Convergence}
We will now specialize to sequences $(X,\omega_i, \alpha_i,\xi_i)$, where
\[
(\xi_i, \alpha_i) \longrightarrow (\xi, \alpha)
\]
in the $C^2$ topology for a fixed background metric, and
\begin{equation}\label{eq:genMOC2}
{\rm Ric}(\omega_i) = 2n c_i \alpha^{\tau}_i
\end{equation}
with $c_{i} \longrightarrow c$.

Suppose we have a sequence of metrics solving~\eqref{eq:genMOC2}, with $c_i \rightarrow c$.  Since the links $(L,g_{i})$ have bounded diameter, positive Ricci curvature, and are uniformly non-collapsed, we can take a Gromov-Hausdorff limit
\[
(L,g_{i}) \xrightarrow{\,\,\,\,d_{GH}\,\,\,\,}(Z,d).
\]
At the same time we will have convergence of the cones
\[
(X, \omega_i)\xrightarrow{\,\,\,\,d_{GH}\,\,\,\,}(C(Z),\hat{d}),
\]
in the pointed Gromov-Hausdorff topology, where we can identify $Z$
with the unit link in $C(Z)$. 

To understand iterated tangent cones in the space $C(Z)$, for any $p\in Z$
we need to study very small balls 
centered around $p \in C(Z)$, scaled to unit size. This in turn means
that we need to study small balls centered at points on the unit link
in $(C(L), \omega_{i})$, scaled to unit size. 
Such a ball $B$ has
the following structure: $B$ is the unit ball with respect to a
K\"ahler metric $\omega$, satisfying the equation 
\begin{equation}\label{eq:Ricomega}
\mathrm{Ric}(\omega) = c\alpha^\tau, 
\end{equation}
and a uniform non-collapsing condition $\mathrm{Vol}(B,\omega) > K^{-1} >
0$. 
There is a holomorphic vector field $\Xi$ on $B$, whose imaginary part
is the Reeb field $\xi$
scaled down, satisfying 
\[ 1-\delta < |\Xi|_{\omega} < 1 + \delta, \quad \mathcal{L}_\Xi \omega =
\lambda\omega \]
for some $\lambda \leq \delta \leq 1/2$. If the ball that we scaled up is
sufficiently small, then $\delta$ can be taken to be arbitrarily
small.  Finally $\alpha^{\tau}$ is a closed,
non-negative $(1,1)$-form, vanishing along $\Xi$, and defining a transverse K\"ahler
metric with bounded geometry on $TB / \mathbf{C}\xi$ in the sense of
Definition~\ref{defn:geombound}.

There are two different cases to study, depending on whether $c_i$ is bounded
away from 0 or $c_i \to 0$. 

\subsection{The case $c_i$ are bounded away from zero} \label{sec:cinot0}

Fix $i$, and suppress the index.  Scaling $\alpha^\tau$ by a bounded factor we can rewrite
Equation~\eqref{eq:Ricomega} as
\[ \mathrm{Ric}(\omega) = \alpha^\tau. \] 
The next proposition, which is based on \cite[Proposition 8]{Sz}
shows that when $(B,\omega)$ is close to the Euclidean ball in the
Gromov-Hausdorff sense, then on a smaller ball the Ricci curvature is
bounded. The quantity $I(B)$, as defined in \cite{CDS2, Sz} is 
\[ I(B) = \inf_{B(x,r)\subset B} VR(x,r), \]
where $VR(x,r)$ is the ratio of the volumes of the ball $B(x,r)$ and
the Euclidean ball $rB^{2n}$. 

\begin{prop}\label{prop:RegRicBnd}
  There is a $\delta = \delta(K) > 0$, depending on the bound $K$
  for the geometry of $\alpha$, such that if $1 -
  I(B) < \delta$, then $\mathrm{Ric}(\omega) < 5\omega$ on
  $\frac{1}{2}B$. 
\end{prop}
\begin{proof}
  The difference with \cite[Proposition 8]{Sz} is that $\alpha^\tau$ is
  not strictly positive, however it is strictly positive on slices
  transverse to $\Xi$, and it is invariant under the flow of $\Xi$,
  since this flow simply scales $\omega$. 

  As in \cite[Proposition
  8]{Sz}, if $\mathrm{Ric}(\omega)$ is not bounded by $5\omega$ on
  $0.5B$, then we can find a small ball inside $B$, which when scaled
  to unit size $(\tilde{B}, \tilde{\omega})$ satisfies $\alpha^\tau\leq
  \tilde{\omega}$, and in addition there is a unit vector $v$ at the
  origin (with respect to $\tilde{\omega}$) such that
  \[ \alpha^\tau(v,\bar v) \geq 1. \]
  The equation for $\omega$ implies that on $\tilde{B}$ the metric
  $\tilde{\omega}$ has bounded Ricci curvature, and so if $\delta$ is
  sufficiently small, then Anderson's result \cite{And} implies that
  we have holomorphic coordinates $z^1, \ldots, z^n$ on the ball
  $\theta \tilde{B}$, with respect to which $\tilde{\omega}$ is close
  to the Euclidean metric in $C^{1,\alpha}$. In these coordinates the
  holomorphic vector field $\Xi$ will satisfy $1/4 < |\Xi|_{Euc} < 4$,
  and so by rotating the coordinates and shrinking $\theta$ we can
  assume that on $\theta \tilde{B}$ the vector field $\Xi$ is very
  close to $\partial_{z^1}$. In particular $\alpha(\partial_{z^1},
  \bar w)$ is very small for any unit vector $w$. 

  It follows that the slice $U = \{z^1 = 0\} \cap \theta\tilde{B}$ is
  transverse to $\Xi$, and so $(U, \alpha^\tau)$ is a K\"ahler manifold
  with bounded geometry. The inequality $\alpha^\tau \leq \tilde{\omega}$
  implies that the diameter of $(U, \alpha^\tau)$ is at most $\theta$, so
  shrinking $\theta$ further if necessary, we have holomorphic
  functions $w^2, \ldots, w^n$ on $U$, defining local coordinates near
  each point, in which the components of $\alpha^\tau$ are controlled in
  $C^2$. 

  The vector $v$ may not be tangent to the slice $U$, but we may
  simply discard its $\partial_{z^1}$-component, while still having
  $\alpha^\tau(v,\bar v) > \frac{1}{2}$. Rotating the $z^2, \ldots, z^n$
  coordinates, we can then assume that 
  \[ \alpha^\tau(\partial_{z^2}, \partial_{\bar z^2}) > \frac{1}{4}. \]
  We now have that the components
  of $\alpha^\tau$ in the $z^i$ coordinates 
  have bounded derivatives along the slice $U$, but in
  addition $\alpha^\tau$ is also constant along the flow of $\Xi$, which
  is very close to $\partial_{z^0}$. It follows that just as in
  \cite[Proposition 8]{Sz} we can obtain a spherical sector in which
  the Ricci curvature of $\tilde{\omega}$ is strictly positive. Applying
  the Bishop-Gromov volume comparison we get a contradiction to
  $1-I(B) < \delta$ if $\delta$ is sufficiently small.  
\end{proof}
         
\begin{cor}
If we have solutions of $(X,g_i)$ of ~\eqref{eq:genMOC2} with $c_{i} \geq c >0$, and if $(B(p_i,1) , g_i) \subset X$
converge in the Gromov-Hausdorff sense to  the Euclidean ball, then
the convergence is $C^{1,\alpha}$ on compact sets.  In particular, if
$(B(p_i,1) , g_i)\longrightarrow Z$, then the regular set in $Z$ is
open, and then convergence on the regular set is locally
$C^{1, \alpha}$.
\end{cor}
\begin{proof}
We combine Colding's volume convergence \cite{ColVol} with
Proposition~\ref{prop:RegRicBnd} to get a uniform Ricci bound, and
then apply Anderson's result in \cite{And} to get $C^{1,\alpha}$ 
convergence to a Euclidean ball.  
\end{proof}

Now assume that we have a sequence of balls as above such that
$(B(p_i, 1), g_i) \longrightarrow Z$, with $p_i
\rightarrow p$, and a tangent cone at $p\in Z$ is of the form
$\mathbf{C}_{\gamma} \times \mathbf{C}^{n-1}$. As in \cite[Proposition
11]{Sz} we have $\gamma\in(\gamma_1,\gamma_2)$ for some $0 < \gamma_1
< \gamma_2 < 1$. 
The results of
\cite{CDS2}  apply, and in particular, arguing as in \cite[Section
2.5]{CDS2}, after scaling up the $\omega_i$:
\[
\tilde{\omega}_{i} = k\omega_i,\qquad \tilde{\Xi}_{i} = \frac{1}{\sqrt{k}} \Xi_i
\]
we can view $\tilde{\omega}_i$ as a metric on the unit Euclidean ball
$B^{2n}$ with coordinates $(u, v_1, \ldots,v_{n-1})$, in which
$\tilde{\omega}_i$ is close to the model cone metric
\[
\eta_{\gamma} = \sqrt{-1} \frac{du\wedge d\overline{u}}{|u|^{2-2\gamma}} + \sum_{i=1}^{n-1} dv_i \wedge d\overline{v}_i.
\]
More precisely, if we scale by a large integer $k$, and take $i$ large
depending on $k$, we have, for some fixed constant $C$:
\begin{itemize}
\item $\tilde{\omega}_{i} = \ddb \phi_i$ with $0 \leq \phi_i \leq C$
\item $\omega_{Euc} < C\tilde{\omega}_{i}$,
\item Given $\delta >0$ and a compact set $K \subset B^{2n}\backslash \{u=0\}$ we can suppose (by taking $i$ large once $k$ is taken sufficiently large) that\\ $|\tilde{\omega}_{i} - \eta_{\gamma}|_{C^{1,\alpha}(K, g_{Euc})} < \delta$.
\end{itemize}

\begin{lem}
In the above setting, for every $\epsilon\in (0,1/2)$ we have
$|\tilde{\Xi}_i|_{g_{Euc}} >\frac{1}{2}$ on the set $\{|u|=\epsilon\}$
provided $i$, and the scaling factor $k$ are sufficiently large.
\end{lem}
\begin{proof}
We fix $\epsilon >0$, and suppose that the conclusion is false. Then
we have a sequence of metrics $\wt \omega_i$ and holomorphic vector fields $\wt \Xi_i$ on $B^{2n}$ (with scaling
factors $k\to\infty$) converging
in $C^{1,\alpha}$ to the standard cone metric $\eta_\gamma$ locally
away from $\{u=0\}$, and $\wt \Xi_{i}$ satisfy
\[
|\wt \Xi_{i}|_{g_{Euc}} \leq C \quad 1-\frac{1}{k} < |\wt \Xi_{i}|_{\wt
  \omega_{i}} <  1 + \frac{1}{k}, \qquad L_{\wt \Xi_{i}}\wt
\omega_{i} = \lambda \wt \omega_i, 
\]
with $\lambda < 1/k$.  Since the $\wt \Xi_{i}$ are holomorphic and bounded we obtain uniform $C^{3,\alpha}$ estimates on $\frac{1}{2}B^{2n}$.  We can choose a subsequence so that $\wt \Xi_{i}$ converges to a holomorphic vector field $\wt \Xi$ in $C^{3}(g_{Euc})$ on $\frac{1}{2}B$, and $\wt\omega_{i} \rightarrow \eta_{\gamma}$ on $\{ |u| \geq \epsilon/2 \} \cap \frac{1}{2}B$.  Furthermore, we have
\[
L_{\wt \Xi}\eta_{\gamma} =0, \quad |\wt \Xi|_{\eta_{\gamma}}=1.
\]
By direct computation one verifies that the only holomorphic vector
fields on $\frac{1}{2}B$ which are Killing for $\eta_{\gamma}$ on $\{
|u| \geq \epsilon/2 \} \cap \frac{1}{2}B$ and with unit length are the
translations in the $v_i$ directions. In particular $|\wt \Xi|_{Euc} =
|\wt \Xi|_{\eta_\gamma}$. The result then follows from the convergence
$\wt\omega_i \to \eta_\gamma$. 
\end{proof}

\begin{prop}\label{prop:liminf>c0}
There is a constant $c_0>0$, such that
if  $(B(p,1),\omega)$ is sufficiently close to the unit ball in the
cone $\mathbf{C}_{\gamma} \times \mathbf{C}^{n-1}$ with $\gamma\in
(\gamma_1,\gamma_2)$, then 
\[
\int_{B(p,1)} \alpha^{\tau}\wedge \omega^{n-1} >c_0.
\]
\end{prop}
\begin{proof}
We argue by contradiction and assume
there is no such $c_0$. Then we will have a sequence $B(p_i,1) \to
B(\underline{0},1)$, where $\underline{0}$ is the vertex in the cone 
$\mathbf{C}_\gamma \times \mathbf{C}^{n-1}$, and such that
\begin{equation}\label{eq:liminf0}
\lim_{i\to \infty} \int_{B(p_i,1)} \alpha_i^\tau \wedge
\omega_i^{n-1} = 0. 
\end{equation}
As discussed above, we can then
 find a small $r_0 > 0$ such that for sufficiently large
$i$, the scaled up metric $r_0^{-2}\omega_i$ can be thought of as a
metric on a set containing the Euclidean unit ball
$B^{2n}$, such that $\omega_{Euc} < Cr_0^{-2}\omega_i$ on
$B^{2n}$. By the previous lemma we can assume (choosing $r_0$
smaller if necessary), that the rescaled Reeb
vector field $\Xi_i$ is a perturbation of the vector field $\partial
/ \partial z^0$ for one of the holomorphic coordinates $z_0$ on
$B^{2n}$, and so the sets $U_c = \{z_0 = c\}$ are transverse to
$\Xi_i$ for $|c| < 1/2$, say. 

 Suppose that
\[ \int_{B^{2n}} \alpha_i^\tau \wedge \omega_{Euc}^{n-1} <
\epsilon_1, \]
for some $\epsilon_1 > 0$.  Then we must have
\[ \int_{U_c} \alpha_i^\tau \wedge \omega_{Euc}^{n-2} < C_1\epsilon_1 \]
for at least one $c$ with $|c| < 1/2$, and a uniform constant $C_1$. The
form $\alpha_i^\tau$ defines a K\"ahler metric on $U_c$ with bounded
curvature, so the $\epsilon$-regularity theorem of
Schoen-Uhlenbeck~\cite{SchoenUhlenbeck} 
(see also \cite[Proposition 7]{Sz}) implies that if $C_1\epsilon_1 <
\epsilon_0$, then we must have $\alpha_i^\tau < C'\omega_{Euc} < C'Cr_0^{-2}\omega_i$ on
$U_c$ (evaluated on vectors tangent to this slice). 
Since $\Xi_i$ is a small perturbation of $\partial / \partial z^0$,
and $\alpha_i^\tau$ vanishes along $\Xi_i$ we obtain a bound for
$\alpha^\tau$ on $\frac{1}{2}B^{2n}$. This implies a uniform bound for
$\mathrm{Ric}(\omega_i)$ on this ball, independent of $i$. The result
of Cheeger-Colding-Tian~\cite{CCT} implies that no conical singularity can then
form, which is a contradiction. 

As a result, we must have
\[ \int_{B^{2n}} \alpha_i^\tau \wedge \omega_{Euc}^{n-1} \geqslant
\epsilon_1, \]
for sufficiently large $i$,
where $\epsilon_1 = C_1^{-1}\epsilon_0$. This in turn implies
\[ \int_{B^{2n}} \alpha_i^\tau\wedge (r_0^{-2}\omega_i)^{n-1} \geqslant
C^{-{n-1}}\epsilon_1, \]
contradicting \eqref{eq:liminf0}. 
\end{proof}

We also have the following, whose proof is the same as that of
\cite[Proposition 13]{Sz}.
\begin{prop}\label{prop:densityupper}
There is a constant $A >0$, such that
if  $(B(p,1),\omega)$ is sufficiently close to either the Euclidean
unit ball or the unit ball in the
cone $\mathbf{C}_{\gamma} \times \mathbf{C}^{n-1}$ with $\gamma\in
(\gamma_1,\gamma_2)$, then 
\[
\int_{B(p,\frac{1}{2})} \alpha^{\tau}\wedge \omega^{n-1} < A.
\]
\end{prop}

We can now show that the iterated tangent cones are good,
similarly to Chen-Donaldson-Sun~\cite{CDS2}, or \cite{Sz}. The
argument in the published version of \cite{Sz} was incomplete, but it
is corrected in the latest version on the arXiv, and that argument can
be used verbatim in our setting, using the estimates on the
``densities'' given by Propositions \ref{prop:liminf>c0} and
\ref{prop:densityupper}. 

\subsection{The case $c_i \rightarrow 0$}\label{sec:ci0}

In this case we study non-collapsed balls $B=B(p,1)$ with metrics
$\omega$ satisfying
\[ \mathrm{Ric}(\omega) = c_i\alpha^\tau, \]
where $c_i \rightarrow 0$. As above, the additional structure is the (rescaled) Reeb field $\Xi$,
which is a holomorphic vector field satisfying $\frac{1}{2} <
|\Xi|_\omega < 2$ and $L_\Xi \omega = \lambda\omega$, for some
$\lambda \leq 2$. Note that as we scale up the metric, we must scale
down the Reeb field, and so scale down $\lambda$. The form $\alpha^\tau$
defines a transverse K\"ahler metric on $TB / \mathbf{C}\xi$, with bounded
geometry as before. Once again, the difficulty when compared to
\cite{Sz} is that $\alpha^\tau$ is not strictly positive, and so the
$\epsilon$-regularity theorem for harmonic maps cannot be applied. Our
strategy, as above, is to find transverse slices. 

We first need a slight refinement of
\cite[Proposition 1]{CDS3}, which will allow us to control the Reeb field $\Xi$. 
First we recall some
definitions. For a subset $A$ in a $2n$-dimensional length space $P$,
and for $\eta < 1$, let $m(\eta, A)$ be the infimum of those $M$ for which
$A$ can be covered by $Mr^{2-2n}$ balls of radius $r$ for all $\eta
\leqslant r < 1$.

For $x\in B$ and $r,\delta > 0$ a holomorphic map $\Gamma : B(x,r) \to
\mathbf{C}^n$ is called an $(r,\delta)$-chart centered at $x$ if
\begin{itemize}
\item $\Gamma(x) =0$,
\item $\Gamma$ is a homeomorphism onto its image,
\item For all $x', x''\in B(x,r)$ we have
  $|d(x',x'') - d(\Gamma(x'), \Gamma(x''))| \leqslant \delta,$
\item For some fixed $p > 2n$, we have
  $ \Vert \Gamma_*(\omega) - \omega_{Euc}\Vert_{L^p} \leqslant
  \delta.$
\end{itemize}

With these definitions, we need the following slight modification of 
\cite[Proposition 1]{CDS3}. 
\begin{prop}\label{prop:CDS}
  Given $M,c$ there are $\rho(M), \eta(M,c), \delta(M, c) > 0$ with
  the following effect. Suppose that $1-I(B) < \delta$ and $W\subset
  B$ is a subset with $m(\eta,W) < M$, such that for any $x\in
  B\setminus W$ there is a $(c\eta, \delta)$-chart centered at
  $x$. There is a constant $C$ depending only on the dimension, such that:
  \begin{enumerate}
  \item
    There is a holomorphic map $F:B(p,\rho)\to \mathbf{C}^n$ which is
    a homeomorphism to its image, $|\nabla F| < C$, and its image lies
    between $0.9\rho B^{2n}$ and $1.1\rho B^{2n}$.
  \item
    There is a local K\"ahler potential $\phi$ for $\omega$ on
    $B(p,\rho)$ with $|\phi|\rho^{-2} < C$.
  \item
    The slices $\{z_n = l\} \cap C^{-1}\rho B^{2n}$ are
    transverse to the Reeb field  for
    $|l| < C^{-1}\rho$.    
  \end{enumerate}
\end{prop} 

The way this proposition is used is that under the assumptions we can
use $F$ to think of $\omega$ as a metric on the Euclidean ball
$0.9\rho B^{2n}$, and because of the gradient bound for $F$ 
we have $\omega_{Euc} < C_1\omega$.  The new statement is (3), which
will essentially follow if we can show that $|\Xi|_{\omega_{Euc}}$ is
not too small near the origin. In fact we will show that given any
$\kappa > 0$, the $\eta, \delta$ in the proposition can be chosen so
that there exists a point $q\in \kappa\rho B^{2n}$ at which
$F_*(\omega) < 2\omega_{Euc}$. 

In order to prove this claim, we recall briefly the proof of
\cite[Proposition 1]{CDS3}. For this, let $\Omega(W,s)\subset B$
denote the set of points in $B$ at a distance greater than $s$ from
$W$ and from the boundary of $B$. One part of the proof of the
proposition is to produce an embedding $\phi : \Omega(W,r) \to
B^{2n}$, for which $\Vert \phi_*(\omega) - \omega_{Euc}\Vert_{L^p}
\leq \tilde{\theta}, |\phi_*(J) - J_{Euc}| \leq \tilde{\theta}$ and
$d(x,\phi(x)) \leq \tilde{\theta}$, where $r, \tilde{\theta}>0$ can be
chosen a priori, and $d$ denotes the distance function realizing the
Gromov-Hausdorff distance from $B$ to $B^{2n}$. In addition $\omega$
is the curvature of a metric on the trivial bundle on $B$, and $\phi$
can be lifted to a bundle map from the trivial bundle on $B$ to the
trivial bundle on $B^{2n}$, which almost identifies the
corresponding connections (c.f. \cite[Proposition 4]{CDS3}). 

The holomorphic function $F$ is now obtained by taking the holomorphic
functions $1, z_1, \ldots, z_n$ on $B^{2n}$, and applying suitable
cutoff functions to obtain approximately holomorphic functions
(sections of the trivial bundle) $\sigma, \sigma_1, \ldots, \sigma_n$
over $B$, vanishing on $\Omega(W,r)$, and near $\partial B$. Using the
H\"ormander $L^2$-estimate these can be projected to holomorphic
sections $s, s_1, \ldots, s_n$  (globally on our cone $X$), and $s_i/s$ give
the components of $F$. The result of Proposition 1 is then obtained by
choosing the parameters in the cutoff functions in a suitable way. 

In order to obtain (3), we just note that given $\kappa > 0$ we simply
need to choose $r$ much smaller than $\kappa\rho$, so that we can find
some $q \in \kappa\rho B$, which is contained in a definite ball $B_q$
disjoint from $\Omega(W, 2r)$, say.  If $\tilde{\theta}$ above is
sufficiently small, then on this ball $B_q$ the geometry of $\omega$
will be almost identical in an $L^p$-sense to the Euclidean
geometry. In particular the holomorphic function $F$ will be very
close to the identity map on $\frac{1}{2}B_q$. There will then exist a
point $q'\in \frac{1}{2}B_q$ at which $F_*(\omega) <
2\omega_{Euc}$. Note that $\Xi$ gives a holomorphic vector field on
$0.9\rho B^{2n}$, and it has bounded length with respect to the Euclidean
metric. In particular on $0.8\rho B^{2n}$ the derivatives of the
components of $\Xi$ are bounded. On the other hand we know that we can
choose a point $q'$ very close to the origin, where
$|\Xi|_{\omega_{Euc}} > 1/4$. Rotating coordinates, we can assume that the $\partial / \partial z_n$
component of $\Xi$ is non-zero inside $C_2^{-1}\rho B^{2n}$. This
implies our claim (3).

Given this result, the rest of the argument is quite similar to that
in \cite{Sz}. We give the required modifications of the proofs. The
following is analogous to \cite[Proposition 16]{Sz}.

\begin{prop}\label{prop:eps2} Given $M$, suppose that the ball $B$ satisfies the
  hypotheses of Proposition~\ref{prop:CDS} for some $c >0$. There are
  $A,\kappa > 0$ depending on $M$, such that if
\[ \int_B \alpha^\tau\wedge \omega^{n-1} < \kappa, \]
  then $\alpha^\tau < A\kappa\omega$ on $\frac{1}{3}C_2(K)^{-1}\rho B$, with
  $\rho=\rho(M), C_2(K)$ from Proposition~\ref{prop:CDS}. 
\end{prop}
\begin{proof}
  The assumption implies that
\[ \int_{0.9\rho B^{2n}} \alpha^\tau \wedge
\omega_{Euc}^{n-1} < C_1\kappa, \]
for some $C_1$. We argue similarly to the proof of
Proposition~\ref{prop:liminf>c0}. There is a slice $U_l = \{z_n = l\}\cap
C_2(K)^{-1}\rho B^{2n}$, with $|l| < \frac{1}{2}K^{-1}\rho B^{2n}$ on
which we have
\[ \int_{U_l} \alpha^\tau\wedge \omega_{Euc}^{n-1} < C_3\kappa. \]
The $\epsilon$-regularity implies that if $\kappa$ is sufficiently
small, then on a slightly smaller set we have $\alpha^\tau <
C_4\omega_{Euc} < C_4'\omega$. The flow of the Reeb field preserves
$\alpha^\tau$, and acts on $\omega$ by a controlled scaling factor, so we
obtain the required estimate on a ball of a definite size,
$\frac{1}{3}C_2(K)^{-1}\rho B^{2n}$.  
\end{proof}

For any ball $B(q,r)\subset B$ we define 
\[ V(q,r) = r^{2-2n} \int_{B(q,r)} \alpha^\tau\wedge \omega^{n-1}. \]
We have the following.
\begin{prop}
There are $\delta, \epsilon > 0$ depending on $K$ satisfying the following: if $1-I(B)
< \delta$ and 
\[ \sup_{B(q,r)\subset B} V(q,r) < \epsilon, \]
then $\alpha^\tau \leq 4\omega$ on $\frac{1}{2}B$. 
\end{prop}
\begin{proof}
  The proof follows the argument of \cite[Proposition 17]{Sz},
  together with the idea we used in the proof of \ref{prop:RegRicBnd}
  to make use of the vector field $\Xi$. 
\end{proof}

There is one other time when the strict positivity of $\alpha$ is used
in \cite{Sz}, namely in the proof of \cite[Proposition 19]{Sz} where the
$\epsilon$-regularity is used again. In the present setting the same
argument can be used, just like in the proof of 
Proposition~\ref{prop:eps2} above.  The remainder of
the proof of Proposition~\ref{prop:goodTang}
is identical to the argument in \cite{Sz}.

\subsection{Polynomial growth holomorphic
  functions}\label{sec:polygrowth}
In this section we assume that we have a K\"ahler cone $(C(L),
\omega)$, which has non-negative Ricci curvature, and is non-collapsed
(i.e. we have a lower bound on the volume of $L$). In addition we
assume that there is a constant $\epsilon_0 > 0$ with
the following property: for every $x\in L$ there is a holomorphic
function $f$ on $C(L)$ such that $f(0)=0$, $\Vert f\Vert_{L^2}=1$, and
$|f(x)|^2 > \epsilon_0$. Here the $L^2$-norm is with respect to the
weight $e^{-\frac{1}{2}r^2}$ as before. This extra property holds for any family
of solutions of~\eqref{eq:genMOC} with data moving in a bounded family
 by using the results of the
previous section on having good tangent cones. Our goal is to show
that up to replacing $\epsilon_0$ by a smaller constant, we can take
$f$ to have polynomial growth with controlled degree. More precisely
we have the following.

\begin{prop}\label{prop:polygrowth}
  There are constants $D, \epsilon_1 > 0$, depending on $\epsilon_0$
  and the lower bound on the volume of $L$, such that the following
  holds: if $x\in L$, then there is a holomorphic function $f$ on
  $C(Y)$ with $f(0)=0$, $\Vert f\Vert_{L^2}=1$, $|f(x)|^2 > \epsilon_1$,
  and in addition $|f| = O(r^D)$. 
\end{prop}
\begin{proof}
  First let us write $\mathcal{H}$ for the space of $L^2$ holomorphic
  functions on $C(L)$, which we can decompose into weight
  spaces under the torus action
  \[ \mathcal{H} = \bigoplus_{\chi\in\mathfrak{t}^*}
  \mathcal{H}_\chi, \]
  where infinite convergent sums are allowed. 
  In addition we have the Reeb field $\xi\in \mathfrak{t}$.

  As we mentioned above, we already have an $f$ with the required
  properties, except for the growth condition. To restate
  our goal, we are trying to construct an $f'$ that also satisfies
  the growth condition, which is equivalent to
  \[ f\in \bigoplus_{\substack{\chi\ne 0 \\ \langle\chi, \xi\rangle
      < D}} \mathcal{H}_\chi. \]

  Lemma~\ref{lem:numberofweights} below shows that we have a constant
  $C$ such that
  \[ \#(\mathcal{C}\cap \{\chi\in \mathfrak{t}^*\,:\, \langle
  \chi,\xi\rangle \in (w-1,w]\}) < C 5^w \]
  for all $w \geq 1$. Note that this estimate is far from optimal, but
  it is enough for our purposes here. 

  Suppose that we have $f$ such that $f(0)=0$, $|f(x)|^2 > \epsilon_0$
  and $\Vert f\Vert_{L^2} =1$. We write
  \[ f = \sum_{w=1}^\infty \sum_{i=1}^{N_w} f_{w,i}, \]
  where each $f_{w,i}$ is in a weight space
  $\mathcal{H}_\chi$ with $\langle \chi,\xi\rangle \in (w-1,w]$. We
  have $N_w < C5^w$.

  From Lemma~\ref{lem:growthandC0} below we have that on $L$
  \[ |f_{w,i}|^2 \frac{w!}{2^w} < C\Vert f_{w,i}\Vert^2_{L^2}, \]
  and so for any $D > 0$ we have
  \[ \begin{aligned}
    \left| \sum_{w=D}^\infty \sum_{i=1}^{N_w} f_{w,i}\right|^2 &\leq
    \left(\sum_{w=D}^\infty \sum_{i=1}^{N_w} \frac{2^w}{w!}\right)
    \left( \sum_{w=D}^\infty \sum_{i=1}^{N_w} |f_{w,i}|^2 \frac{
        w!}{2^w} \right) \\
    &\leq C \Vert f\Vert_{L^2}^2 \sum_{w=D}^\infty
    \frac{10^w}{w!}. 
  \end{aligned}\]

  Since the series on the right converges, and $|f(x)|^2 >
  \epsilon_0$, we can choose $D$ sufficiently large so that at $x$ we have
   \[ \left| \sum_{w=1}^D \sum_{i=1}^{N_w} f_{w,i}\right|^2 >
   \frac{\epsilon_0}{2}. \]
   We can then let
   \[ f' = \sum_{w=1}^D \sum_{i=1}^{N_w} f_{w,i}, \]
   and $f' \Vert f'\Vert_{L^2}^{-1}$ will satisfy the required properties. 
\end{proof}

\begin{lem}\label{lem:growthandC0}
  Suppose that $f\in \mathcal{H}_\chi$, and let $w = \lceil \langle\chi,
  \xi\rangle \rceil$. Then we have
  \[ |f|^2 < \frac{2^w C \Vert f\Vert_{L^2}^2}{w!} \quad\text{ on } L. \]
\end{lem}
\begin{proof}
  Assume $\Vert f\Vert_{L^2}=1$.  We have (with dimensional factors $c_n$)
\[ \begin{aligned} 1 = \Vert f\Vert_{L^2}^2 &\geq c_n \Vert f\Vert_{L^2(L)}^2
\int_1^\infty r^{2w+2n-3} e^{-\frac{1}{2}r^2}\,dr \\
  &\geq c_n' 2^w (w+n-2)! \Vert f\Vert_{L^2(L)}^2. 
\end{aligned} \]
The $L^2$ norm on the link gives a $C^0$-estimate on the half ball
around the cone vertex, and then using the assumed growth rate on $f$
we obtain a $C^0$-estimate on the link $L$: 
\[ \sup_L |f|^2 < \frac{2^w C}{(w+n-2)!}. \]
\end{proof}

We have also used the following simple estimate 
on the dimension of the space of holomorphic functions of polynomial growth. 
\begin{lem}\label{lem:numberofweights}
  For an integer $w \geq 1$, let us write $\mathcal{H}_w$ for the
  space of holomorphic functions satisfying the growth condition $|f|
  = O(r^w)$ as $r \to\infty$. We then have
  \[ \dim \mathcal{H}_w < C5^w. \]
\end{lem}
\begin{proof} 
  This follows a standard argument using the previous lemma. 
  Let $\{f_1,\ldots, f_N\}$ be an $L^2$-orthonormal basis for
  $\mathcal{H}_w$, and define the function $B$ by
  \[ B(x) = \sum_{i=1}^N |f_i(x)|^2 e^{-\frac{1}{2}r^2}. \]
  The previous lemma implies that 
  \[ B \leq \frac{2^w C}{w!} \max\{1, r^{2w}\}, \]
  and so
  \[ \int_{C(L)} B\,\omega^n \leq C 5^w. \]
  On the other hand, by definition the integral of $B$ is the
  dimension of $\mathcal{H}_w$. 
\end{proof}

\subsection{Synthesis}
Let us recall where we stand.  Suppose we have a sequence of cone K\"ahler metrics $\omega_{t_k}$
on $(X, t_k^{-1}\xi_k)$, with radial functions $r_k$, satisfying the equations
\[ \mathrm{Ric}(\omega_{t_k}) = \frac{1-t_k}{t_k} \alpha_{k}^\tau, \]
and $t_k\to T$, $\xi_k \rightarrow \xi$ and $\alpha_k \rightarrow \alpha$.
In addition the sequence $(X, \omega_{t_k})$ converges in the
Gromov-Hausdorff sense to $(Z, d_Z)$. Our work so far leads to 
the following ``partial $C^0$-estimate''.

\begin{prop}
There exists a constant $D>0$ depending only on the non-collapsing constant,
the dimension and a bound for the geometry of $(\xi_k, \alpha_k)$ such that if 
$\{f_{1}^{(k)},\ldots, f_{N}^{(k)}\}$ denotes an $L^2$ orthonormal basis of $R_{<D}(X)$,
then the map $F_k :X \rightarrow \mathbf{C}^{N}$ whose components are given by the $f^{(k)}_{i}$
gives an embedding of $X$.  Furthermore, there is a uniform constant $C$ so that
\[ C^{-1} < |F_k| < C,\qquad C^{-1}F_{k}^{*}\omega_{Euc} < \omega_{t_k},
\]
on the set $\{1/2 < r_k < 2\}$.
\end{prop}
\begin{proof}
  In Sections \ref{sec:cinot0} and \ref{sec:ci0} 
  we have shown that the iterated tangent cones in suitable
  Gromov-Hausdorff limits are good, and so the arguments of
  Donaldson-Sun~\cite{DS} show that the assumptions of
  Proposition~\ref{prop:polygrowth} apply: 
  for each $x\in L$, we can find
  a holomorphic function $f$ on $X$ with polynomial growth of bounded
  degree, unit $L^2$-norm on $L$, and $|f(x)|^2 > \epsilon_1$ for a
  fixed number $\epsilon_1$ (independent of $k$). Increasing $D$ if
  necessary, we can assume that $F_k$ gives an embedding (note that
  $X$ is fixed, only the metric is changing). It then follows directly
  that $C^{-1} < |F_k|$ on $L$, while the bound on the growth rate
  implies that the same estimate (with different $C$) also holds on
  the annulus $1/2 < r < 2$. A uniform bound $|F_k| < C$ and 
  derivative bound $|\nabla
  F_k|_{\omega_{t_k}} < C$ on the annulus follows by using Moser
  iteration, and this implies the estimate $C^{-1}F_k^*\omega_{Euc} <
  \omega_{t_k}$. 
\end{proof}

Arguing as in Donaldson-Sun \cite[Section 2]{DS2} we can deduce
the convergence of the affine varieties
\[
F_{k}(X) \longrightarrow Y
\]
where $Y$ is a normal, $\mathbf{Q}$-Gorenstein affine variety with Reeb vector field $\xi$,
and furthermore $Y$ is homeomorphic to $Z$.
In order to finish the proof of Theorem~\ref{thm:partialc0} it suffices to prove the last statement
regarding convergence to a weak solution of the twisted equation, after possibly passing
to a further subsequence.

Consider the pushed forward forms $(F_k)_*\alpha^\tau$.  We would like to take a weak limit of these
forms.  As a first step we prove the volume is bounded below.  Suppose that $\alpha = \frac{1}{2}\ddb r^2$
is a cone metric, with transverse form $\alpha^{\tau} = \ddb \log r$.  Recall \cite{Sparks_survey} we have the form
$\eta := \sqrt{-1}(\dbar - \del)\log r$.  In terms of $\eta$ we have $\alpha^{\tau} = \frac{1}{2}d\eta$, and 
$\alpha = \frac{1}{2}d(r^2\eta)$.  Then we compute
\[
\alpha^{\tau} \wedge \alpha^{n-1} = 2^{1-n}(n-1)r^{2n-3}dr\wedge \eta\wedge(d\eta)^{n-1}.
\]
Writing everything on the cone we get
\[
\begin{aligned}
\int_{\{r \leq 1\} } \alpha^{\tau} \wedge \alpha^{n-1} &= 2^{1-n}(n-1)\int_{0}^{1}r^{2n-3}dr \int_{r=1}\eta\wedge (d\eta)^{n-1} \\
&=c_n\,Vol(L,\xi)>0
\end{aligned}
\]
where $Vol(L,\xi)$ is the volume of the link, which is a topological invariant depending only on the Reeb field.  Furthermore, if $\omega$ is another metric compatible with $\xi$, with radial function $\tilde{r}$ then we can write $r= e^{\psi}\tilde{r}$ for a function $\psi$ which is independent of $r$, and $L_{\xi}\psi=0$.  Then it is easy to check that
\[
\int_{\{\tilde{r}\leq 1\}} \alpha^{\tau} \wedge \omega^{n-1} = c_n\, Vol(L,\xi).
\]
Now suppose we have  sequence of maps $F_{k} : X \rightarrow \mathbf{C}^{N}$ as above and consider the closed positive currents
\[
(F_k)_{*}\alpha^{\tau} \wedge [F_{k}(X)].
\]
Let $A:= \{p\in \mathbf{C}^{N} :M^{-1}<|p|<M\}$ for some constant $M >
0$.
Let $A':= F_{k}^{-1}(A \cap F_{k}(X))$.  Suppose that $\nu$ is a
smooth positive $(n-1,n-1)$ form with compact support in $A$.  Then we
have 
\[
\int_{\mathbf{C}^{N}}(F_k)_{*}\alpha^{\tau} \wedge [F_{k}(X)]\wedge \nu = \int_{A'} \alpha^{\tau} \wedge F_{k}^{*}\nu
\]
We can find a constant $C_1$ such that on $A$ we have $ \nu \leq C_1\omega_{Euc}^{n-1}$ as $(n-1,n-1)$ forms.  Furthermore, by the properties of $F_k$ we have
\[
  F_{k}^{*}\omega_{Euc} < C \omega_k
\]
on $A'$, where $\omega_k$ is our metric, and the partial $C^0$ estimate implies that
\[
\{1/2 <r_{t_k} <2\} \subset A' \subset \{r < C'\}
\]
for some constant $C'$, provided $M$ is sufficiently large.  The above
discussion implies a uniform upper bound  
\[
\int_{A'} \alpha^{\tau} \wedge F_{k}^{*}\nu<C,
\]
with the constant $C$ depending only on the form $\nu$. 
It follows that $(F_k)_{*}\alpha^{\tau} \wedge [F_{k}(X)]$ converges
weakly to a 
closed positive current $\beta^{\tau}$ on $Y = \lim_{k} F_{k}(X)$.

By arguing as in \cite{Datar-Sz} we can show that $Y$
admits a weak solution $\omega_T$ of the equation
\[ \mathrm{Ric}(\omega_T) = \frac{1-T}{T}\beta^\tau. \]
Note that the proof of this result in \cite{Datar-Sz} is essentially
local, working in neighborhoods of points $p$ in the limit $Y$ where the
complex structure of $Y$ is smooth, and where in terms of the metric
structure we have a tangent cone of the form $\mathbf{C}^n$ or $\mathbf{C}^{n-1}\times
\mathbf{C}_\gamma$. On this set we have
good local coordinates, and so we can study the limiting equation,
while at the same time the complement of this
set has Hausdorff codimension greater than two, and therefore it can
be ignored in terms of writing down weak solutions of the twisted
equation, as in \cite[Remark 4]{Datar-Sz}.
This completes the proof of Theorem~\ref{thm:partialc0}.

\section{Proof of the main result}\label{sec:mainproof}
In this section we prove on direction of Theorem~\ref{thm:mainthm}. Recall that we
have a normalized Fano cone singularity $(X,\xi)$ together with the action of a
torus $\mathbf{T}$, whose Lie algebra contains $\xi$. We assume that $(X,\xi)$ is
$\mathbf{T}$-equivariantly K-stable, and our goal is to show that $(X,\xi)$
admits a Ricci flat K\"ahler cone metric. The proof
naturally splits into two cases depending on whether $(X,\xi)$ is
quasi-regular, or irregular. We will first focus on the former, and
then we will deal with irregular $\xi$ by approximating it with a
sequence of quasi-regular Reeb fields.

\subsection{The quasi-regular case}

Suppose that $(X,\xi)$ is quasi-regular. We first fix a
$\mathbf{T}$-invariant transverse
K\"ahler metric $\alpha^\tau$, using an embedding $X\to\mathbf{C}^N$
by a collection of (non-constant) holomorphic functions just as in
Equation~\eqref{eq:hatr2}. We can then solve the
continuity method \eqref{eq:contmethod} up to some time $T\leq 1$.

According to Theorem~\ref{thm:partialc0} there is a number $D$,
depending on the bound on the geometry of the twisting form
$\alpha^\tau$ and on the pair $(X,\xi)$ through the non-collapsing
condition,  so
that using orthonormal bases of holomorphic functions with growth
rates less than $D$ we obtain embeddings $F_k : X\to \mathbf{C}^N$, such that
$F_k(X) \to Y$ with a normal limit space $Y$. In addition we have convergence
of twisting forms
$(F_k)_*\alpha^\tau \to \beta^\tau$ and the limit $Y$ admits a weak
solution of the twisted equation
\[ \mathrm{Ric}(\omega_T) = \frac{1-T}{T} \beta^\tau. \]

In order to apply the results from Section~\ref{sec:Ding} we need to
make sure that $\alpha^\tau$ can be written as an integral over
currents of integration (with respect to a positive measure) over
suitable hypersurfaces. The
problem with $\alpha^\tau$ defined as in \eqref{eq:hatr2} is that it
does not use all the functions in $R_{<D}(X)$ and so in the integral
expression \eqref{eq:alphaintegral} we are not using a positive
measure on $\mathbf{P}$. To fix this, we
 will modify $\alpha^\tau$ slightly by adding small terms corresponding to
 the remaining functions in $R_{<D}(X)$, and so that the new transverse
metric still has the same bound on its geometry. In particular the above
discussion still holds with the same constant $D$. Let us write
$\alpha^\tau = \ddb \log \hat{R}$, and we will call the perturbed
radial function $\hat{r}$. 

Let us decompose
\[ R_{<D}(X) = R_1\oplus \ldots \oplus R_m \oplus R_{m+1}
\oplus\ldots \oplus R_k \]
into weight spaces of $\mathbf{T}$, where $\alpha^\tau$ is defined as
in \eqref{eq:hatr2} using bases of $R_i$ for $i\leq m$. Suppose that
$\xi$ acts on the functions in $R_i$ with weight $a_i$ as before, and
recall that we chose $M$ so that $M/a_i\in \mathbf{Z}$ for $i\leq
m$. We also choose $K$ so that $K/a_i\in \mathbf{Z}$ for all $i$, and
then for $\delta > 0$ we define a new radial function $\hat{r}$ by 
\[ \begin{aligned} \hat{r}^{2K} &= \left[ \sum_{i=1}^m\left(\sum_{j=1}^{N_i}
      |z^{(i)}_j|^2\right)^{M/a_i}\right]^{K/M} + \delta\sum_{i=m+1}^k\left(
      \sum_{j=1}^{N_i} |z^{(i)}_j|^2\right)^{K/a_i} \\
&= \hat{R}^K + \delta\sum_{i=m+1}^k\left(
      \sum_{j=1}^{N_i} |z^{(i)}_j|^2\right)^{K/a_i}.
\end{aligned}  \]
Here the $z^{(i)}_j$ form a basis for $R_i$. 
As $\delta\to 0$, we recover the original radial
function $\hat{R}$, and so for sufficiently small $\delta$ the
transverse metric $\ddb\log\hat{r}$ has the same bounded geometry as
$\ddb\log\hat{R}$, but at the same time the methods of
Section~\ref{sec:Ding}, in particular
Proposition~\ref{prop:twistedformula}, can be applied.

At this point we are in essentially
the same setup as in \cite[Section 3.1]{Datar-Sz}, and can follow
the argument there closely. We have a sequence $\rho_k\in GL(N)^\mathbf{T}$,
such that $F_k = \rho_k\circ F_1$. For simplicity of notation let us
write $F_1(X) = X$ and $(F_1)_*(\alpha^\tau) = \alpha^\tau$. Then
in the notation of Section~\ref{sec:Ding}, on $X$ we have
\begin{equation}\label{eq:alphatau}
  \alpha^\tau = \frac{\pi}{M} \int_{\mathbf{P}\setminus F}
    [V_\mu]\,d\mu,
\end{equation}
where $F$ is a hyperplane in the projective space $\mathbf{P}$ and
each $V_\mu$ is a hypersurface in $X$. Similarly to \cite[Lemma
14]{Datar-Sz} we can choose a subsequence of the $\rho_k$, such that
$\rho_k(V_\mu)$ converges for all $\mu$. Let us
define
\[ \rho_\infty(V_\mu) = \lim_{k\to\infty} \rho_k(V_\mu). \]
Generalizing Definition~\ref{defn:stab}, for any $(1,1)$-current on
$Y$ we let $\mathfrak{g}_{Y,\xi,\chi}$
denote the holomorphic vector fields on $Y$, commuting with $\xi$ such
that $\iota_\xi\chi=0$. Then as in \cite[Lemma 15]{Datar-Sz} we have

\begin{lem}\label{lem:reductiveDS}
We can find $\mu_1, \ldots, \mu_d$ for some $d$, such that
  \[ \mathfrak{g}_{Y, \xi,\beta^\tau} = \bigcap_{i=1}^d
  \mathfrak{g}_{Y, \xi, [\rho_\infty(V_{\mu_i})]}. \]
\end{lem}
The proof is the same as in \cite{Datar-Sz}, except instead of the
Fubini-Study volume form we use $e^{-\frac{1}{2}\hat{r}^2} \omega^n$,
where $\omega=\frac{1}{2}\ddb \hat{r}^2$ as usual. 

By Lemma~\ref{lem:reductiveDS} and Proposition~\ref{prop:reductive}
 we can choose hypersurfaces $V_1', \ldots, V_d'$ in
$\mathbf{P}\setminus F$ such that the
automorphism group of the $(d+1)$-tuple $(Y, \rho_\infty(V_1'),\ldots,
\rho_\infty(V_d'))$ is reductive, by
Proposition~\ref{prop:reductive}. Since $Y$ may be contained in a
hyperplane, the stabilizer of this $(d+1)$-tuple in the 
multigraded Hilbert scheme, under the action of
$GL(N)^\mathbf{T}$ may contain extra factors of $GL(k_i)$ for suitable
$k_i$, but this product is still reductive.  At this point, we can use
 the Luna slice theorem (see \cite{Luna} and also
\cite{Don12,CDS3})
to find a $\mathbf{C}^*$-subgroup $\lambda(t)\subset GL(N)^\mathbf{T}$ with
$\lambda(S^1)\subset U(N)^\mathbf{T}$, and an
element $g\in GL(N)^\mathbf{T}$ such that
\[ \Big( Y, \rho_\infty(V_1'),\ldots, \rho_\infty(V_d')\Big) =
\lim_{t\to 0} \lambda(t)g\cdot \Big(X, V_1', \ldots, V_d'\Big). \]
More generally, enlarging $F$ by a set of measure zero, if $V_1,
\ldots, V_K$ are hypersurfaces in $\mathbf{P}\setminus F$, then the
stabilizer of the $(1+d+K)$-tuple
\[ (Y, \rho_\infty(V_i'), \rho_\infty(V_j))_{i=1,\ldots,d,
  j=1,\ldots, K}\]
is unchanged, and so the Luna slice theorem provides a
corresponding $\mathbf{C}^*$-subgroup $\lambda(t)$ and element $g$.
These will satisfy
\[ \label{eq:rho=lambda}
\begin{aligned} Y &= \lim_{t\to 0} \lambda(t)g\cdot X, \\
  \rho_\infty(V_i) &= \lim_{t\to 0} \lambda(t)g \cdot V_i.
\end{aligned}\]
Note that we do not necessarily have $\beta^\tau = \lim_{t\to 0}\lambda(t)g\cdot
\alpha^\tau$, and that the one-parameter subgroup $\lambda(t)$ and $g$
may depend on the choice of $V_1,\ldots,V_K$.  Note also that 
the vector field $w$ on $Y$ induced by the $\mathbf{C}^{*}$ subgroup $\lambda$ will stabilize
$\rho_{\infty}(V_i')$ for $1 \leq i \leq d$, and hence by Lemma~\ref{lem:reductiveDS} we have
$w\in\mathfrak{g}_{Y,\xi,\beta^{\tau}}$.  

It is important to choose $V_1,\dots,V_K$ above correctly,
and we will discuss how this is done shortly, but for the time being, let us assume we have a $\mathbf{C}^{*}$ subgroup
$\lambda$ generated by a vector field $w$ commuting
with $\xi$.  Let $\theta_w$ denote the transverse Hamiltonian with
respect to the radial function $\hat{r}$. We will assume
that $\theta_w$ is normalized so that
\[ \int_Y \theta_w e^{-\frac{1}{2}\hat{r}^2} \omega^n = 0, \]
and write $\Vert w\Vert = \sup_Y |\theta_w|$. Note that any two
choices of norm are equivalent on the finite dimensional space of
holomorphic vector fields on $Y$ commuting with $\xi$. In addition we
cannot have $\theta_w=0$ on $Y$, unless $\lambda$ already acts
trivially on $X$. To see this note that $\lambda$ induces a filtration
on the coordinate ring of $X$, and the coordinate ring of $Y$ is the
associated graded ring, with the action $\lambda$ on $Y$ being induced
by the corresponding grading. If this grading is trivial, then the
original filtration must have been trivial.

As in \cite{Datar-Sz}, the idea is to
use Proposition~\ref{prop:generic_hyper} to estimate the twisted
Futaki invariant of the test-configuration $\lambda(t)$ for $g\cdot
X$.  In order to apply Proposition~\ref{prop:generic_hyper} we must ensure
that our hypersurfaces avoid the set $E$, which is a union two hyperplanes in $\mathbf{P}$.  On
other other hand, if no $N+1$ of our hypersurfaces $V_1,\ldots, V_K$ are
on a hyperplane in $\mathbf{P}$, then at most $2N$ of them can be
contained in $E$. So once we choose $K$ very large
compared to $N$, then we get a good approximation to the twisted
Futaki invariant by replacing $g\cdot\alpha^\tau$ by the average of
$g\cdot[V_i]$.

We now describe how to choose the $V_1,\ldots,V_K$.  
For simplicity of notation we will suppose that $g$ is the identity. 
From \eqref{eq:alphatau} we have that on $Y$
\[ \beta^\tau = \frac{\pi}{M} \int_{\mathbf{P}\setminus F}
[\rho_\infty(V_\mu)]\,d\mu. \]
In addition since $(Y, \xi, (1-T)\beta^\tau)$ admits a weak solution
of the twisted equation, from Proposition~\ref{prop:Fut0} we have
\[ \mathrm{Fut}_{Y, \xi, (1-T)\beta^\tau}(w) = 0 \]
for all $w\in \mathfrak{g}_{Y, \xi, \beta^\tau}$. By the formula for
the twisted Futaki invariant in Proposition~\ref{prop:FuteqProp} this means that
\[ \mathrm{Fut}_{Y, \xi}(w) - c_n\frac{1-T}{V}\int_Y \theta_w
e^{-\frac{1}{2}\hat{r}^2} \beta^\tau \wedge \omega^{n-1} = 0, \]
and so if we define the function $h: \mathbf{P}\setminus F \to
\mathbf{R}$ to be
\[ h_w(\mu) = \frac{\pi}{M} \int_{\rho_\infty(V_\mu)} \theta_w
e^{-\frac{1}{2} \hat{r}^2} \omega^{n-1}, \]
then we have
\[ \mathrm{Fut}_{Y, \xi}(w) = c_n\frac{1-T}{V} \int_{\mathbf{P}\setminus
  F} h_w(\mu)\,d\mu. \]
Since the possible $w$ form a finite dimensional space, for any
$\epsilon > 0$ we can choose a large $K_0$ with the following effect:  for any $K \geq K_0$ we can
find hypersurfaces $V_1,
\ldots, V_K$ with no $N+1$ on a hyperplane in $\mathbf{P}$, such that
\begin{equation}\label{eq:Futsum}
  \mathrm{Fut}_{Y, \xi}(w) \leq \epsilon\Vert w\Vert +
c_n\frac{1-T}{V}\cdot \frac{1}{K}\sum_{i=1}^K \frac{\pi}{M}
\int_{\rho_\infty(V_i)} \theta_w e^{-\frac{1}{2}\hat{r}^2}
\omega^{n-1}.
\end{equation}
To see this more precisely, we choose a basis $w^1, \ldots, w^k$ for $\mathfrak{g}_{Y,\xi,\beta^{\tau}}$,
and then using Lusin's
theorem applied to each $h_{w^i}$ we can find a compact set $B\subset
\mathbf{P}\setminus F$ with arbitrarily small complement, on which
each $h_{w^i}$ is continuous. We can then approximate the integrals on
$B$ using Riemann sums over discrete finite sets.

Assume $T<1$.  There is a constant $\delta>0$ so that
\[
\int_{Y}(\max_{Y}\theta_{w}-\theta_w)e^{-\frac{1}{2}\hat{r}^2}\omega^n > \delta\|w\|
\]
since the left hand side is a norm on the finite dimensional vector space $\mathfrak{g}_{Y,\xi,\beta^{\tau}}$. 
Choose $0<\epsilon < 4^{-1}c(n)(1-T)\delta$ where $c(n)$ is the constant from Proposition~\ref{prop:generic_hyper}.
Take $K$ sufficiently large as above, and so that $K>2N\epsilon^{-1}$.
Using the hypersurfaces $V_1,\ldots, V_K$
we now find $\lambda$ as discussed before (and for simplicity assume $g$ is the
identity). It follows that
$\rho_\infty(V_i) = \lim_{t\to 0}\lambda(t)\cdot V_i$.

We now use Proposition~\ref{prop:generic_hyper}, which implies that in
the sum in \eqref{eq:Futsum}, for all but $2N$ of the integrals we have
\[ \frac{\pi}{M} \int_{\rho_\infty(V_i)} \theta_w e^{-\frac{1}{2}\hat{r}^2}
\omega^{n-1} = -c(n) \int_Y \mathrm{max}_Y\theta_w
e^{-\frac{1}{2}\hat{r}^2} \omega^n. \]
From our choice of $K$ we have
\begin{equation}\label{eq:Futbound2}
\begin{aligned}
 \mathrm{Fut}_{Y, \xi}(w) &\leq 2\epsilon\Vert w\Vert -c(n)
\frac{1-T}{V} \int_Y \mathrm{max}_Y\theta_w
e^{-\frac{1}{2}\hat{r}^2} \omega^n\\
&\leq -c(n)\frac{1-T}{2}\max_{Y}\theta_{w}.
\end{aligned}
\end{equation}
This is a contradiction if $(X,\xi)$ is $K$-stable.

If $T=1$, then $(Y,\xi)$ admits a weak Ricci flat metric, and then
as in \cite{Datar-Sz} or Donaldson-Sun~\cite[Section 3.3]{DS2}, we obtain a
test-configuration for $(X,\xi)$ with vanishing Futaki invariant. 

\subsection{The irregular case}
Suppose now that $(X,\xi)$ is irregular, and let $\xi_k\to \xi$ be a
sequence of normalized quasi-regular Reeb fields approximating
$\xi$. Let $\alpha_k^\tau\to \alpha^\tau$ be a sequence of compatible
transverse K\"ahler metrics, obtained by restricting suitable
reference forms under an embedding $X\to\mathbf{C}^M$. Note that we
may not be able to choose the $\alpha_k^\tau$ to be exactly of the
form considered in the previous section. On the other hand it is easy to
show in an identical way to the argument in \cite{SzRM} that if
$(X,\xi, (1-t)\alpha^\tau)$ admits a solution of the twisted equation,
then so does $(X, \xi, (1-t)\tilde{\alpha}^\tau)$ for any
$\tilde{\alpha}^\tau$ in the same transverse cohomology class as
$\alpha^\tau$. 

For each $k$ we obtain a $T_k < 1$, such that we can solve
\[ \mathrm{Ric}(\omega_t^{(k)}) = 2n (1-t)[\alpha_k^\tau -
\omega_t^{(k), \tau}] \]
on $(X, \xi_k)$ for $t\in [0,T_k)$. Note that if we choose $\xi$
to be the Reeb field with minimal volume, then necessarily $T_k < 1$,
since each $(X,\xi_k)$ will be strictly unstable if $\xi_k \ne
\xi$. Indeed by \cite{MSY} the Reeb field with minimal volume is
unique and a Sasaki-Einstein metric can exist only for that Reeb
field.  
\begin{prop}
  If $(X,\xi)$ is K-stable, then $T_k\to 1$ as $k\to\infty$.
\end{prop}
\begin{proof}
  Let us suppose that $\limsup T_k < 1$. 
 From \eqref{eq:Futbound2}, for each $k$ we obtain a 
  vector field $w_k$, inducing a test-configuration for $X$ with
  central fiber $Y_k$, such that
  \begin{equation}\label{eq:Futupper}
\mathrm{Fut}_{Y_k, \xi_k}(w_k) \leq -\delta\,
  \mathrm{max}_{Y_k} \theta_{w_k}, 
\end{equation}
  for some $\delta > 0$. Applying Theorem~\ref{thm:partialc0} to a
  diagonal sequence, we can assume that $Y_k \to Y$, and 
  moreover that $Y$ is normal.

  Our goal is to show that for sufficiently large $k$ we have
  \[ \mathrm{Fut}_{Y_k, \xi}(w_k) < 0, \]
  since this will contradict the K-stability of $(X,\xi)$. Recall that
  we have normalized each $\theta_{w_k}$ to have zero integral on
  $Y_k$. It is worth pointing out that the background metrics on
  $\mathbf{C}^N$ are also varying with $k$, and so the transverse
  Hamiltonians are all computed with different metrics, but these
  background metrics also converge as $k\to\infty$ to a metric for the
  limiting Reeb field $\xi$.  In addition as discussed in
  Remark~\ref{rem:maxthetaweight}, the normalization of $w_k$ and
  $\max_{Y_k}\theta_{w_k}$ only depends on the induced action on
  $Y_k$, and not for instance on the background metric used to compute
  the Hamiltonians. 

 By \eqref{eq:Futdef2} the Futaki invariant is
  \[ \mathrm{Fut}_{Y_k, \xi_k}(w_k) = -\frac{\int_{Y_k} \theta_{w_k}
    e^{-\frac{1}{2}\hat{r}_k^2}
    dV_k}{\int_{Y_k}e^{-\frac{1}{2}\hat{r}_k^2} dV_k}, \]
  where $dV_k$ is the canonical volume form on $Y_k$. Note that these
  can be scaled so that they converge to the canonical volume form on
  $Y$. In addition if we scale each $w_k$ so that $\mathrm{osc}_{Y_k}
  \theta_{w_k} = 1$, then we can extract a limit $\theta_w$ on $Y$,
  corresponding to a vector field $w$. Note that $w$ may not be
  defined on all of $\mathbf{C}^N$ if $Y$ is contained in a
  hyperplane, since in this case our scaling might make some of the
  weights of $w_k$ become unbounded.  
  However we must still have $\mathrm{osc}_Y \theta_w = 1$.

From our
normalization this implies a positive lower bound for 
  $\max_{Y}\theta_w$ and hence a uniform positive lower bound for
  $\max_{Y_k}\theta_{w_k}$. 
 From  \eqref{eq:Futupper} we obtain 
  \[ \mathrm{Fut}_{Y_k,\xi_k}(w_k) \leq -\delta'.\]
  Now the required result follows, since both $\mathrm{Fut}_{Y_k,
    \xi_k}(w_k)$ and $\mathrm{Fut}_{Y_k, \xi}(w_k)$ converge to
  $\mathrm{Fut}_{Y,\xi}(w)$, which we have just seen must be
  negative.  
\end{proof}

We can now choose metrics $\omega^{(k)}_{t_k}$ on $(X,
\xi_k)$, satisfying
\[ \mathrm{Ric}(\omega_{t_k}^{(k)}) = 2n(1-t_k)[
\alpha_k^\tau - \omega_{t_k}^{(k),\tau}], \]
with $t_k = T_k - k^{-1}$. We can apply the same arguments as above to
the sequence $(X, \omega^{(k)}_{t_k})$, to obtain a normal limit space
$Y$, which admits a weak Ricci flat metric, since $t_k\to 1$. As
before, we can then realize $Y$ as the central fiber of a
test-configuration for $X$, contradicting the assumption that
$(X,\xi)$ is K-stable, unless $Y\cong X$. But if $Y\cong X$, then we
have obtained the desired Ricci flat K\"ahler cone metric on $X$.

\section{The algebraic Futaki invariant}\label{sec:algfutaki}
In this section we collect some results of a more algebraic
nature. First we will consider the normalization (or gauge fixing) condition for a Fano
cone singularity $(X,\mathbf{T},\xi)$, 
and show that normalized Reeb fields $\xi$ form a
linear subspace of the Lie algebra of $\mathbf{T}$. We then discuss
the relation between the algebro-geometric definition of the Futaki
invariant in Definition~\ref{defn:Futaki} and the differential
geometric definition in Equation~\ref{eq:Futdef2}. The analogous
result in the smooth projective case was shown by 
Donaldson~\cite[Proposition 2.2.2]{Don_toric}, using the
equivariant Riemann-Roch formula. Here our approach is more algebraic
in order to avoid having to resolve possible singularities. 

\subsection{The Gauge Fixing Condition}
We begin with a lemma alluded to in Section~\ref{sec:background}.
\begin{lem}
Suppose $(X,\mathbf{T},\xi)$ is $\mathbf{Q}$-Gorenstein with an isolated singularity at the origin.  Suppose $\Omega \in \Gamma(X,mK_{X})$ is a non-vanishing section with $L_{\xi}\Omega= i\lambda \Omega$ for some $\lambda \in \mathbf{R}$.  Then $X$ has log-terminal singularities at $0$ if and only if $\lambda>0$.
\end{lem}
\begin{proof}
Define a volume form $dV = i^{n^2}\left(\Omega \wedge \overline{\Omega}\right)^{1/m}$.  By \cite[Lemma 6.4]{EGZ} it suffices to determine conditions for $dV$ to have finite volume in a neighborhood of $0$.  Let $r$ be a radial function for $\xi$.  Using the flow by $-J\xi$, we can write
\[
\int_{\{2^{-k} \leq r <2^{-(k-1)}\}}dV = e^{-\lambda k \log 2} \int_{\{\frac{1}{2}\leq r<1\}} dV
\]
and so we get
\[
\int_{\{0<r<1\}}dV = \sum_{k=0}^{\infty} e^{-\lambda k \log 2} \cdot \int_{\{\frac{1}{2}\leq r <1\}}dV
\]
which proves the lemma.  
\end{proof}
Note that the proof implies slightly more.  Namely, that if $(X,\mathbf{T},\xi)$ is a normal,  $\mathbf{Q}$-Gorenstein, with $\Omega$ as above, then $X$ is log terminal if and only if $\lambda >0$, and $X$ is log-terminal at all points away from the origin.

Suppose we have an $n$-dimensional polarized affine variety
$(X,\mathbf{T},\xi)$ which is a Fano cone singularity, so $X$ is
normal $\mathbf{Q}$-Gorenstein with log-terminal singularities. In
addition suppose that $m\in \mathbf{N}$ is minimal such that $mK_{X}$ is trivial.  Let $R$ be the coordinate ring of $X$, which is an integral domain since $X$ is a variety.  Furthermore, since $X$ has log-terminal singularities, it is known that $R$ is Cohen-Macaulay.  When $X$  is Gorenstein with an isolated singularity at the cone point, Martelli-Sarks-Yau \cite{MSY} showed the existence of a $\mathbf{T}$-equivariant trivialization of $K_X$.  Their argument applies verbatim to the singular, $\mathbf{Q}$-Gorenstein case.  We include the short proof for the reader's convenience.

\begin{lem}\label{lem:MSYtriv}
There exists a unique, up to scale, non-vanishing section $\Omega \in \Gamma(X, mK_{X})$, and a unique linear function $\ell: \mathcal{C}_{R} \rightarrow \mathbf{R}_{>0}$ with the property that, for any Reeb field $\xi \in \mathcal{C}_{R}$ we have
\[
L_{\xi} \Omega= \ell(\xi) \Omega,
\]
In particular, for any $c>0$ the set $\{ \xi \in \mathcal{C}_{R} : \ell(\xi) =c \}$ defines an affine hyperplane in $\mathfrak{t}$ intersecting $\mathcal{C}_{R}$ in a set of codimension 1.
\end{lem}
\begin{proof}
Fix a non-vanishing section $\Omega \in \Gamma(X, mK_{X})$, and $\xi \in \mathcal{C}_{R}$.  Since $mK_X$ is trivial we have
\[
L_{\xi} \Omega= k(z) \Omega
\]
for some function $k(z) \in R$, the coordinate ring of $X$.  We decompose $k(z)$ into its graded pieces according to the grading defined by $\xi$:
\[
k(z) = \sum_{\alpha \in \mathfrak{t}^{*}} k_{\alpha}(z).
\]
This decomposition converges in $L^2_{loc}$ and therefore also locally
uniformly. Indeed, we can restrict
$k$ to the link of $X$, and consider its weight decomposition in $L^2$
for the torus action. 

Since $X$ has log-terminal singularities, we can assume that $k_{0} \in \mathbf{R}_{>0}$.  We project to the degree zero part to conclude.  Explicitly, define
\[
f(z) = \sum_{\alpha \in \mathfrak{t}^{*}-\{0\}} \frac{1}{\alpha(\xi)} k_{\alpha}(z)
\]
then $\tilde{\Omega} := e^{-f(z)}\Omega$ satisfies our requirements.
The last two claims are clear from the linearity of the projection and
the positivity of $k_{0}$.  The uniqueness follows from the fact that
$\xi$ induces a positive grading, and hence the only homogeneous,
holomorphic, non-vanishing holomorphic functions on $X$ are constant. 
\end{proof}

The above discussion makes it possible to introduce the gauge fixing
condition, as in Martelli-Sparks-Yau \cite{MSY}
\begin{defn}
We say that $\xi \in \mathcal{C}_{R}$ satisfies the gauge fixing condition, (or is normalized) if
\[
L_{\xi} \Omega = inm\Omega
\]
where $\Omega$ is as in Lemma~\ref{lem:MSYtriv}.
\end{defn}

An important point for us is that the gauge fixing condition can in fact be read off from the index character.  This observation is implicit in the work of Martelli-Sparks-Yau \cite{MSY} when $X$ is Gorenstein with an isolated log-terminal singularity.  We now extend this to the case of general $\mathbf{Q}$-Gorenstein affine varieties with log-terminal singularities.  To this end, fix a trivializing section $\Omega \in \Gamma(X, mK_{X})$ and suppose that $L_{\xi} \Omega= i\lambda\Omega$ for some $\lambda>0$.  Then we have an isomorphism of graded  $R$-modules
\begin{equation}\label{eq:qGor iso}
R(-\lambda) \simeq \Gamma(X, mK_{X})
\end{equation}
by the map $f \mapsto f\Omega$.  The index character of the ring $R$, denoted $F_{R}(\xi, t)$, expands as a Laurent series
\[
F_{R}(\xi, t) = \frac{a_{0}(\xi)(n-1)!}{t^{n}} + \frac{a_{1}(\xi)(n-2)!}{t^{n-1}} + O(t^{2-n}).
\]
We have the following
\begin{prop}\label{prop:a1a0formula}
In the above setting, the coefficients $a_{0}(\xi), a_{1}(\xi)$ satisfy
\[
\frac{a_{1}(\xi)}{a_{0}(\xi)} = \frac{\lambda(n-1)}{2m}.
\]
In particular, the set $\{\xi \in \mathcal{C}_{R}: 2a_1(\xi)=n(n-1)a_{0}(\xi)\}$ is an affine subset of $\mathcal{C}_{R}$ with codimension $1$, which agrees with the normalized Reeb fields.
\end{prop}
\begin{proof}
The proof is essentially a computation in commutative algebra.  Recall that a $\mathbf{Q}$-Gorenstein ring $R$ with log-terminal singularities is Cohen-Macaulay.   The main tool in the proof is a duality equation due to Stanley \cite{Stan}, \cite[Corollary 4.4.6]{BrunH},  which says that if $R$ is a Cohen-Macaulay, positively graded $\mathbf{C}$-algebra of dimension $n$ with canonical module $\Omega_{R}$, then the Hilbert series satisfies
\[
H_{\Omega_{R}}(s) = (-1)^nH_{R}(s^{-1})
\]
as rational functions, or in terms of the index character we have
\begin{equation}\label{eq:hilbStan}
F_{\Omega_{R}}(t) = (-1)^{n}F_{R}(-t).
\end{equation}
Roughly speaking this is a form of Serre duality.  Let us first explain the proof in the easier case that $R$ is Gorenstein, so that $m=1$.  Then the isomorphism in \eqref{eq:qGor iso} becomes
\[
R(-\lambda) \simeq \Omega_{R},
\]
and so, in particular $F_{\Omega_{R}}(t) = e^{-\lambda t}F_{R}(t)$.  Combining this with~\eqref{eq:hilbStan} gives
\[
F_{R}(t) = (-1)^ne^{\lambda t}F_{R}(-t).
\]
In terms of the index character this implies
\[
\frac{a_0(n-1)!}{t^n} + \frac{a_{1} (n-2)!}{t^{n-1}} = (1+\lambda t)\left(\frac{a_0 (n-1)!}{t^{n}} - \frac{a_{1}(n-2)!}{t^{n-1}}\right) + O(t^{2-n}).
\]
Comparing coefficients we get that
\[
2a_{1}(n-2)! = \lambda a_{0} (n-1)!
\]
which proves the proposition in the Gorenstein case.

We now consider the case when $X$ is a $\mathbf{Q}$-Gorenstein.  Since the coordinate ring $R$ is a Cohen-Macaulay integral domain, \cite[Proposition 3.3.18]{BrunH} says that there is a homogeneous ideal $I \subset R$ such that $\Omega_R \simeq I$ as $R$-modules.  In the current case $X$ is normal and affine, so the canonical sheaf is given by $K_{X} = i_{*}K_{U}$, where $i: U= X_{reg}\hookrightarrow X$ and $K_U$ is the canonical sheaf of $U$.  Then
\[
\Omega_{R} = \Gamma(X, K_{X})
\]
as $R$-modules. Thanks to the fact that  $X$ is $\mathbf{Q}$-Gorenstein we have
\[
I^{(m)} \simeq R \cdot \sigma \simeq R(-\lambda)
\]
where $\sigma$ is the trivializing section of $\Gamma(X, mK_{X})$, and $I^{(m)}$ denotes the $m-th$ symbolic power of $I$.  Since $I$ is a reflexive $R$-module of rank one, the symbolic power may be defined as
\[
I^{(m)} = \left(\overbrace{I \otimes_{R} I \cdots \otimes_{R} I}^{m -{\rm times}} \right)^{**} = \left(I^{\otimes m}\right)^{**},
\]
where if $M$ is an $R$-module, then $M^{*} = {\rm Hom}_{R}(M, R)$ denotes the dual.  Geometrically, we have an exact sequence of sheaves on $X$
\[
0\rightarrow \mathcal{K} \rightarrow K_{X}^{\otimes m} \rightarrow \left(K_{X}^{\otimes m}\right)^{**} \rightarrow \mathcal{Q}\rightarrow 0.
\]
Since $K_{X}$ is reflexive and locally free on $X_{reg}$, the sheaves $\mathcal{K}, \mathcal{Q}$ are supported on a subvariety of codimension at least $2$.  By Serre's criterion for affineness, we have $H^{1}(X,\mathcal{F})=0$ for any coherent sheaf $\mathcal{F}$.  In particular, by taking global sections, using \cite[Proposition 5.2]{Ha} we get an exact sequence of graded $R$-modules
\[
0 \rightarrow K:= \Gamma(X,\mathcal{K}) \rightarrow I^{\otimes m} \rightarrow I^{(m)} \rightarrow Q:= \Gamma(X, \mathcal{Q}) \rightarrow 0.
\] 
On the level of Hilbert series this implies that
\[
H_{I^{(m)}}(s) = H_{I^{\otimes m}}(s) + H_{Q}(s) - H_{K}(s).
\]
Since $\mathcal{Q}, \mathcal{K}$ are supported in codimension $2$, their associated index characters satisfy
\[
F_{Q}(t) = O(t^{2-n}), \qquad F_{K}(t) = O(t^{2-n}).
\]
We now consider the index character (or Hilbert series) of $I^{\otimes m}$.  We need the
following lemma.
\begin{lem}Suppose that $M, N$ are graded $R$-modules, with $M$ free
  in codimension one. Then
  \[ F_{M\otimes_R N}(t) = \frac{F_M(t)F_N(t)}{F_R(t)} +
  O(t^{2-n}).\]
\end{lem}
\begin{proof}
  Let us take a free resolution of $N$, i.e. a complex
  \[ 0 \to E_k \to E_{k-1} \to\ldots \to E_0 \to 0, \]
  whose only cohomology is $H^0 = N$. Tensoring with $M$ we obtain a
  complex
  \[ 0 \to M\otimes_R E_k \to \ldots \to M\otimes_R E_0 \to 0, \]
  whose cohomology is $H^i = \mathrm{Tor}^R_i(M, N)$. It is easy to check that
  the alternating sum of index characters of a complex is the same as
  the alternating sum of the index characters of its cohomology,
  i.e. we have
  \[ \sum_{i=0}^k (-1)^i F_{M\otimes_R E_i}(t) = \sum_{i=0}^k (-1)^i
  F_{ \mathrm{Tor}^R_i(M,N)}(t). \]
  Since $M$ is free in codimension 1, we have that
  $\mathrm{Tor}^R_i(M,N)$ is supported in codimension 2 for $i > 0$,
  i.e. its index character is of order $t^{2-n}$. It follows that
  \[ \begin{aligned}
    F_{M\otimes_R N}(t) &= \sum_{i=0}^k (-1)^i
    \frac{F_M(t)F_{E_i}(t)}{F_R(t)} + O(t^{2-n}) \\
    &= \frac{F_M(t)}{F_R(t)}\sum_{i=0}^k (-1)^i F_{E_i}(t) +
    O(t^{2-n}) \\
    &= \frac{F_M(t)F_N(t)}{F_R(t)} + O(t^{2-n}),
  \end{aligned}\]
  which is what we wanted to prove. 
\end{proof}

From this lemma a simple induction gives
\[
F_{I^{\otimes m}}(t) = \frac{F_{\Omega_{R}}(t)^m}{F_{R}(t)^{m-1}} + O(t^{2-n}).
\]
As remarked earlier, the $\mathbf{Q}$-Gorenstein assumption implies that $I^{(m)} \simeq R(-\lambda)$, and so
\begin{equation}\label{eq:hilbSeq}
e^{-t\lambda}F_{R}(t) = \frac{F_{\Omega_{R}}(t)^m}{F_{R}(t)^{m-1}}+ O(t^{2-n}).
\end{equation}
Expanding this equation to order $t^{1-n}$ we obtain
\[
\frac{a_{1}}{a_0} = \frac{\lambda(n-1)}{2m},
\]
as required.
\end{proof}
\subsection{The Futaki invariant}\label{sec:compareFutaki}
Suppose that $(X,\xi)$ is a normalized Fano cone singularity and we have a 
test-configuration $\lambda$ for $X$ with central fiber $Y$. Recall
that $\lambda$ gives a $\mathbf{C}^*$-action on $Y$ generated by a
vector field $w$. We have given two versions of the Futaki invariant
of this test-configuration. One was purely algebraic in terms of the
weights of the action $\lambda$ on the coordinate ring of $Y$, in
Definition~\ref{defn:Futaki}. As discussed below that definition, for
small $s$ we can consider the Reeb field $\xi + sw$ on $Y$, and we
can normalize the test-configuration (i.e. modify the vector field $w$
by adding a multiple of $\xi$) in such a way that 
\[ \frac{a_1(Y, \xi + sw)}{a_0(Y, \xi+sw)} = \frac{n(n-1)}{2}. \]
Under this normalization the Futaki invariant is given by
\begin{equation}
\label{eq:f1} \mathrm{Fut}(Y, \xi, w) = \frac{1}{2}D_w a_0(Y, \xi). 
\end{equation}
At the same time, from Proposition~\ref{prop:a1a0formula} we see that
this normalization is equivalent to requiring $L_w dV=0$, where $dV$
is the canonical volume form on $Y$. The differential geometric
definition of the Futaki invariant in that case is given in
Equation~\ref{eq:normalizedFut} by
\begin{equation}\label{eq:f2}
 \mathrm{Fut}_{Y,\xi}(w) = \frac{1}{V}\int_Y \theta_w
e^{-\frac{1}{2}r^2} \omega^n, 
\end{equation}
where $\omega = \frac{1}{2}\ddb r^2$ is a suitable reference metric
with Reeb field $\xi$, the function $\theta_w$ is the transverse
Hamiltonian of $w$ and  and $V$ is the volume of $(Y,\xi)$. To relate
this to the algebraic definition we have the following two results. 

\begin{prop}\label{prop:volumeformula}
  Suppose $X\subset \mathbf{C}^N$ is a polarized affine variety with
  Reeb field $\xi$, which has weights $w_i$ on the coordinates $z_i$,
  i.e. $\xi$ is the imaginary part of
  \[ 2\sum_{i=1}^N w_i z_i\frac{\del}{\del z_i}. \]
  Let
  \[ \hat{r}^2 = \sum_{i=1}^N |z_i|^{2/w_i}, \]
  and $\omega = \frac{1}{2}\ddb \hat{r}^2$. We have
  \begin{equation}\label{eq:a0eqint}
    a_0(n-1)! = \frac{1}{(2\pi)^n} \int_X e^{-\frac{1}{2}\hat{r}^2}
  \frac{\omega^n}{n!}, 
  \end{equation}
  where $a_0$ is defined by the index character
  \[ F(\xi, t) = \frac{a_0 (n-1)!}{t^n} + O(t^{-n+1}). \]
\end{prop}
\begin{proof}
  We can choose a generic $\mathbf{C}^*$-action $\lambda(t)$ commuting
  with $\xi$, and degenerate $X$ to $Y = \lim_{t\to 0} \lambda(t)\cdot
  X$. This will not affect the integral in \eqref{eq:a0eqint}, since
  the integral over $\lambda(t)\cdot X$ is the same as the integral
  over $X$ using a different metric. The volume, however is a function
  of just the Reeb field. At the same time the index character is also
  unchanged in passing to the limit since $Y$ is a flat limit. 

  For a generic $\mathbf{C}^*$-action the
  top dimensional part of $Y$ is a union of $n$-dimensional
  coordinate subspaces with multiplicity. The leading term in the
  index character will be the sum of the corresponding terms for these
  subspaces, with multiplicity, and the limiting integral on $Y$ is
  also given by a corresponding sum. As such, we only need to check
  the formula on $\mathbf{C}^n$ for a given Reeb field. 
  But both the index character and the integral is multiplicative when
  taking products of varieties, so it is enough to do the calculation
  for $\mathbf{C}$, with a Reeb field
  \[ \xi = \mathrm{Im} \left(w z\frac{\del}{\del z}\right), \]
  with corresponding radial function $\hat{r}^2=|z|^{2/w}$. The metric is
  then
  \[ \omega = \frac{i}{2} \cdot\frac{1}{w^2} |z|^{2/w-2}\, dz\wedge
  d\bar{z}, \]
  so a calculation gives
  \[ \int_{\mathbf{C}} e^{-\frac{1}{2}\hat{r}^2} \omega =
  \frac{2\pi}{w}. \]
  At the same time the index character is
  \[ \sum_{k=0}^\infty e^{-tkw} = \frac{1}{1-e^{-tw}} = \frac{1}{tw} +
  O(1), \]
  from which the result follows. 
\end{proof}

We also have the following formula for the variation of the volume as
we vary the Reeb field. 

\begin{prop}\label{prop:volumevariation}
  Suppose that we consider a variation $\delta\xi = w$ of the Reeb
  field. The corresponding variation in the volume
  \[ V(\xi) = \int_X e^{-\frac{1}{2}r^2}\,\frac{\omega^n}{n!} \]
  is given by
  \[ \delta V(\xi) = n\int_X \theta_w
  e^{-\frac{1}{2}r^2}\,\frac{\omega^n}{n!}. \]
\end{prop}
\begin{proof}
  This result was shown by Martelli-Sparks-Yau ~\cite{MSY}, see also
  Donaldson-Sun~\cite{DS2}. In comparing these formulas recall that by
  our convention $\theta_\xi = -1$. The result and its proof are valid even if $X$ is not
  normal, interpreting the integral as just a sum of integrals on the
  $n$-dimensional components of $X$, with multiplicity. 
\end{proof}

Using these results we can now compare the definitions \eqref{eq:f1}
and \eqref{eq:f2} to see that the two Futaki invariants agree up to a
dimensional constant, obtaining the following.
\begin{prop}\label{prop:equalFut}
  We have $\mathrm{Fut}_{X, \xi}(w) = c(n) \mathrm{Fut}(X,\xi,
  \lambda)$ in terms of Definition~\ref{defn:Futaki}, where $\lambda$
  is the $\mathbf{C}^*$-action generated by $w$, and $c(n) > 0$ is  a
  dimensional constant. 
\end{prop}

\section{K-stability of affine varieties with Ricci-flat K\"ahler cone metrics}\label{sec:Kstab}

The main theorem of this section is

\begin{thm}\label{thm:Kpolystab}
Suppose that $(X,\xi)$ admits a Ricci-flat K\"ahler cone metric.  Then $(X,\xi)$ is $K$-stable.
\end{thm}

The idea of the proof follows work of Berman \cite{RBer}, and goes as follows.  First, we will show that the Ding functional is convex along (sub)-geodesics.  Geodesics in the space of Sasakian metrics have been studied by Guan-Zhang \cite{GZ11}, but we will need a slightly different formulation than the one given there.  Since the Ricci-flat K\"ahler cone metric is a critical point of the Ding functional, the strict convexity along geodesics means that, along a (sub)-geodesic $\phi_s$ emanating from a Sasaki-Einstein potential $\phi_{SE}$ we have
\[
\frac{d}{ds} \mathcal{D}(\phi_s) \geq 0
\]
and the limit slope $\lim_{s\rightarrow \infty} \frac{d}{ds} \mathcal{D}(\phi_s)$ exists in $[0, \infty]$.  Furthermore, a result of Berndtsson \cite{Ber2} says that the limit must be strictly positive unless the geodesic was generated by a real holomorphic vector field (see \cite{DS2} for the generalization of Berndtsson's result to our setting).  Next, we show that any special degeneration gives rise to a (sub)-geodesic, and we show that the limit slope is precisely the Futaki invariant.

Let $(X,\xi)$ be a polarized cone, which we assume is $\mathbf{Q}$-Gorenstein and log-terminal.  Recall that we have defined the Ding functional
\[
\mathcal{D}(\phi) = -E(\phi) - \frac{1}{2n} \log\int_{X}e^{-\frac{1}{2}r_{\phi}^2 }dV
\]
where $dV = \left(\Omega \wedge \overline{\Omega}\right)^{1/m}$ and
$\Omega$ is a $T$-equivariant trivialization of $mK_{X}$,
$r:X\rightarrow \mathbf{R}_{+}$ is a radial function compatible with
$\xi$, and $r_{\phi} = e^{\phi}r$ where $\phi$ is basic, and
independent of $r$.  As before, the function $E$ is defined by its variation
\[
\delta E(\phi) = \frac{1}{V(\xi)} \int_{X}\dot{\phi}e^{-\frac{1}{2}r_{\phi}^2} \omega_{\phi}^{n}.
\]
Our goal is to compute the second variation of $E(\phi_s)$.  
For our computation, we will assume that $\phi_s$ is a {\em smooth} variation. Dropping the $V(\xi)$ term for convenience an easy computation shows that
\[
-\frac{d}{ds}E(\phi) = \frac{1}{2n}\int_{X}\dot{\phi}r_{\phi}^2e^{-\frac{1}{2}r_{\phi}^2} \omega_{\phi}^{n} =  \frac{1}{2n}\int_{X}\dot{\left(\frac{r_{\phi}^2}{2}\right)}e^{-\frac{1}{2}r_{\phi}^2} \omega_{\phi}^{n}
\]
Let us suppress the dependence on $\phi$ to ease notation.  Then we have
\begin{equation}\label{eq:2varEnergy}
\begin{aligned}
-2^{n+2}n\frac{d^2}{ds^2} E &=\int_{X} \left[\ddot{r^2} - \frac{(\dot{r^2})^2}{2}\right]e^{-\frac{1}{2}r^2} (\ddb r^2)^n\\
&\quad + n\int_{X}\dot{r^2}e^{-\frac{1}{2}r^2} \ddb \dot{r^2}\wedge (\ddb r^2)^{n-1}.
\end{aligned}
\end{equation}
Let us manipulate the last term.  Integrating by parts gives
\[
\begin{aligned}
 &n\int_{X}\dot{r^2}e^{-\frac{1}{2}r^2} \ddb \dot{r^2}\wedge (\ddb r^2)^{n-1}\\
 &= -n\int_{X}e^{-\frac{1}{2}r^2}\sqrt{-1}\del \dot{r^2} \wedge \dbar \dot{r^2} \wedge (\ddb r^2)^{n-1}\\
 &\quad+\frac{n}{2}\int_{X} \dot{r^2}e^{-\frac{1}{2}r^2}\sqrt{-1}\del r^2 \wedge \dbar \dot{r^2}\wedge (\ddb r^2)^{n-1}
\end{aligned}
\]
We focus now on the second term of this expression.  Integration by parts gives
\[
\begin{aligned}
&\frac{n}{2}\int_{X} \dot{r^2}e^{-\frac{1}{2}r^2}\sqrt{-1}\del r^2 \wedge \dbar \dot{r^2}\wedge (\ddb r^2)^{n-1}\\
&= \frac{n}{2} \int_{X} \dot{r^2}e^{-\frac{1}{2}r^2}\sqrt{-1}\, \dbar{\dot{r^2}}\wedge\del r^2 \wedge (\ddb r^2)^{n-1}\\
&\quad - \frac{n}{4} \int_{X} (\dot{r^2})^2e^{-\frac{1}{2}r^2}\sqrt{-1}\, \dbar r^2\wedge\del r^2 \wedge (\ddb r^2)^{n-1}\\
&\quad+\frac{n}{2} \int_{X}
(\dot{r^2})^2e^{-\frac{1}{2}r^2}\sqrt{-1}\, \dbar\del r^2 \wedge (\ddb
r^2)^{n-1} \\
& = -\frac{n}{2} \int_{X} \dot{r^2}e^{-\frac{1}{2}r^2}\sqrt{-1} \del{r^2}\wedge\dbar\dot{ r^2} \wedge (\ddb r^2)^{n-1}\\
&\quad+ \frac{n}{4} \int_{X} (\dot{r^2})^2e^{-\frac{1}{2}r^2}\sqrt{-1} \del r^2\wedge\dbar r^2 \wedge (\ddb r^2)^{n-1}\\
&\quad -\frac{n}{2} \int_{X} (\dot{r^2})^2e^{-\frac{1}{2}r^2}(\ddb r^2)^{n}.
\end{aligned}
\]
Thus, we get
\[
\begin{aligned}
&\frac{n}{2}\int_{X} \dot{r^2}e^{-\frac{1}{2}r^2}\sqrt{-1}\del r^2 \wedge \dbar \dot{r^2}\wedge (\ddb r^2)^{n-1} \\
&= \frac{n}{8} \int_{X} (\dot{r^2})^2e^{-\frac{1}{2}r^2}\sqrt{-1} \del r^2\wedge\dbar r^2 \wedge (\ddb r^2)^{n-1}-\frac{n}{4} \int_{X} (\dot{r^2})^2e^{-\frac{1}{2}r^2}(\ddb r^2)^{n}
\end{aligned}
\]
Now, since $\omega^{\tau} = \ddb \log r^2$ satisfies $(\omega^{\tau})^{n}=0$, a direct computation shows that
\[
\del r^2\wedge\dbar r^2 \wedge (\ddb r^2)^{n-1} = \frac{r^2}{n}(\ddb r^2)^n.
\]
From this observation, an easy computation shows that
\[
\begin{aligned}
\frac{n}{8} \int_{X} (\dot{r^2})^2e^{-\frac{1}{2}r^2}\sqrt{-1} \del r^2\wedge\dbar r^2 \wedge (\ddb r^2)^{n-1} &= \frac{1}{8} \int_{X} (\dot{r^2})^2e^{-\frac{1}{2}r^2}r^2(\ddb r^2)^{n}\\
&= \frac{n+2}{4}\int_{X}(\dot{r^2})^2e^{-\frac{1}{2}r^2}(\ddb r^2)^n,
\end{aligned}
\]
and hence
\[
\frac{n}{2}\int_{X} \dot{r^2}e^{-\frac{1}{2}r^2}\sqrt{-1}\del r^2 \wedge \dbar \dot{r^2}\wedge (\ddb r^2)^{n-1}= \frac{1}{2}\int_{X}(\dot{r^2})^2e^{-\frac{1}{2}r^2}(\ddb r^2)^n.
\]
Plugging this back into the equation~\eqref{eq:2varEnergy} we get
\[
-2^{n+2}n\frac{d^2}{ds^2} E= \int_{X}e^{-\frac{1}{2}r^2}[\ddot{r^2}\ddb r^2-n\sqrt{-1}\del\dot{r^2}\wedge \dbar\dot{r^2}]\wedge (\ddb r^2)^{n-1}.
\]
A geodesic is a path $\{\phi_s\}$ along which the function $E$ is affine.  In particular, we can write the geodesic equation as
\[
[\ddot{r^2}\ddb r^2-n\sqrt{-1}\del\dot{r^2}\wedge \dbar\dot{r^2}]\wedge (\ddb r^2)^{n-1}=0.
\]
We will also make use of subgeodesics, along which the functional $E$ is concave; that is, paths $\phi_s$ satisfying
\[
[\ddot{r^2}\ddb r^2-n\sqrt{-1}\del\dot{r^2}\wedge \dbar\dot{r^2}]\wedge (\ddb r^2)^{n-1} \geq 0.
\]
Another standard computation shows that if we introduce a holomorphic coordinate $\tau$ with $|\tau| = e^{-s}$, then the (sub)-geodesic equation can be written as
\[
(\sqrt{-1}D\overline{D} r^2)^{n+1}= 0
\]
where now $D,\overline{D}$ denote the $\del, \dbar$ operators in the variables $\tau, z$ jointly.  In particular, a subgeodesic is nothing but a family of radial functions $r(z,\tau): X \rightarrow \mathbf{R}_{+}$, all compatible with $\xi$, and such that $\sqrt{-1}D\overline{D} r^2 \geq0$.  Furthermore, the above description of the geodesic equation makes it clear that we can produce (weak) geodesics using the standard techniques of envelopes and subsolutions, but we will not need this here.

In order to obtain the convexity of the Ding functional, we also need the convexity of the second term in its definition.  Berndtsson's theorem \cite{Ber2} gives this and even more;  we refer the reader to \cite{DS2} for an extension of Berndtsson's theorem this to our setting.

\begin{prop}
Let $r(x,\tau) : X\times \Delta^{*} \rightarrow \mathbf{R}_{>0}$ be a path of radial functions compatible with $\xi$, and $S^{1}$ invariant.  Suppose that $\sqrt{-1}D\overline{D} r \geq 0$, and that $\mathcal{D}(r_s)$ is affine ($s = -\log|\tau|$).  Then there exists a holomorphic vector field $\Xi$ on $X$, commuting with $\xi, r\del_r$ so that $r_s = F_{s}^{*}r(x,0)$ where 
\[
F_s = \exp(s {\rm Re}(\Xi)).
\]

\end{prop}

We now explain how a special degeneration gives rise to a subgeodesic.  Let $T\subset Aut(X)$ be a torus containing $\xi$.  Recall that a special degeneration consists of an embedding $X\rightarrow \mathbf{C}^{N}$, which we may assume is not contained in a linear subspace, and such that $T\subset {\rm Aut}(X)$ acts linearly and diagonally through and embedding $T \subset U(N)$, together with a one-parameter subgroup $\lambda: \mathbf{C}^{*}\rightarrow GL(N)^{T}$ commuting with $T$, and such that $\lambda(S^1) \subset U(N)$, and $Y = \lim_{t\rightarrow 0} \lambda(t)\cdot X$ is normal.  We may package this as an affine scheme $\mathcal{Y} \subset \mathbf{C}^{N} \times \mathbf{C}$, together with a $\mathbf{C}^{*}$ equivariant projection
\[
\pi: \mathcal{Y}\rightarrow \mathbf{C}
\]
where the $\mathbf{C}^{*}$ action is by $\lambda$; we will usually
restrict our attention to $\pi^{-1}(\Delta)$ where $\Delta$ is the
closed unit disk.  Abusing notation, we will also denote this by $\mathcal{Y}$.  By definition $\xi \in \mathfrak{u}(N)$ induces a Reeb vector field on $\mathbf{C}^{N}$, and hence we may find $r_0: \mathbf{C}^{N} \rightarrow \mathbf{R}_{+}$, a $U(N)$ invariant radial function compatible with $\xi$.  Let $p_1: \mathbf{C}^{N} \times \mathbf{C} \rightarrow \mathbf{C}^{N}$, and consider $p_1^{*}r_0 : \mathbf{C}^{N}\times \mathbf{C} \rightarrow \mathbf{R}_{+}$.   The $\mathbf{C}^{*}$ action allows us to identify  $\mathcal{Y}^{*} := \pi^{-1}(\Delta^{*})$ with $X\times \Delta^{*}$, and hence $p_1^{*}r_0$ induces a radial function $r(\tau)$ compatible with $\xi$ on each fiber of $X\times \Delta^{*}$.  By the $U(N)$ invariance of $r_0$, the function $r(\tau)$ is $S^{1}$-invariant, and since the map $X\times \Delta^{*} \rightarrow \mathcal{Y}^{*}$ is holomorphic we have
\[
\sqrt{-1}D\overline{D} r(\tau) \geq 0.
\]
We can therefore write
\[
r(\tau) = r_0e^{\psi(|\tau|)}.
\]
Let $\phi$ be any potential of a radial function on $X$.  By taking $A, C$ large and $\epsilon$ small we define
\[
\Phi(z, \tau) = \Phi(z, |\tau|) := \begin{cases} \widetilde{max}_{\epsilon} \{ \phi(z) +A\log|\tau|, \psi(z, |\tau|)-C \} & |\tau| \geq \frac{1}{2}\\
					\psi(z, |\tau|)-C & |\tau|\leq \frac{1}{2}
					\end{cases}
\]
where $\widetilde{\max}_{\epsilon}$ is the regularized maximum
\cite[Section 5.E.]{DemBook}.  A few words are in order about how to choose $A, C$.  First, we choose $C\gg 1$ large so that $\psi(z,\tau)-C < \phi(z) - 100$ on $X\times\{1\}$.  Next we choose $A$ large so that
\[
 \phi(z) +A\log|\tau| \leq \psi(z,|\tau|)-C -100
 \]
 for all $\frac{1}{2} <|\tau| < \frac{3}{4}$, and finally choose $0<\epsilon \ll 1$.  Clearly $A, C$ exist since $\phi, \psi(\tau)$ are smooth on $X$, and uniformly bounded for $\tau$ in any compact subset of $\Delta^{*}$.  By our choices and the properties of the regularized maximum we obtain that for every $\tau \in \Delta^{*}$,  $r_{\Phi} := re^{\Phi(\tau)}$ defines a smooth radial function on $X$ compatible with $\xi$, and furthermore, $re^{\Phi(\tau)}$ is plurisubharmonic on $X\times \Delta^{*}$.  In particular, $\Phi(\tau)$ defines a subgeodesic emanating from $\phi$.

We now compute the limit slope of the Ding functional along $\Phi(\tau)$.  Note that $\Phi(\tau)$ for $|\tau|<\frac{1}{2}$ depends only on $r_0$ and the test configuration, and not on the initial data.  Since we are computing the limit
\[
\lim_{s\rightarrow \infty}\frac{d}{ds}\mathcal{D}(\Phi(e^{-s}))
\]
we can assume that in fact $\Phi(\tau) = \psi(\tau)$.  Before proceeding, we need a preliminary lemma.

\begin{lem}
The total space of the special degeneration $\mathcal{Y}$ is a polarized cone.  In particular, it is $\mathbf{Q}$-Gorenstein, with log-terminal singularities and admits a Reeb vector field.
\end{lem}
\begin{proof}
That $\mathcal{Y}$ is $\mathbf{Q}$-Gorenstein and log-terminal follow from the fact that $\mathcal{Y}$ is flat over $\mathbf{C}$ and every fiber is normal, $\mathbf{Q}$-Gorenstein with log-terminal singularities. We only need to show that $\mathcal{Y}$ has a Reeb vector field.  Let $\eta$ be the generator of the $\mathbf{C}^{*}$ action defining the test configuration.  By definition $\eta$ acts on the coordinate $t$ on $\mathbf{C}$ with weight one, and commutes with $\xi$. Hence for $s$ sufficiently small, $\xi+s\eta$ is a Reeb vector field for $\mathcal{Y}$.
\end{proof}

Let $T'$ be the torus in ${\rm Aut}(\mathcal{Y})$ containing $T$, and
$\eta$, the generator of the $\mathbf{C}^{*}$ action defining the
special degeneration.  It follows from Lemma~\ref{lem:MSYtriv} that we
can choose $\widehat{\Omega}$ a $T'$ equivariant trivializing section
of $mK_{\mathcal{Y}}$ so that $\iota_{\left(\frac{\del}{\del
      \tau}\right)^{\otimes m}}\widehat{\Omega}$ is a $T$-equivariant
trivialization $\Omega$ of $mK_{X_t}$ for all $t\in \mathbf{C}$, where
$X_{t} = \pi^{-1}(t)$.  
By the uniqueness part of Lemma~\ref{lem:MSYtriv} we must have
\[
(\lambda^{-1}(\tau))^{*}\Omega = c(\tau)\iota_{\left(\frac{\del}{\del \tau}\right)^{\otimes m}}\widehat{\Omega}
\]
where $c(\tau)$ is non-vanishing holomorphic function, constant on the fibers.  In particular, on the level of volume forms we have
\[
(\lambda^{-1}(\tau))^{*}dV = |\hat{c}(\tau)|^{2/m}dV_{X_{\tau}}  := |c(\tau)|^{2/m} \left(\iota_{\left(\frac{\del}{\del \tau}\right)^{\otimes m}}\widehat{\Omega}\wedge \overline{\iota_{\left(\frac{\del}{\del \tau}\right)^{\otimes m}}\widehat{\Omega}}\right).
\]
We now compute the limit slope of the Ding functional.  As explained above, it suffices to compute the limit slope of $\mathcal{D}(r_s)$, where $s= -\log|\tau|$, and $r_s = \lambda(e^{-s})^{*}p_1^{*}r_0$.  Recall that
\[
\mathcal{D}(r_s) = -E(r_s) - \frac{1}{2n} \log\int_{X}e^{-\frac{1}{2}r_{s}^2 }dV.
\]
Let us first focus on the $E$ term.  By definition we have
\[
-\frac{d}{ds}E(r_s) =  -\frac{1}{V(\xi)} \int_{X}\frac{d}{ds}\log\left(\frac{r_s}{r_0}\right)e^{-\frac{1}{2}r_{s}^2} \omega_{s}^{n}
\]
We now use the biholomorphism $\lambda(e^{-s})$ to push the integral
forward to $X_{e^{-s}}$.  Note that
\[
\lambda(e^{-s})_{*}\frac{d}{ds} \log\frac{\lambda(e^{-s})^{*}r_0}{r_0} = -\theta_{\lambda}
\]
where $\theta_{\lambda}$ is the Hamiltonian function,  with respect to $r_0$, of the real part of the holomorphic vector field generating the action of $\lambda$ on $\mathbf{C}^{N}$.  Thus we have
\[
-\frac{d}{ds}E(r_s) = \frac{1}{V(\xi)}\int_{X_{e^{-s}}} \theta_\lambda e^{-\frac{1}{2}r^2}\omega^n.
\]
Since $\mathcal{Y}$ is flat over $\mathbf{C}$, the current of integration $[X_{e^{-s}}]$ converges to $[X_0]$ weakly and we obtain
\[
\lim_{s\rightarrow \infty} -\frac{d}{ds}E(r_s) = \frac{1}{V(\xi)}\int_{X_0} \theta_\lambda e^{-\frac{1}{2}r^2}\omega^n,
\]
which is justified by the weak convergence since $\theta_{\lambda}e^{-\frac{1}{2}r^2}\omega^n$ is a smooth $(n,n)$ form defined on the ambient space $\mathbf{C}^{N}$.

We now compute the contribution of the second term.  Specifically, we are computing
\[
\begin{aligned}
-&\frac{d}{ds}\frac{1}{2n} \log\int_{X}e^{-\frac{1}{2}r_{s}^2 }dV\\
& = \frac{1}{2n}\frac{\int_{X}r_{s}^2\frac{d}{ds}\log\left(\frac{r_s}{r}\right)e^{-\frac{1}{2}r_s^2} dV}{\int_{X}e^{-\frac{1}{2}r_s^2}dV}
\end{aligned}
\]
Pulling this back to $X_{\tau}$ and using that $(\lambda^{-1}(\tau))^{*}dV = |c(\tau)|^{2/m}dV_{X_{\tau}}$ we get
\[
-\frac{d}{ds}\frac{1}{2n} \log\int_{X}e^{-\frac{1}{2}r_{s}^2 }dV =  -\frac{1}{2n}\frac{\int_{X_{\tau}}r^2\theta_{\lambda}e^{-\frac{1}{2}r^2} dV_{\tau}}{\int_{X_{\tau}}e^{-\frac{1}{2}r^2}dV_{\tau}}.
\]
Since that $\mathcal{L}_{\xi}dV_{\tau} = \sqrt{-1}2n dV_{\tau}$ and $\theta_{\lambda}$ is basic we obtain
\[
\int_{X_{\tau}}r^2\theta_{\lambda}e^{-\frac{1}{2}r^2} dV_{\tau} = 2n\int_{X_{\tau}}\theta_{\lambda}e^{-\frac{1}{2}r^2} dV_{\tau}.
\]
Putting everything together we have
\[
-\frac{d}{ds}\frac{1}{2n} \log\int_{X}e^{-\frac{1}{2}r_{s}^2 }dV =  -\frac{\int_{X_{\tau}}\theta_{\lambda}e^{-\frac{1}{2}r^2} dV_{\tau}}{\int_{X_{\tau}}e^{-\frac{1}{2}r^2}dV_{\tau}}.
\]
Now, by definition $\widehat{\Omega}$ is a non-vanishing, holomorphic section of $mK_{\mathcal{Y}}$, which is in particular smooth on $\mathcal{Y}_{reg}$.  Thanks to the fact that $X_0 = \pi^{-1}(0)$ is reduced and normal, Hartog's theorem implies that
\[
\iota_{\left(\frac{\del}{\del \tau}\right)^{\otimes m}}\widehat{\Omega}\big|_{\tau=0} = c_0 \Omega_0
\]
where $\Omega_0$ is the unique (up to scale) $T'$-equivariant trivialization of $mK_{X_0}$, and $c_0$ is a non-zero constant.  In particular, it follows that
\[
dV_{\tau} \rightarrow |c_0|^{2/m} \left(\Omega_0 \wedge \overline{\Omega}_0\right)^{1/m}
\]
smoothly on $X_{0,reg}$.  Furthermore, by the log-terminal assumption, $X_{\tau,sing}$ and $X_{0,sing}$ have zero volume with respect to $dV_{\tau}, dV_{0}$ respectively.  Finally, since $\theta_{\lambda}, r$ are smooth functions on $\mathbf{C}^{N}$, flatness implies
\[
\lim_{\tau \rightarrow 0}  \frac{\int_{X_{\tau}}\theta_{\lambda}e^{-\frac{1}{2}r^2} dV_{\tau}}{\int_{X_{\tau}}e^{-\frac{1}{2}r^2}dV_{\tau}} =  \frac{\int_{X_{0}}\theta_{\lambda}e^{-\frac{1}{2}r^2} dV_{0}}{\int_{X_{0}}e^{-\frac{1}{2}r^2}dV_{0}}.
\]
Putting everything together we get
\[
\lim_{s\rightarrow \infty} \frac{d}{ds} \mathcal{D}(r_s) = \int_{X_0} \theta_\lambda e^{-\frac{1}{2}r^2}\omega^n -\frac{\int_{X_0}\theta_{\lambda}e^{-\frac{1}{2}r^2} dV}{\int_{X_{0}}e^{-\frac{1}{2}r^2}dV}.
\]
By Proposition~\ref{prop:equalFut} this is (up to a positive constant $c(n)$), the algebraic Futaki invariant of the test configuration $(\mathcal{Y},\lambda)$.  From the convexity of the Ding functional we conclude
\[
Fut(\mathcal{Y},\lambda,\xi) \geq 0.
\]
If $Fut(\mathcal{Y},\lambda,\xi) =0$, then we must have that that $\frac{d}{ds}\mathcal{D}(r_s)=0$ identically, and then Berndtsson's theorem implies that $r_s=F_{s}^{*}r_0 $ on $X$, where
\[
F_s = \exp\left(sV\right)
\]
and $V$ is the real part holomorphic vector field on $X$ commuting with $\xi$.  Consider the map
\[
\rho := \lambda(\tau) \circ F^{-1}_{\tau} : X \times \Delta^* \rightarrow \mathcal{Y}^{*},
\]
and $\rho_\tau = \rho(\cdot, \tau)$ for $\tau\in \Delta^*$. 
By definition we have $\rho_{\tau}^{*}r_0 = r_0$, and
$\rho_{\tau}\,_{*} \xi = \xi$.  Since $r_0$ is the potential for a
K\"ahler cone metric on $\mathbf{C}^{N}$ compatible with $\xi$, this
implies that for any compact set $K\subset X$, the image $\rho( K
\times (\frac{1}{2}\overline{\Delta}\setminus\{0\}))$ is compact in
$\mathcal{Y}$.  By Riemann's extension theorem $\rho$ extends to a map
\[
\rho : X\times \Delta \rightarrow \mathcal{Y}
\]
which is an isomorphism away from $\tau=0$.  The same argument applied
to $\rho^{-1}$ shows that $\rho:X\times \Delta \rightarrow
\mathcal{Y}$ is an isomorphism, and so 
$\mathcal{Y}$ is a trivial test-configuration.  This completes the proof of Theorem~\ref{thm:Kpolystab}.

\section{Examples and Applications}\label{sec:examples}
In this section we check $K$-stability for a family of hypersurfaces of dimension 3 admitting a $2$ torus action.  Rational hypersurface singularities in $\mathbf{C}^{4}$ admitting a $\mathbf{C}^{*}$ action were classified by Yau-Yu \cite{YY}.  In the terminology of the Yau-Yu classification, will study the links of type $I-III$ admitting a $\mathbf{T}^2$ action.  For this section we will describe a holomorphic vector field in terms of its $S^1$ action by specifying the weights;  in particular, a vector $(a_1,a_2,a_3,a_4)$ should be understood to act on $\mathbf{C}^2_{(z_1,z_2, z_3, z_4)}$ by acting on $z_i$ with weight $a_i$.  With this notation, the main theorem of this section is
\begin{thm}
The following affine affine varieties admit conical Ricci flat K\"ahler metrics with respect to the given Reeb field:
\begin{itemize}
\item[(I)]  The Brieskorn-Pham singularity 
\[
\begin{aligned}
Z_{BP}(p,q) &:= \{ uv+z^p+w^q=0 \} \subset \mathbf{C}^{4}_{(u,v,z,w)}\\
\xi &= \frac{3}{2(p+q)}(pq,pq,2q,2p)
\end{aligned}
\]
if $2p>q$ and $2q>p$.  Topologically, the link is $\#m(S^2 \times S^3)$, where $m= \gcd(p,q)-1$ and $\#0(S^2 \times S^3)= S^5$.
\item[(II)]  The Yau-Yu singularities of type $II$
\[
\begin{aligned}
Z_{II}(p,q)&:=\{uv+ z^p+zw^q=0\} \subset \mathbf{C}^{4}_{(u,v,z,w)}\\
\xi &= \frac{3}{2(q+p-1)}(qp,qp,2q,2(p-1))
\end{aligned}
\]
provided $3(p-1) > (q+p-1)$ and $2qp+1>p^2+q$.  Topologically, the link is $\#m(S^2 \times S^3)$, where $m=\gcd(p-1,q)$.
\item[(III)] The Yau-Yu singularities of type $III$
\[
\begin{aligned}
Z_{III}(p,q) &:= \{ uv+ z^pw + zw^q=0\} \subset \mathbf{C}^{4}_{(u,v,z,w)}\\
\xi &= \frac{3}{2(p+q-2)}(pq-1,pq-1,2(q-1),2(p-1))
\end{aligned}
\]
provided $3(p-1)^2(q-1) > (p+q-2)(pq-2p+1)$, and $3(q-1)^2(p-1) > (p+q-2)(pq-2q+1)$.  Topologically, the link is $\#m(S^2\times S^3)$ where $m= \gcd(p-1,q-1)+1$.
\end{itemize}
In particular, there exists:
\begin{itemize}
\item One infinite family of inequivalent Sasaki-Einstein metrics on $S^5$, 
\item Two distinct infinite families of inequivalent Sasaki-Einstein metrics on $S^2 \times S^3$,
\item Three distinct families of inequivalent Sasaki-Einstein metrics on $\#m(S^2 \times S^3)$ for any $m \geq 2$.
\end{itemize}
\end{thm}
  The proof of this theorem will occupy the remainder of this section.  We use the theory of polyhedral divisors.  The original paper in this area is the work of Altmann-Hausen \cite{AH}, and there is a nice survey by Altmann-Ilten-Petersen-S\"u\ss\,  \cite{AIPS}. For applications of this theory to K-stability of Fano manifolds, see Ilten-S\"u{\ss} \cite{IS}. 
 \subsection{The Brieskorn-Pham singularities}
 We begin with the Brieskorn-Pham (BP) singularities, denoted $Z_{BP}(p,q)$ above, which appear as type  I singularities in the Yau-Yu classification.  
In order to avoid the trivial cases, we will assume that $\max\{p,q\} >2$. Let $\gcd (p,q)=m$, and choose relatively prime integers $a,b \in \mathbf{Z}$ such that
\[
a\frac{q}{m}-b\frac{p}{m}=1.
\] 
The affine variety admits $2$ torus action, which is generated by $\mathbf{C}^{*}$ actions with weights $(1,-1,0,0)$, and $(0,\frac{pq}{m}, \frac{q}{m}, \frac{p}{m})$ on $(u,v,z,w)$ respectively.  Let $\mathfrak{t}$ denote the Lie algebra of the compact torus, equipped with a basis $\{e_1, e_{2}\}$ corresponding to the above actions.  The Reeb cone inside of $\mathfrak{t}$ is given by
\[
\mathcal{C}_{\mathcal{R}} =\left\{(x,y) := xe_1+ye_2 \big|\,\, x>0,\quad xpq-ym>0\right\}.
\]
A simple symmetry argument shows that the Reeb field minimizing the volume of the link is a multiple of $(2,pq)$, which corresponds to the $\mathbf{C}^{*}$ action $(pq,pq,2q,2p)$.  Let $F :\mathbf{Z}^{2} \rightarrow \mathbf{Z}^4$ be the inclusion of the algebra $\mathfrak{t}$ into the lie algebra of the diagonal torus acting on $\mathbf{C}^{4}$, equipped with the standard basis.  With these choices, $F$ is represented by the matrix
\[
F = \begin{bmatrix}
0 &1\\
\frac{pq}{m} &-1\\
\frac{q}{m} & 0\\
\frac{p}{m}&0
\end{bmatrix}.
\]
Let $P: \mathbf{Z}^{4} \rightarrow \mathbf{Z}^{2}$ be the orthogonal projection to the cokernel of $F$.  We have
\[
P = \begin{bmatrix}
0&0&\frac{-p}{m}&\frac{q}{m}\\
1&1&-p&0
\end{bmatrix}.
\]
So we have an exact sequence
\[
0 \rightarrow \mathbf{Z}^2 \stackrel{F}{\longrightarrow} \mathbf{Z}^{4} \stackrel{P}{\longrightarrow} \mathbf{Z}^2 \rightarrow 0.
\]
we choose a splitting $s:\mathbf{Z}^{4}\longrightarrow \mathbf{Z}^2$ such that $s \circ F = 1$,
\[
s = \begin{bmatrix}
0&0&a&-b\\
1&0&0&0
\end{bmatrix}.
\]
In the language of \cite{AH}, the tail cone is given by
\[
\sigma := s(F(\mathbf{Q}^{2}) \cap \mathbf{Q}^{4}_{\geq 0}) = \text{cone} \left[\begin{pmatrix}1\\0\end{pmatrix}, \begin{pmatrix}1\\pq/m\end{pmatrix}\right].
\]
It is straightforward to check that $\sigma = \overline{\mathcal{C}_{\mathcal{R}}}$ is the closure of the Reeb cone.  The $2$ torus action on $\mathbf{C}^{4}$ induces a fibration of $\mathbf{C}^4$ over a surface whose fibers are non-compact toric surfaces.  Using the techniques of \cite{AH} we can describe this fibration in terms of combinatorial data called a p-divisor.  We first compute the base of the fibration, which is given by the fan $\Sigma_{Y}$ with maximal cones
\[
\text{cone}\left[ \begin{pmatrix}0\\1\end{pmatrix}, \begin{pmatrix} -1\\-m\end{pmatrix} \right],\quad
\text{cone}\left[ \begin{pmatrix}1\\0\end{pmatrix}, \begin{pmatrix} -1\\-m\end{pmatrix} \right], \quad
\text{cone}\left[ \begin{pmatrix}0\\1\end{pmatrix}, \begin{pmatrix} 1\\0\end{pmatrix} \right]
\]
which is the fan corresponding to the projective space
$\mathbf{P}^{2}_{[1:1:m]}$. The rays of this fan are the columns of
$P$.  The p-divisor is then a formal finite sum
\[
\sum \Delta_{\rho} \otimes D_{\rho}
\]
where $D_{\rho}$ is a divisor on $\mathbf{P}^2_{[1:1:m]}$ and $\Delta_{\rho}$ is a convex polytope with tail cone $\sigma$.  For the case at hand we have
\[
\begin{aligned}
\Delta_{(1,0)} &= (-bm/q,0) + \sigma\\
\Delta_{(-1,-m)} &= (am/p,0) + \sigma \\
\Delta_{(0,1)} &= \{0\} \times [0,1] + \sigma,
\end{aligned}
\]
where $\Delta_{\rho} = s(P^{-1}(\rho) \cap \mathbf{Q}^4_{\geq 0} )$.
This restricts to define a p-divisor on the curve $C:=\{X^m+Y^m+Z =0\}
\subset \mathbf{P}^{2}_{[1:1:m]}$, which is precisely the base of the
induced fibration of $Z_{BP}$.  This curve intersects the divisor
$Z=0$ at $m$ points, so the polytope $\Delta_{(0,1)}$ will appear $m$
times. Let $\sigma^{\vee}$ denote the dual of the tail cone as a
subset of $\mathfrak{t}^{\vee}$ the dual of the lie algebra.  In this
case $\sigma^{\vee}$ is described explicitly by 
\[
\sigma^{\vee} = \text{ cone } \left[ \begin{pmatrix} 0\\1 \end{pmatrix}, \begin{pmatrix}pq/m\\-1 \end{pmatrix} \right].
\]
For each p-divisor $\Delta_{\rho}$ we get a function $\Psi: \sigma^{\vee} \rightarrow \mathbf{R}$ defined by $\Psi_{\rho}(w) = \min_{u\in \Delta_{\rho}} \langle w,u\rangle$, where $\langle \cdot, \cdot \rangle$ denotes the natural pairing between $\mathfrak{t}$ and $\mathfrak{t}^{\vee}$.  In our case we get
\[
\begin{aligned}
\Psi_{(1,0)}(s,t) &= -\frac{bm}{q}s\\
\Psi_{(-1,-m)}(s,t) &= \frac{am}{p}s\\
\Psi_{(0,1)}(s,t) &= \min\{t,0\},
\end{aligned}
\]
where again $\Psi_{(0,1)}$ is repeated $m$ times. 
If $\frac{am}{p}$ is an integer then we obtain an equivalent p-divisor
by replacing $\Psi_{(-1,-m)}$ by zero, and replacing $\Psi_{(1,0)}$ by
$\left(\frac{am}{p}-\frac{bm}{q}\right)s$. This new p-divisor only has
two distinct polytopes, and so from the description of $T$-equivariant
test configurations with normal central fiber due to Ilten-Suss
\cite{IS} we see that there are at most two
non-trivial test configurations. The same applies when $\frac{bm}{q}$
is an integer. When neither is an integer, then both $\Psi_{(1,0)}$
and $\Psi_{(0,1)}$ have non-integral slope, and so again from
\cite[Proposition 4.2]{IS} we obtain at most
two non-trivial test configurations.  These can be obtained from the
methods of \cite{IS}, but here we can simply guess the test
configurations.  They are 
\[
\begin{aligned}
\mathcal{X}_{1} &= \{ uv+(t\cdot z)^p +w^q =0\},\\
\mathcal{X}_{2} &= \{uv+z^p +(t\cdot w)^q=0\}.
\end{aligned}
\]
We will compute the Futaki invariant for $\mathcal{X}_{1}$, the other case being identical.  The special fiber of $\mathcal{X}_{1}$ is
\[
Z_0:=\Spec \frac{\mathbf{C}[u,v,z,w]}{uv+w^q}
\]
which is polarized by the Reeb vector field $(pq,pq,2q,2p)$.  Since $Z_{BP}(p,q)$ is a hypersurface, it is straightforward to check (see, e.g. \cite{MSY}) that $(Z_{BP}(p,q),\xi)$ is normalized Fano if  
\[
\xi = \frac{3}{2(p+q)} (pq,pq,2q,2p).
\]
On the central fiber there is a new $\mathbf{C}^{*}$ action corresponding to the one parameter subgroup induced by $\eta = (0,0,1,0)$.  The index character for the Reeb field $\xi+s\eta$ can be computed directly.  Let  $\lambda = 3/2(p+q)$, then
\[
\begin{aligned}
F(\xi+s\eta, t) &= \frac{1 -e^{-2pq\lambda t}}{(1-e^{-pq\lambda t})^2(1-e^{-(2q\lambda-s) t})(1-e^{-2p\lambda t})}\\
&= \frac{1}{\lambda^2p^2q(2q\lambda-s)t^3} + \frac{2p\lambda+2q\lambda-s}{2\lambda^2p^2q(2q\lambda-s)t^2}+O(t^{-1})
\end{aligned}
\]
From this we read off
\[
a_0(\xi-s\eta) =  \frac{1}{2\lambda^2p^2q(2q\lambda-s)}, \quad \left(\frac{a_1}{a_0}\right)(\xi-s\eta) = \lambda(2p+2q)-s
\]
and so
\[
D_{\eta}a_0(\xi) = -\frac{1}{2\lambda^2p^2q(2q\lambda)^2}, \quad \frac{1}{2}D_{\eta}a_0(\xi) = 1.
\]
By definition, the Futaki invariant is
\[
\begin{aligned}
Fut(\mathcal{X}_{1},\xi) &= \frac{a_0(\xi)}{2}D_{\eta}\left(\frac{a_{1}}{a_0}\right)(\xi) +\frac{1}{2}D_{\eta}a_0(\xi)\\
&=\frac{a_0(\xi)}{2}\left(1- \frac{1}{2q\lambda}\right)\\
&=\frac{a_0(\xi)}{2}\left(\frac{2q-p}{3q}\right),
\end{aligned}
\]
which is positive if and only if $2q > p$. Similarly for the other
test-configuration the condition is $2p > q$, and so
we obtain that $(Z_{BP}(p,q),\xi)$ is K-stable, and hence admits a
conical Ricci-flat K\"ahler metric if and only if 
\[
2p>q,\qquad 2q>p,
\]
which is precisely the Lichnerowicz obstruction discovered by
Gauntlett-Martelli-Sparks-Yau \cite{GMSY}. In dimension 5 there is
standard machinery for computing the topology of links of isolated
hypersurface singularities, see \cite{BGbook} for a complete
description.  In particular, it is straightforward to compute that the
link of $Z(p,q)$ is topologically $\#(\gcd(p,q)-1)S^2\times S^3$.  In
particular, whenever $\gcd(p,q)=1$, and $2p>q, 2q>p$ we obtain a
Sasaki-Einstein metric on $S^5$.  Furthermore, as a function of $p,q$,
the (unnormalized) volume of $(Z_{BP}(p,q),\xi)$ is given by 
\[
a_0(\xi) = \frac{2(p+q)^3}{27p^2q^2},
\]
and hence infinitely many of these metrics are inequivalent. For
example, fix a positive integer $m$ and let $p>2$. Then the affine
varieties $Z_{BP}(pm,(p-1)m)$ are K-stable, and the link is
topologically $\#m(S^2\times S^3)$.  Furthermore, the volume is given
by 
\[
Vol(Z_{BP}(pm,(p-1)m), \xi) = \frac{2(2p-1)^3}{27mp^2(p-1)^2}
\]
which is a strictly decreasing function as $p\rightarrow \infty$.  By
taking a sequence of primes going to infinity we obtain the existence
of infinite families of inequivalent, non-toric Sasaki-Einstein
metrics on $\#m(S^2\times S^3)$ for any $m \geq 0$ (where $m=0$ means
$S^5$). Furthermore, we note that $Z_{BP}(2,3)$ is also K-stable,
confirming the result of Li-Sun \cite{LiSun} that the $A_2$
singularity admits a Ricci-flat cone metric.

\subsection{The Yau-Yu singularities of type $II$}
One can apply similar techniques to treat the Yau-Yu links of type II and III.  We mention these applications briefly.  Consider the family of hypersurface singularities described by $Z_{II}(p,q)$.
which admits a two torus generated by the $\mathbf{C}^{*}$ actions with weights $(0,pq,q,p-1)$ and $(1,-1,0,0)$.  An easy symmetry argument shows that the normalized Reeb field minimizing the volume is given by
\[
\xi = \frac{3}{2(q+p-1)}(qp,qp,2q,2(p-1)), \qquad Vol(Z_{II}(p,q), \xi) = \frac{2(p+q-1)^3}{27 pq^2(p-1)}.
\]
By similar techniques used for the Brieskorn-Pham links one can show that there are only two $T$-equivariant test configurations.  The first of these test configurations is
\[
\mathcal{X}_{1} := \{uv +z^p +z(t\cdot w)^q=0\}
\]
which is induced by the $\mathbf{C}^{*}$ action with weights $(0,0,0,1)$.  A straightforward computation using the index character yields
\[
Fut(\mathcal{X}_1, \xi) >0 \iff \frac{3(p-1)}{(q+p-1)} >1.
\]
The second test configuration is more interesting, given by 
\[
\mathcal{X}_{2} := \{uv +t^{pq}w^p + wz^q=0\}
\]
which is induced by the $\mathbf{C}^{*}$ action with weights
$(0,0,q,-1)$.  Computing the Futaki invariant yields
\[
Fut(\mathcal{X}_{2},\xi)>0 \iff 2qp+1>p^2+q
\]
Note, in particular, that this obstruction is {\em strictly stronger} than the Lichnerowicz.  For example, the affine variety $Z_{II}(6,3)$ is not obstructed by the Lichnerowicz bound, but is destabilized by the test configuration $\mathcal{X}_{2}$.  Topologically, the link of $Z_{II}(p,q)$ is $\#\gcd(p-1,q) (S^2\times S^3)$.  If $p \geq 2$ is a prime number then one can easily check that
\[
Z_{II}(m(p-1)+1, mp)
\]
is $K$-stable, and hence generates a Sasaki-Einstein metric on $\#m (S^2\times S^3)$ with volume
\[
 Vol(Z_{II}(m(p-1)+1, mp),\xi) = \frac{2(2p-1)^3}{27p^2(p-1)(m(p-1)+1)}.
 \]
Taking a sequence of primes going to $\infty$ yields a second infinite sequence of distinct Sasaki-Einstein metrics on $\#m (S^2\times S^3)$ for any $m\geq 1$.
\subsection{The Yau-Yu links of type $III$}
Finally, a similar analysis works for the Yau-Yu links of type III, given by
\[
Z_{III}(p,q) = \{ uv+ z^pw + zw^q=0\} \subset \mathbf{C}^{4}_{(u,v,z,w)}
\]
which have a 2 torus action generated by the $\mathbf{C}^{*}$ actions with weights $(0, (pq-1),(q-1),(p-1))$ and $(1,-1,0,0)$.  The critical Reeb field is then
\[
\xi = \frac{3}{2(p+q-2)} ( pq-1, pq-1, 2(q-1), 2(p-1)).
\]
There are two non-trivial $T$-equivariant test configurations generated by the $\mathbf{C}^*$ actions with weights $(0,0,-1,p)$ and $(0,0,q,-1)$.  Computing the Futaki invariants as above we find that the link of $Z_{III}(p,q)$ admits a Sasaki-Einstein metric if and only if
\[
\begin{aligned}
3(p-1)^2(q-1)&>(p+q-2)(pq-2p+1)\\
 3(q-1)^2(p-1)&>(p+q-2)(pq-2q+1).
 \end{aligned}
\]
If we let $m= \gcd(p-1,q-1)+1$, then using \cite[Chapter 9]{BGbook} one can check that the link of $Z_{III}(p,q)$ is topologically $\#m(S^2 \times S^3)$.  As before, we obtain a third infinite family of distinct Sasaki-Einstein metrics on $\#m(S^2 \times S^3)$ for any $m\in \mathbf{N}$ with $m \geq 2$.

\section{Further Discussion}\label{sec:discuss}

The results contained in this paper motivate the following picture, which is the Sasakian analog of the general picture described in \cite{SzICM}.  Fix a polarized affine variety $(X, \mathbf{T}, \xi)$ of dimension $n$, where $\xi\in\mathfrak{t}$ is normalized, and has minimal volume. 
We try to find a Ricci flat K\"ahler cone metric
compatible with $\xi$ by deforming along the method of continuity.  If $(X,\xi)$ is $K$-stable, then we succeed.
If not, then the method of continuity breaks at some time $T_1\leq 1$, and we get a test configuration with central fiber
a normal, polarized affine variety $(Y_1,\mathbf{T}_1,\xi)$, where $\mathbf{T}_1$ is the torus generated by $\mathbf{T}$ and the vector field $w_1$ giving the test-configuration. In particular  $\dim \mathbf{T}_1=\dim \mathbf{T}+1$, but it is possible that $Y_1\cong X$ if the torus $\mathbf{T}$ that we started with was not maximal. 

The test configuration is destabilizing, and so by the discussion in Section~\ref{sec:background}, if $w_1$ is normalized, then we have
\[
D_{w_{1}}a_0(Y_1,\xi) \leq 0,
\]
with strict inequality if $Y_1\cong X$. 
We can now repeat the volume minimization for $(Y_1,\mathbf{T}_1,\xi)$ to obtain a new Reeb field $\xi_1$.
We expect that it will be possible to restart the method of continuity with the data $(Y_1,\mathbf{T}_1,\xi_1)$.  Assuming that the results here carry over to the case of non-isolated singularities, we can repeat the above process to get
\[
X \rightarrow Y_1 \rightarrow \cdots \rightarrow Y_k := Y,
\]
where the final $(Y,\xi_k)$ is $K$-stable, since after finitely many steps we must reach a toric variety. Once the variety is toric, then it is automatically $K$-stable, after volume minimization, since there are no non-trivial toric test-configurations with normal central fiber. Note that it was previously shown by Futaki-Ono-Wang~\cite{FOW} that toric Fano cone singularities with an isolated singularity admit Ricci flat cone metrics, and we expect the same to hold when there are singularities away from the cone point.  It then follows that given any $(X, \mathbf{T}, \xi)$, it should be possible to deform $X$ to a $K$-stable affine variety $(Y,\mathbf{T}',\xi')$ by at most $n-1$ test-configurations.

It is natural to wonder whether this process can be made canonical and it seems reasonable to expect that the K-stable variety $(Y, \xi_k)$ is canonically associated to $(X,\xi)$. In view of the discussion in Donaldson-Sun~\cite[Section 3.3]{DS2} and the example of Hein-Naber mentioned there, we expect that the Ricci flat cone metric on $(Y,\xi_k)$ is the metric tangent cone at the vertex of any Ricci flat K\"ahler metric on a neighborhood of the vertex on $X$. 

One can also ask for a more algebraic description of each $Y_i$ in the sequence above, and for this at each step it would be necessary to distinguish one particular destabilizing  test-configuration.  Motivated by conformal field theory (see \cite{CXY}) the natural way to choose between any two destabilizing test configurations with central fibers $Y_1, Y_2$ is to compare their volumes, after volume minimization.  That is, we repeat the volume minimization on $Y_i$ and get new polarized affine varieties $(Y_1,\xi_1)$ and $(Y_2,\xi_2)$.  We choose $Y_1$ over $Y_2$ if
\[
{\rm Vol}(Y_1, \xi_1) > {\rm Vol}(Y_2, \xi_2),
\]
and vice versa, where the volume can be computed algebraically from the index character.  When equality occurs, it may be that either $Y_1 \cong Y_2$ or there is a test configuration taking $(Y_1,\xi_1)$ to $(Y_2,\xi_2)$ or vice versa. These statements are confirmed to some extent by example calculations, but so far there is still little evidence for them.

In a less speculative vein there are many interesting questions regarding the existence of Sasaki-Einstein metrics on various manifolds. From Cho-Futaki-Ono~\cite{CFO} we know that $m\#(S^2\times S^3)$ admits infinitely many irregular Sasaki-Einstein metrics for $m \geq 1$, and we have shown that $S^5$ admits infinitely many quasi-regular Sasaki-Einstein metrics. It is natural to ask therefore:
\begin{que}
Does there exist an irregular Sasaki-Einstein metric on $S^5$? More generally can we classify all Sasaki-Einstein metrics on $S^5$ with an isometric 2-torus action? 
\end{que}

The combinatorial description of T-varieties should help with this classification, as long as one develops a method for reading off the topology of the link from the p-divisor (see the work \cite{LLM} in this direction). A more thorough study should also lead to higher dimensional existence results, in particular on odd dimensional spheres. We expect the following.
\begin{conj}
There are infinitely many families of Sasaki-Einstein metrics on $S^{2n+1}$ for all $n$.
\end{conj}

The same question can be asked for exotic spheres which bound
parallelizable manifolds, and the existence of Sasaki-Einstein metrics
on these was conjectured by Boyer-Galicki-Koll\'ar~\cite{BGK}. This
conjecture was verified up to dimension 15 by
Boyer-Galicki-Koll\'ar-Thomas~\cite{BGKT}.

\subsection*{Acknowledgements}
The authors would like to thank Charles Boyer and S. T. Yau for their
interest in this work, and for helpful comments. In addition we would like
to thank Hendrik S\"u\ss{}  for helpful discussions on T-varieties.

\bibliographystyle{amsplain}
\bibliography{Sasaki}

\end{document}